\documentclass[a4paper,reqno]{amsart}
\usepackage[english]{babel}
\usepackage{amsmath,amsthm,amssymb,amsfonts}

\usepackage{enumitem}
\setlist[enumerate]{leftmargin=*}
\setlist[itemize]{leftmargin=*}

\usepackage{mathrsfs}

\usepackage{hyperref}

\numberwithin{equation}{section}

\newtheorem{theorem}{Theorem}[section]
\newtheorem{lemma}[theorem]{Lemma}
\newtheorem{corollary}[theorem]{Corollary}
\newtheorem{proposition}[theorem]{Proposition}
\theoremstyle{definition}
\newtheorem{definition}[theorem]{Definition}
\newtheorem{assumption}[theorem]{Assumption}
\newtheorem{remark}[theorem]{Remark}

\newcommand{\DOIlink}[1]{\href{https://doi.org/#1}{\texttt{doi:#1}}}
\newcommand{\arXivlink}[1]{\href{https://arxiv.org/abs/#1}{\texttt{arXiv:#1}}}

\newcommand\supp{\mathop{\mathrm{supp}}}
\newcommand\arccosh{\mathop{\mathrm{arccosh}}}

\newcommand{\RR}{\mathbb{R}}
\newcommand{\CC}{\mathbb{C}}
\newcommand{\ZZ}{\mathbb{Z}}
\newcommand{\NN}{\mathbb{N}}

\newcommand{\Rpos}{\RR_+}
\newcommand{\Npos}{\NN_+}

\newcommand{\dil}{\mathfrak{d}}

\newcommand{\loc}{\mathrm{loc}}
\newcommand{\sloc}{\mathrm{sloc}}
\newcommand{\rad}{\mathrm{rad}}

\newcommand{\Gf}{\mathcal{G}}

\newcommand{\ee}{\mathrm{e}}

\newcommand{\dd}{\mathrm{d}}

\newcommand{\id}{\mathrm{id}}

\newcommand{\RBd}{\mathfrak{R}}

\newcommand{\Rz}{\mathcal{R}}

\newcommand{\DfD}{\mathcal{\dot D}}
\newcommand{\DfE}{\mathcal{\ddot D}}

\newcommand{\DfTD}{\mathcal{\dot E}}
\newcommand{\DfTE}{\mathcal{\ddot	E}}

\newcommand{\lie}[1]{\mathfrak{#1}}

\newcommand{\BesK}{\mathrm{K}}

\newcommand{\dimV}{d_1}
\newcommand{\dimZ}{d_2}
\newcommand{\dimN}{d}
\newcommand{\dimWZX}{\dimZ}
\newcommand{\dimX}{n}
\newcommand{\dimY}{k}

\newcommand{\opLN}{\mathcal{L}}
\newcommand{\opLaBe}{\Delta_{\mathrm{LB}}}
\newcommand{\opSubLa}{\Delta_{\mathrm{sub}}}

\newcommand\LinBdd{\mathfrak{L}}

\newcommand{\cT}{\mathcal{T}}

\newcommand{\bx}{\boldsymbol{x}}
\newcommand{\by}{\boldsymbol{y}}

\newcommand{\tc}{\mathrel{\,:\,}}
\newcommand{\defeq}{\mathrel{:=}}
\newcommand{\eqdef}{\mathrel{=:}}

\newcommand{\chr}{\mathbf{1}}

\newcommand{\vfX}{X}
\newcommand{\vfXN}{X^\flat}

\newcommand{\thr}{\varsigma}

\newcommand{\fS}{\mathrm{S}}
\newcommand{\fT}{\mathrm{T}}

\newcommand{\Dist}{\mathscr{D}}

\newcommand{\krn}{\mathrm{k}}

\renewcommand{\Re}{\mathop{\mathfrak{Re}}}

\title[Riesz transforms on Damek--Ricci spaces]{Riesz transforms for the distinguished Laplacian on Damek--Ricci spaces and operator-valued multivariate spectral multipliers}
\author{Jie Liu}
\address[J. Liu]{School of Mathematics and Statistics \\ Northwestern Polytechnical University \\ Xi'an 710129 \\ China, and Dipartimento di Scienze Matematiche \\ Politecnico di Torino \\ Corso Duca degli Abruzzi 24 \\ 10129 Torino \\ Italy}
\email{jay2000@mail.nwpu.edu.cn}
\email{jie.liu@polito.it}

\author{Alessio Martini}
\address[A. Martini]{Dipartimento di Scienze Matematiche \\ Politecnico di Torino \\ Corso Duca degli Abruzzi 24 \\ 10129 Torino \\ Italy}
\email{alessio.martini@polito.it}

\keywords{Damek--Ricci space, Riesz transform, H-type group, operator-valued spectral multiplier, heat kernel, singular integral operator}

\subjclass[2020]{22E30, 42B20}
	
\thanks{The first-named author gratefully acknowledges the financial support of the China Scholarship Council (Grant No.\ 202406290140).}

\begin{document}

\begin{abstract}
Let $\Delta = \nabla^* \nabla$ be the distinguished Laplacian on a Damek--Ricci space. We prove the $L^{p}$-boundedness of the vector of first-order Riesz transforms $\nabla \Delta^{-1/2}$ in the full range $p\in(1,\infty)$.
The most demanding part of the proof is establishing the boundedness for $p \in (2,\infty)$; this is obtained as a consequence of an operator-valued spectral multiplier theorem for the joint functional calculus of a commuting system of self-adjoint operators, which we prove here and may be of independent interest.
\end{abstract}

\maketitle

\section{Introduction}

Let $N$ be a group of Heisenberg type (H-type in short), in the sense of Kaplan \cite{K80}. In other words, $N$ is a connected, simply connected $2$-step nilpotent Lie group, whose Lie algebra $\lie{n}$ is equipped with an inner product such that, if $\lie{v}$ is the orthogonal complement of the centre $\lie{z}$ of $\lie{n}$, and $\lie{z}^*$ is the dual of $\lie{z}$, then
\begin{equation}\label{eq:htype}
\left|\mu[x,\cdot]\right| = \left|\mu\right| \left|x\right| \quad\forall x \in \lie{v} \quad\forall \mu \in \lie{z}^*,
\end{equation}
where $\mu [x,\cdot]$ is thought of as an element of the dual $\lie{v}^*$ of $\lie{v}$
(see Section \ref{s:sectionpreliminaries} for more details). The direct sum decomposition $\lie{n} = \lie{v} \oplus \lie{z}$ is a stratification of $\lie{n}$, as $[\lie{v},\lie{v}] = \lie{z}$; we shall denote by $\dimV$ and $\dimZ$ the dimensions of the layers $\lie{v}$ and $\lie{z}$ of $\lie{n}$, and set $\dimN \defeq \dimV+\dimZ$. The group $N$ is equipped with a natural family $(\dil_a)_{a \in \Rpos}$ of automorphic dilations, where $\Rpos = (0,\infty)$, acting on the Lie algebra $\lie{n}$ by
\begin{equation}\label{eq:dil}
\dil_a(x,z) = (a^{1/2} x,a z) \qquad \forall (x,z) \in \lie{v} \oplus \lie{z} \quad \forall a \in \Rpos.
\end{equation}

Let $S = N \rtimes \Rpos$ be the one-dimensional semidirect extension of $N$ obtained by letting the multiplicative group $\Rpos$ act on $N$ by automorphic dilations. Equipped with the natural left-invariant Riemannian metric induced by the factors $\Rpos$ and $N$, the manifold $S$ is also known as a Damek--Ricci space \cite{D87,DR92,DR921}. Damek--Ricci spaces are a class of Riemannian manifolds with exponential volume growth that include, but are not limited to, the rank-one Riemannian symmetric spaces of the non-compact type \cite{ADY96,CDKR91,CDKR98}. Actually, in order to include real hyperbolic spaces, one must extend the definition above, by admitting that, beside H-type groups, $N$ may also be chosen among simply connected abelian Lie groups (namely, ``degenerate H-type groups'' where $\lie{v} = \{0\}$ and thus $N \cong \RR^d$); here, however, we shall mainly focus on the nondegenerate case, where $N$ is nonabelian.

As the Riemannian metric on $S$ is left-invariant, the corresponding Riemannian measure is a left Haar measure on the group $S$. However, the group $S$ is not unimodular, so any right Haar measure $\rho$ on $S$ differs from the Riemannian measure. In this paper, following \cite{CGHM94,GS96,HS03,H74,MV10,V07} among others, we shall equip $S$ with the left-invariant Riemannian metric and the right Haar measure $\rho$; in other words, we consider $S$ as a weighted Riemannian manifold, where the weight is the modular function on $S$. In particular, for all $p\in [1,\infty]$, we shall denote by $L^{p}(S)$ the $L^{p}$ space with respect to the right Haar measure $\rho$.

Let $\vfX_{0},\dots,\vfX_{\dimN}$ be an orthonormal basis of left invariant vector fields on $S$. The so-called \emph{distinguished Laplacian} $\Delta$ on $S$ is the second-order elliptic differential operator on $S$ defined by 
\begin{equation}\label{eq:distinguishedlaplacian_intro}
\Delta = -\sum_{j=0}^{\dimN} \vfX_{j}^{2},
\end{equation}
which is left-invariant and essentially self-adjoint on $C_{c}^{\infty}(S)\subseteq L^{2}(S)$. The distinguished Laplacian $\Delta$ can also be written in divergence form as $\nabla^* \nabla$, where 
\begin{equation}\label{eq:nabla}
	\nabla =(\vfX_{0},\dots,\vfX_{\dimN})
\end{equation}
is the Riemannian gradient and $\nabla^*$ is its adjoint with respect to the measure $\rho$. Notice that $\Delta$ differs from the Laplace--Beltrami operator on $S$, as the latter is self-adjoint with respect to the Riemannian measure on $S$.

In this paper, we study the boundedness of the first-order Riesz transforms associated with $\Delta$ on $S$, namely, the operators
\begin{equation}\label{eq:Riesz}
\Rz_{j} = \vfX_{j} \Delta^{-1/2},\qquad j=0,\dots,\dimN.
\end{equation}
Our main result reads as follows.

\begin{theorem}\label{thm:maintheorem}
For $j=0,\dots,\dimN$, the first-order Riesz transform $\Rz_{j}$ is of weak type $(1,1)$ and bounded on $L^p(S)$ for any $p\in(1,\infty)$.
\end{theorem}

As a corollary of the above result, we recover the $L^p$-boundedness for $p \in (1,\infty)$ of the second-order Riesz transforms $\Rz_j \Rz_k^* = -X_j \Delta^{-1} X_k$ for $j,k=0,\dots,\dimN$, which was proved in \cite{GS96} for a certain class of Damek--Ricci spaces $S$. We point out that the convolution kernels of such second-order transforms are integrable at infinity; however, this is not the case for the first-order transforms $\Rz_j$, which constitutes one of the challenges in their analysis.

The $L^{p}$-boundedness of Riesz transforms associated with elliptic and sub-elliptic operators on Lie groups and manifolds has been studied by several authors in the past decades (see, e.g., \cite{AL94,ACDH04,B89,CD03,L06,LM98} and references therein), and it is not possible for us to provide a complete literature review on the matter. We simply observe that many works in the literature are devoted to the case where the underlying metric measure space is doubling; this is not the case here, as the manifold $S$, equipped with the Riemannian metric and the measure $\rho$, has exponential volume growth, which creates several difficulties in the analysis of the Riesz transforms at infinity. One common technique to deal with the latter problem makes use of the fact that the Laplacian under consideration has a spectral gap on $L^2$; however, this is not applicable here, as the distinguished Laplacian $\Delta$ on $S$ has no spectral gap.

As already mentioned, the class of Damek--Ricci spaces $S = N \rtimes \Rpos$ is often extended to include the degenerate case where $N \cong \RR^\dimN$ is abelian, i.e., to the case of the $ax+b$ groups $\RR^\dimN \rtimes \Rpos$. In this degenerate case, the analogue of our main result was established in \cite{HS03} for $p\leq 2$ and in \cite{M23} for $p>2$ (see also \cite{GS99,Sj99} for previous partial results); this is why here we stick to nonabelian, $2$-step H-type groups $N$. While our approach follows by and large that of \cite{HS03,M23}, a number of additional nontrivial difficulties arise when $N$ is nonabelian.

For $p\in(1,2]$, the proof of Theorem \ref{thm:maintheorem} follows a by now fairly standard scheme, based on gradient heat kernel estimates for $\Delta$, combined with the Calder\'on--Zygmund singular integral theory adapted to the non-doubling structure of $S$; the latter theory was developed by Vallarino \cite{V07} within the general framework established by Hebisch and Steger in \cite{HS03}.

More precisely, the proof that $\Rz_{j}$ is of weak type (1,1) and bounded on $L^{p}(S)$ for $p \in (1,2]$ proceeds by showing that $\Rz_{j}$ is a singular integral operator whose kernel satisfies a suitable integral H\"ormander condition (see Section \ref{s:sectionpleq2} for details). In turn, this integral H\"ormander condition is established by subordinating the Riesz transforms to the heat semigroup $\ee^{-t\Delta}$, and exploiting the following weighted $L^1$ estimates for its convolution kernel $h_{t} = \krn_{\ee^{-t\Delta}}$ and the gradient thereof, which we prove here (see Section \ref{s:sectionheatkernel}) and may be of independent interest.

\begin{theorem}\label{thm:heatkernel_intro}
Let $|\bx|$ denote the Riemannian distance of $\bx \in S$ from the identity element. For all $\varepsilon \in [0,1)$ and all $t > 0$,
\begin{equation}\label{eq:heatkernel_intro}
\int_S |h_t(\bx)| \,\ee^{\varepsilon |\bx|^2/4t} \,\dd\rho(\bx) \lesssim_\varepsilon 1,
\qquad
\int_S |\nabla h_t(\bx)| \,\ee^{\varepsilon |\bx|^2/4t} \,\dd\rho(\bx) \lesssim_\varepsilon t^{-1/2}.
\end{equation}
\end{theorem}

Here and in what follows, for any two nonnegative quantities $A$ and $B$, we use the notation $A \lesssim B$ to indicate an inequality up to an unspecified multiplicative constant, and we write $A\simeq B$ to mean that $A\lesssim B$ and $B\lesssim A$; subscripted variants such as $\lesssim_{\varepsilon}$ and $\simeq_{\varepsilon}$ indicate that the constants may depend on a parameter $\varepsilon$.

A crucial feature of the estimates in Theorem \ref{thm:heatkernel_intro} is the fact that they are global in time and space (local estimates of a similar flavour are known to be true for more general elliptic Laplacians on manifolds). The unweighted case $\varepsilon = 0$ of the estimates \eqref{eq:heatkernel_intro} was already established in \cite{V07}, where it was used to prove a spectral multiplier theorem for the distinguished Laplacian $\Delta$; much as in \cite{V07}, a key tool in our proof are the estimates for the Laplace--Beltrami operator on Damek--Ricci spaces from \cite{ADY96}. On $ax+b$ groups, weighted estimates similar to \eqref{eq:heatkernel_intro} were proved in \cite{HS03}, though with weights of the form $\ee^{\kappa |\bx|/t^{1/2}}$ for $\kappa > 0$, instead of the Gaussian-type weights $\ee^{\varepsilon |\bx|^2/4t}$; those weaker estimates were anyway enough to obtain the $L^p$-boundedness for $p \leq 2$ of the Riesz transforms in that setting.

Proving that the Riesz transforms $\Rz_{j}$ are bounded on $L^{p}(S)$ for $p\in(2,\infty)$ is a more delicate problem, since the adjoint Riesz transforms $\Rz_j^*$ do not satisfy the aforementioned integral H\"ormander condition (this was established in \cite{M23,SV08} in the degenerate case $N = \RR^{\dimN}$), and thus the method used for $p \in (1,2)$ is not applicable via duality considerations for $p\in(2,\infty)$. 
Actually, the local parts of the adjoint Riesz transforms $\Rz_j^*$ (obtained by truncating the corresponding convolution kernels via compactly supported smooth cutoffs) are amenable to the aforementioned approach, i.e., their $L^p$-boundedness can also be derived via Calder\'on--Zygmund theory from appropriate localized weighted estimates of derivatives of the heat kernel $h_{t}$. So, as may be expected, the main difficulty in tackling the $\Rz_j^*$ comes from their parts at infinity.

To address the latter problem, much as in \cite{M23}, by means of subordination to the heat kernel, we obtain essentially explicit formulas and asymptotics for the convolution kernels $\krn_{\Rz_j^*}$ of the adjoint Riesz transforms (see Section \ref{s:sectionkernelasy}), thus writing them as sums of appropriate ``main terms'' $K_j$ (singular at infinity, but with simpler expressions) and remainders (integrable at infinity, or anyway with easily established $L^p$-boundedness properties). 
In this way, the $L^p$-boundedness of the $\Rz_j^*$ is reduced to that of the convolution operators by the kernels $K_j$.

To deal with those kernels, in the degenerate case $N=\RR^{\dimN}$ treated in \cite{M23}, one key observation was the fact that the right Haar measure on $S$ factorizes (in suitable coordinates) as the product of the Haar measures on $\RR^{\dimN}$ and $\Rpos$, and therefore the Lebesgue spaces $L^p(S)$ can be thought of as vector-valued Lebesgue spaces $L^p(\RR^{\dimN};L^p(\Rpos))$. In addition, left-invariant operators on $S = \RR^{\dimN} \rtimes \Rpos$ are also translation-invariant in the $\RR^{\dimN}$-coordinate; so they can be thought of as operator-valued Fourier multipliers on $\RR^{\dimN}$, whose $L^p$-boundedness can be established by means of an operator-valued Fourier multiplier theorem available in the literature \cite{SW07,W01}. The problem in \cite{M23} was thus reduced to checking that the the operator-valued symbols $M_j : \RR^{\dimN} \to \LinBdd(L^2(\Rpos))$, corresponding to the convolution operators with kernels $K_j$, satisfy the assumptions of the multiplier theorem of \cite{SW07}; here with $\LinBdd(Z)$ we denote the space of bounded linear operators on a Banach space $Z$.

In the nondegenerate case studied in the present paper, where $N$ is a (nonabelian) H-type group, it is still true that $L^p(S) = L^p(N;L^p(\Rpos))$ and that left-invariant operators on $S$ are also left-invariant in the $N$-variable. However, an operator-valued Fourier multiplier theorem on $N$, analogous to that in \cite{SW07} for $\RR^{\dimN}$, is not available in the literature (nor is it immmediately clear what it would look like, given the complicated nature of the group Fourier transform on $N$, which is already operator-valued even when applied to scalar-valued functions).

Our workaround to the latter problem consists in considering, instead of Fourier multipliers, operator-valued \emph{spectral} multipliers for the sub-Laplacian $\opLN$ and the vector of central derivatives $-i\nabla_z = (-i\partial_{z_1},\dots,-i\partial_{z_{\dimZ}})$ on the H-type group $N$.
More specifically, we reduce the $L^p$-boundedness of the convolution operators $f \mapsto f * K_j$ to that of three operators $M_0(\opLN,-i\nabla_z)$, $M_{\lie{v}}(\opLN,-i\nabla_z)$ and $M_{\lie{z}}(\opLN,-i\nabla_z)$, belonging to an operator-valued joint functional calculus for $\opLN$ and $-i\nabla_z$, whose operator-valued symbols $M_0,M_{\lie{v}},M_{\lie{z}} : \Rpos \times \RR^{\dimZ} \to \LinBdd(L^2(\Rpos))$ we compute explicitly. We then deduce their $L^p$-boundedness from the following result (see Sections \ref{s:sectionmultipliertheorem} and  \ref{ss:convKopvalmult} for more details).

\begin{theorem}\label{thm:opval_intro}
Let $\kappa \in (0,\dimV/2)$ and $\Omega_\kappa \defeq \{ (\lambda,\mu) \in \Rpos \times \RR^{\dimZ} \tc \kappa |\mu| \leq \lambda \}$. Fix an integer $B > (\dimV+5\dimZ)/2+3$, and let $p \in (1,\infty)$. If $M : \Rpos \times \RR^{\dimZ} \to \LinBdd(L^2(\Rpos))$ is (weakly) of class $C^B$ and the family of operators
\[
\left\{ \lambda^{|\alpha|} \partial_{(\lambda,\mu)}^{\alpha}M(\lambda,\mu) \tc \alpha \in \NN \times \NN^{\dimZ}, \ |\alpha| \leq B, \ (\lambda,\mu) \in \Omega_\kappa \right\}
\]
is R-bounded on $L^p(\Rpos)$, then $M(\opLN,-i\nabla_z)$ is bounded on $L^p(S)$.
\end{theorem}

The latter result is a particular instance of a more general operator-valued spectral multiplier theorem that we establish here and may be of independent interest. Namely, in Section \ref{s:sectionmultipliertheorem} we prove a conditional result, essentially stating that, if for a system of commuting self-adjoint operators one has a scalar-valued $L^p$ spectral multiplier theorem of Mihlin--H\"ormander type, then such result can be upgraded to an operator-valued multiplier theorem. In the case of the system $(\opLN,-i\nabla_z)$ on $L^2(N)$, the scalar-valued counterpart to Theorem \ref{thm:opval_intro} is a corollary of results of \cite{MRS96}.

A similar problem was encountered in \cite{MP24bis}, where the $L^p$-boundedness of the Riesz transforms associated with certain \emph{sub-Laplacians} on extensions $N \rtimes \Rpos$ of Carnot groups $N$ is proved; in that case, the problem was reduced to the $L^p$-boundedness of operator-valued spectral multipliers $M(\opLN)$ of a \emph{single} operator (the sub-Laplacian $\opLN$ on $N$), for which one could apply the multiplier theorem from \cite{MP24}.
The operator-valued multiplier theorem that we prove in Section \ref{s:sectionmultipliertheorem} can be thought of as a multivariate extension of that in \cite{MP24}.

To appreciate the difference between \cite{MP24bis} and our set-up, notice that, in the case where $N$ is an H-type group, the sub-Laplacian $\opSubLa$ on $N \rtimes \Rpos$ considered in \cite{MP24bis}
is obtained by restricting the sum in \eqref{eq:distinguishedlaplacian_intro} to the vector fields in the directions of the factor $\Rpos$ and the first layer of $N$; in other words, we can write
\[
\opSubLa = -\sum_{j=0}^{\dimV} \vfX_j^2 = -(a \partial_a)^2 + a \opLN,
\]
where $a$ is the coordinate on the factor $\Rpos$ (see \eqref{eq:leftinvariantN}, \eqref{eq:subLaplacianN} and \eqref{eq:leftinvariant} for details); thus, the Riesz transform relative to $\opSubLa$ in the direction of $\vfX_0 = a \partial_a$ can be written, at least formally, as
\[
\vfX_0 \opSubLa^{-1/2} = (a \partial_a) \left[-(a \partial_a)^2 + a \opLN\right]^{-1/2} = F(\opLN),
\]
where the operator-valued symbol $F : \Rpos \to \LinBdd(L^2(\Rpos))$ is given by
\[
F(\lambda) = (a \partial_a) \left[-(a \partial_a)^2 + \lambda a \right]^{-1/2}.
\]
However, in our case, we work with the elliptic Laplacian
\[
\Delta = \opSubLa - \sum_{j=1}^{\dimZ} \vfX_{\dimV+j}^2 = -(a \partial_a)^2 + a \opLN + a^2 \left|-i\nabla_z\right|^2
\]
on $N \rtimes \Rpos$; thus the Riesz transform relative to $\Delta$ in the direction of $\vfX_0$ can be written as
\[
\Rz_0 = \vfX_0 \Delta^{-1/2} = (a \partial_a) \left[-(a \partial_a)^2 + a \opLN + a^2 \left|-i\nabla_z\right|^2\right]^{-1/2} = G(\opLN,-i\nabla_z),
\]
where the multivariate symbol $G : \Rpos \times \RR^{\dimZ} \to \LinBdd(L^2(\Rpos))$ is given by
\[
G(\lambda,\mu) = (a \partial_a) \left[-(a \partial_a)^2 + \lambda a + |\mu|^2 a^2 \right]^{-1/2}.
\]
This roughly explains why here we need to consider joint functions of multiple operators, and therefore the multiplier theorem of \cite{MP24}, which applies to the functional calculus for a single operator, is not enough for our purposes.

As previously mentioned, here we do not apply the operator-valued multiplier theorem directly to the $\Rz_j$, but to the operators $M_\ell(\opLN,-i\nabla_z)$ with $\ell \in \{0,\lie{v},\lie{z}\}$, which are related to the main terms $K_j$ of the parts at infinity of the $\krn_{\Rz_j^*}$.
As a matter of fact, some of the operators $f \mapsto f * K_j$ are not in the functional calculus for $(\opLN,-i\nabla_z)$ themselves, due to lack of radiality in the first-layer variable on $N$,
and the reduction to the operators $M_\ell(\opLN,-i\nabla_z)$ is based on the known $L^p$-boundedness properties of first- and second-order Riesz transforms on $N$ (see Section \ref{s:sectionlocalpart} and Remark \ref{rem:notradial} for details).
Crucially, while the identification of the main terms $K_j$ in Section \ref{s:sectionkernelasy} is done on the space side (i.e., in terms of coordinates on $S$), in Section \ref{ss:gelfand} we manage, by exploiting some formulas from \cite{RT20}, to obtain relatively explicit expressions for the operator-valued symbols $M_\ell$ on the spectral side (i.e., in terms of coordinates $(\lambda,\mu)$ in the joint spectrum of $\opLN$ and $-i\nabla_z$). Thanks to those expressions, in Section \ref{ss:symbol_estimates} we are eventually able to check that the $M_\ell$ satisfy the assumptions of Theorem \ref{thm:opval_intro} and prove the claimed $L^p$-boundedness.

An interesting open problem remains about endpoint bounds at $p=\infty$ for the Riesz transforms, and specifically, whether the adjoint Riesz transforms $\Rz_j^*$ are of weak type $(1,1)$; indeed, the operator-valued multiplier approach does not seem to provide information on this. In the degenerate case $N = \RR^{\dimN}$, this weak type bound is known to be true for $j > 0$, i.e., for the adjoint Riesz transforms on $S = N \rtimes \Rpos$ in the direction of the factor $N$ \cite{GS99,M23}, but it is not known whether the same bound holds for $\Rz_0^*$. Similar open problems appear in the case of the Riesz transforms associated to the sub-Laplacians on extensions of Carnot groups considered in \cite{MP24bis,MV21}, as well as in the discrete setting of flow Laplacians on trees \cite{HS03,LMSTV23,MSTV25}.

\subsection*{Notation}

Here $\NN$ denotes the set of nonnegative integers (including 0), while $\Npos = \NN \setminus \{0\}$.
For $n \in \Npos$, we write $\dot\RR^n$ for $\RR^n \setminus \{0\}$.
The characteristic function of a set $S$ is denoted by $\chr_S$.
For a real number $a$, we write $a_+$ to denote its positive part $\max\{a,0\}$.
We denote by $\Re s$ the real part of a complex number $s$.

\section{Preliminaries}\label{s:sectionpreliminaries}

As in the introduction, let $N$ be a $2$-step stratified Lie group, i.e., a connected, simply connected, nilpotent Lie group, whose Lie algebra decomposes as a direct sum $\lie{n} = \lie{v} \oplus \lie{z}$, where $[\lie{v},\lie{v}] = \lie{z}$ and $[\lie{n},\lie{z}] = \{0\}$. We equip $\lie{n}$ with an inner product compatible with the stratification, i.e., such that the layers $\lie{v}$ and $\lie{z}$ are orthogonal, and denote by $|\cdot|$ the corresponding norm.
We assume that $\lie{n}$ is an H-type Lie algebra, which means that \eqref{eq:htype} holds, where $\lie{v}^*$ and $\lie{z}^*$ are equipped with the dual norms.
Via the exponential map, we identify the group $N$ with its Lie algebra $\lie{n}$ by 
\[
\begin{array}{rcl}
\lie{v}\oplus \lie{z}&\to&N\\
(x,z)&\mapsto&\exp(x+z).
\end{array}
\]	
In these ``exponential coordinates'', the group law on $N$ is given by
\[
(x,z) \cdot_N (x',z')= \left( x+x',z+z'+\frac{1}{2}[x,x']\right)
\]
for all $(x,z), (x',z') \in N$, with $(0,0)$ as the identity element, and moreover
\[
(x,z)^{-1} = (-x,-z).
\]
The group $N$ is unimodular, and Lebesgue measure $\dd x \, \dd z$ is a left and right Haar measure on $N$.

Let $\dimV$ and $\dimZ$ denote the dimensions of $\lie{v}$ and $\lie{z}$, and let $\dimN \defeq \dimV+\dimZ$ be the topological dimension of $N$. A basis $\vfXN_1,\dots,\vfXN_{\dimN}$ of left-invariant vector fields on $N$, obtained by extending an orthonormal basis $e_1,\dots,e_{\dimN}$ of $\lie{n} = \lie{v} \oplus \lie{z}$, is given by 
\begin{equation}\label{eq:leftinvariantN}
\begin{aligned}
\vfXN_{j} &=	\partial_{x_{j}}+\frac{1}{2} [x,e_{j}]\cdot\nabla _{z} &\qquad\text{for } j=1,\dots,\dimV,\\
\vfXN_{\dimV+j} &= \partial_{z_j} &\qquad\text{for } j=1,\dots,\dimZ.
\end{aligned}
\end{equation}
The corresponding homogeneous left-invariant sub-Laplacian on $N$ is defined by
\begin{equation}\label{eq:subLaplacianN}
\opLN \defeq -\sum_{j=1}^{\dimV} (\vfXN_{j})^{2}.
\end{equation}

Let $S$ be the one-dimensional semidirect extension of $N$ obtained by letting $\Rpos$ act on $N$ via the automorphic dilations $(\dil_a)_{a \in \Rpos}$ given by \eqref{eq:dil}.
Thus, $S$ is a Lie group with the product law
\begin{equation}\label{eq:productlaw}
\begin{split}
	(x,z,a) \cdot (x',z',a')
	&=\left( x+a^{1/2}x',z+az'+\frac{1}{2}a^{1/2}[x,x'],aa' \right) \\
	&=((x,z)\cdot_{N}\dil_{a}(x',z'),aa').
\end{split}
\end{equation}
for all $(x,z,a),(x',z',a')\in S$; here the identity element is $(0,0,1)$, while
\begin{equation}\label{eq:inverse}
	(x,z,a)^{-1}=(-a^{-1/2}x,-a^{-1}z,a^{-1}).
\end{equation} 
We equip $S$ with the right Haar measure $\rho$, given by 
\begin{equation}\label{eq:haar}
\dd\rho(x,z,a)= \dd x \,\dd z \,\frac{\dd a}{a}.
\end{equation}
If $Q \defeq (\dimV+2\dimZ)/2$ denotes the homogeneous dimension of $N$ associated with the dilations $(\dil_a)_{a \in \Rpos}$,
then the modular function of $S$ is given by
\begin{equation}\label{eq:modularfunction}
	\delta(x,z,a)=a^{-Q}
\end{equation}

We point out that, in many references on H-type groups (and more general stratified Lie groups), the automorphic dilations are parametrized as $(\dil_{t^2})_{t \in \Rpos}$, and it is the quantity $\dimV+2\dimZ$ that is called the homogeneous dimension of $N$. The definition of the homogeneous dimension $Q$ that we use here is consistent with that used in many works on Damek--Ricci spaces, such as \cite{ADY96,DR92,DR921,V07}.

Group product and inversion have a counterpart in convolution and $L^1$-isometric involution for functions on $S$, given by
\begin{equation}\label{eq:convinv}
f * g(\bx) = \int_S f(\bx \cdot \by^{-1}) \, g(\by) \,\dd\rho(\by), \qquad f^*(\bx) = \delta(\bx) \, \overline{f(\bx^{-1})};
\end{equation}
the above formulas make sense for sufficiently well-behaved functions on $S$, but can be extended to more general functions, and also to distributions on $S$. Analogous formulas give convolution and involution for functions on the group $N$, where the modular function is identically $1$.

Let $\Dist'(S)$ denote the space of distributions on $S$, i.e., the dual of the LF-space $C^\infty_c(S)$ of test functions. By the Schwartz kernel theorem, any left-invariant, continuous linear operator $T : C^\infty_c(S) \to \Dist'(S)$ can be represented as a convolution operator, i.e.,
\[
T \phi = \phi * \krn_T \qquad \forall \phi \in C^\infty_c(S),
\]
for a suitable convolution kernel $\krn_T \in \Dist'(S)$; moreover, the convolution operator with kernel $\krn_T^*$ is the formal adjoint of $T$. Analogous results are true for left-invariant operators $T : C^\infty_c(N) \to \Dist'(N)$, and we shall use the same notation $\krn_T$ for the corresponding convolution kernels.

The relation \eqref{eq:productlaw} between the product laws on $S$ and $N$ corresponds to the following relation between convolution $*$ on $S$ and convolution $*_N$ on $N$:
\begin{equation}\label{eq:convSN}
(f*g)^{[a]} = \int_{\Rpos} f^{[a']} *_N g^{[a/a']}_{(a')} \,\frac{\dd a'}{a'},
\end{equation}
where we use the notation
\begin{equation}\label{eq:sect_dil}
f^{[a]}(x,z) \defeq f(x,z,a), \qquad
\phi_{(\lambda)}(x,z) \defeq \lambda^{-Q} \phi(\dil_{1/\lambda}(x,z)).
\end{equation}

A basis of left-invariant vector fields on $S = N \rtimes \Rpos$ can be constructed by lifting the vector fields $\vfXN_1,\dots,\vfXN_{\dimN}$ on $N$ to left-invariant vector fields $\vfX_1,\dots,\vfX_{\dimN}$ on $S$, and adding a vector field $\vfX_0$ in the direction of $\Rpos$. Namely,
\begin{equation}\label{eq:leftinvariant}
\begin{aligned}
\vfX_0 &= a\partial_{a},\\
\vfX_{j} &= a^{1/2} \left( \partial_{x_{j}}+\frac{1}{2} [x,e_{j}]\cdot\nabla_{z} \right) &\quad\text{for } j=1,\dots,\dimV,\\
\vfX_{\dimV+j} &= a\partial_{z_{j}},\quad &\quad\text{for } j=1,\dots,\dimZ.
\end{aligned}
\end{equation}
The corresponding gradient $\nabla$ on $S$ takes the form \eqref{eq:nabla},
while the distinguished Laplacian $\Delta = \nabla^* \nabla$ on $S$ is given by
\eqref{eq:distinguishedlaplacian_intro}.

We equip $S$ with the Riemannian metric associated with $\nabla$, i.e., such that $\vfX_{0},\dots,\vfX_{\dimN}$ form an orthonormal frame. Then, for the Riemannian distance $|\bx|$ of a point $\bx =(x,z,a) \in S$ from the identity element, from \cite[eq.\ (2.18)]{ADY96} and \cite[eq.\ (3)]{V07}, we have the explicit formula 
\begin{equation}\label{eq:distance}
\begin{split}
\cosh^2 \frac{|\bx|}{2}
&=\left(\frac{a^{1/2}+a^{-1/2}}{2}+\frac{1}{8}a^{-1/2}|x|^2\right)^2+\frac{1}{4}a^{-1}|z|^2 \\
&=\frac{1}{4a} \left[ \left( 1 + a + \frac{|x|^2}{4} \right)^2 + |z|^2 \right].
\end{split}
\end{equation}
Notice that from \eqref{eq:distance} it follows that 
\begin{equation}\label{eq:est_d_a}
\cosh \frac{|\bx|}{2} \geq \frac{a^{1/2}+a^{-1/2}}{2} = \cosh \frac{\log a}{2}, \quad\text{i.e.,}\quad |\bx| \geq \left|\log a\right|,
\end{equation}
see also \cite[eq. (1.20)]{ADY96}.

For $b,c \in \RR$ and $r \in \Rpos$, we shall use the notation
\begin{equation*}
r^{[b,c]} \defeq 
\begin{cases}
r^{b} &\text{if } r \leq 1,\\
r^{c} &\text{if } r \geq 1.
\end{cases}
\end{equation*}
The next lemma, which may be compared to \cite[Lemma 3.3]{M23} and \cite[Lemma 2.6]{MP24}, provides integration formulas for nonnegative radial functions on $S$ against certain weights, which we will use repeatedly.
 
\begin{lemma}\label{lem:intradial}
Let $b,c,\gamma,\tilde\gamma \in [0,\infty)$ and $s \in \RR$, and set
\begin{equation}\label{eq:intwgt}
\varpi \defeq (b,c,s,\gamma,\tilde\gamma), \qquad 
w_\varpi(\bx) \defeq a^s \left|\log a\right|^{[\gamma,\tilde\gamma]} |x|^b \,|z|^c 
\end{equation}
for $\bx = (x,z,a) \in S$.
For all measurable functions $F : [0,\infty) \to [0,\infty)$,
\begin{align}
\label{eq:Rinta}
\int_S \delta^{1/2}(\bx) \, F(|\bx|) \,w_\varpi(\bx) \,\dd\rho(\bx)
&\simeq_\varpi \int_0^\infty F(r) \,\phi_\varpi(r) \,\dd r \\
\label{eq:Rintm}
\int_{\{a \leq \ee\}} \delta^{1/2}(\bx) \, F(|\bx|) \,w_\varpi(\bx) \,\dd\rho(\bx)
&\simeq_\varpi \int_0^\infty F(r) \,\phi^-_\varpi(r) \,\dd r \\
\label{eq:Rintp}
\int_{\{a \geq 1/\ee\}} \delta^{1/2}(\bx) \, F(|\bx|) \,w_\varpi(\bx) \,\dd\rho(\bx)
&\simeq_\varpi \int_0^\infty F(r) \,\phi^+_\varpi(r) \,\dd r \\
\label{eq:Rintz}
\int_{\{1/\ee \leq a \leq \ee\}} \delta^{1/2}(\bx) \, F(|\bx|) \,w_\varpi(\bx) \,\dd\rho(\bx)
&\simeq_\varpi \int_0^\infty F(r) \,\phi^0_\varpi(r) \,\dd r,
\end{align}
where
\begin{align}
\label{eq:RintaD}
\phi_\varpi(r) &\defeq \begin{cases}
r^{[\dimN+b+c+\gamma,\tilde\gamma]} \,\ee^{\left(\frac{Q}{2} + \frac{b}{4} + \frac{c}{2} + \left| \frac{b}{4} + \frac{c}{2} + s\right|\right) r} &\text{if } \frac{b}{4}+\frac{c}{2}+s \neq 0,\\
r^{[\dimN+b+c+\gamma,1+\tilde\gamma]} \,\ee^{\left(\frac{Q}{2} + \frac{b}{4} + \frac{c}{2} \right) r} &\text{if } \frac{b}{4}+\frac{c}{2}+s = 0,
\end{cases} \\
\label{eq:RintmD}
\phi^-_\varpi(r) &\defeq \begin{cases}
r^{[\dimN+b+c+\gamma,\tilde\gamma]} \,\ee^{\left(\frac{Q}{2} - s\right) r} &\text{if } \frac{b}{4}+\frac{c}{2}+s < 0,\\
r^{[\dimN+b+c+\gamma,1+\tilde\gamma]} \,\ee^{\left(\frac{Q}{2} - s\right) r} &\text{if } \frac{b}{4}+\frac{c}{2}+s = 0,\\
r^{[\dimN+b+c+\gamma,0]} \,\ee^{\left(\frac{Q}{2} + \frac{b}{4} + \frac{c}{2}\right) r} &\text{if } \frac{b}{4}+\frac{c}{2}+s > 0,
\end{cases} \\
\label{eq:RintpD}
\phi^+_\varpi(r) &\defeq \begin{cases}
r^{[\dimN+b+c+\gamma,\tilde\gamma]} \,\ee^{\left(\frac{Q}{2} + \frac{b}{2} + c + s \right) r} &\text{if } \frac{b}{4}+\frac{c}{2}+s > 0,\\
r^{[\dimN+b+c+\gamma,1+\tilde\gamma]} \,\ee^{\left(\frac{Q}{2} + \frac{b}{4} + \frac{c}{2}\right) r} &\text{if } \frac{b}{4}+\frac{c}{2}+s = 0,\\
r^{[\dimN+b+c+\gamma,0]} \,\ee^{\left(\frac{Q}{2} + \frac{b}{4} + \frac{c}{2}\right) r} &\text{if } \frac{b}{4}+\frac{c}{2}+s < 0,
\end{cases} \\
\label{eq:RintzD}
\phi^0_\varpi(r) &\defeq 
r^{[\dimN+b+c+\gamma,0]} \,\ee^{\left(\frac{Q}{2} + \frac{b}{4} + \frac{c}{2}\right) r}.
\end{align}
\end{lemma}
\begin{proof}
By \eqref{eq:inverse}, \eqref{eq:distance} and \eqref{eq:intwgt},
\[
|\bx^{-1}| = |\bx|, \qquad
w_{(b,c,s,\gamma,\tilde\gamma)}(\bx^{-1}) 
= w_{(b,c,-s-b/2-c,\gamma,\tilde\gamma)}(\bx);
\]
so, by the change of variables $\bx \mapsto \bx^{-1}$, we see that \eqref{eq:Rintm} is equivalent to \eqref{eq:Rintp}, and moreover \eqref{eq:Rinta} follows by summing \eqref{eq:Rintm} and \eqref{eq:Rintp}. Therefore, we only need to prove \eqref{eq:Rintm} and \eqref{eq:Rintz}.

Let $\omega : \RR \to [0,\infty)$. From \eqref{eq:distance}, by the change of variables $a = e^u$ and using polar coordinates in $x \in \RR^{\dimV}$ and $z \in \RR^{\dimZ}$, we see that
\[\begin{split}
&\int_S \delta^{1/2}(\bx) \, F(|\bx|) \, \omega(\log a) \, |x|^b \, |z|^c \,\dd\rho(\bx)\\
&\simeq \int_\RR \ee^{-Qu/2} \, \omega(u) \int_0^\infty \int_0^\infty  F\left(2 \arccosh \sqrt{\left(\cosh \frac{u}{2}+\frac{s^2}{8 \ee^{u/2}}\right)^2 + \frac{t^2}{4\ee^{u}}}\right) \\
&\quad\times s^{b+\dimV} \, t^{c+\dimZ} \,\frac{\dd s}{s} \, \frac{\dd t}{t} \,\dd u\\
&\simeq_{b,c} \int_\RR \ee^{\left(\frac{b}{4}+\frac{c}{2}\right) u} \, \omega(u) \int_0^\infty \int_0^\infty  F\left(2 \arccosh \sqrt{\left(\cosh \frac{u}{2}+s\right)^2 + t}\right) \\
&\quad\times s^{\frac{b+\dimV}{2}} \, t^{\frac{c+\dimZ}{2}} \,\frac{\dd s}{s} \, \frac{\dd t}{t} \,\dd u.
\end{split}\]
The change of variables $s= \cosh(y/2)-\cosh(u/2)$, followed by $t=\cosh^2(r/2)-\cosh^2(y/2)$, then gives
\[\begin{split}
&\int_S \delta^{1/2}(\bx) \, F(|\bx|) \, \omega(\log a) \, |x|^b \, |z|^c \,\dd\rho(\bx)\\
&\simeq_{b,c} \int_\RR \ee^{\left(\frac{b}{4}+\frac{c}{2}\right) u} \, \omega(u) 
\int_{|u|}^\infty \int_0^\infty F\left(2 \arccosh \sqrt{\cosh^2 \frac{y}{2} + t}\right) \, t^{\frac{c+\dimZ}{2}} \, \frac{\dd t}{t} \\
&\quad\times \left(\cosh \frac{y}{2} - \cosh\frac{u}{2}\right)^{\frac{b+\dimV}{2}-1} \sinh \frac{y}{2} \,\dd y \,\dd u\\
&\simeq_{b,c} \int_\RR \ee^{\left(\frac{b}{4}+\frac{c}{2}\right) u} \, \omega(u) 
\int_{|u|}^\infty \int_{y}^\infty F(r) \left(\cosh^2\frac{r}{2}-\cosh^2\frac{y}{2}\right)^{\frac{c+\dimZ}{2}-1} \, \sinh r \, \dd r \\
&\quad\times \left(\cosh \frac{y}{2} - \cosh\frac{u}{2}\right)^{\frac{b+\dimV}{2}-1} \sinh \frac{y}{2} \,\dd y \,\dd u \\
&\simeq_{b,c} \int_0^\infty F(r) \, D_{\omega,b,c}(r) \,\dd r,
\end{split}\]
where
\[\begin{split}
D_{\omega,b,c}(r) &\defeq \int_0^r \int_{-y}^y \ee^{\left(\frac{b}{4}+\frac{c}{2}\right) u} \, \omega(u) \left(\cosh^2\frac{r}{2}-\cosh^2\frac{y}{2}\right)^{\frac{c+\dimZ}{2}-1} \, \sinh r \\
&\quad\times \left(\cosh \frac{y}{2} - \cosh\frac{u}{2}\right)^{\frac{b+\dimV}{2}-1} \sinh \frac{y}{2} \,\dd u \,\dd y.
\end{split}\]
Notice now that, for $r \geq y \geq 0$, we have $\sinh r \simeq r^{[1,0]} \,\ee^r$ and
\begin{equation}\label{eq:chchest}
\begin{split}
\cosh^2\frac{r}{2}-\cosh^2\frac{y}{2} &= \frac{\cosh r - \cosh y}{2} \\
&= \sinh \frac{r+y}{2} \sinh \frac{r-y}{2} \simeq (r-y)^{[1,0]} \, r^{[1,0]} \,\ee^r,
\end{split}
\end{equation}
whence also, for $y \geq |u|$,
\[
\cosh \frac{y}{2} - \cosh \frac{u}{2} \simeq (y-|u|)^{[1,0]} \, y^{[1,0]} \, \ee^{y/2},
\]
and therefore
\[\begin{split}
D_{\omega,b,c}(r) &\simeq_{b,c} r^{[(c+\dimZ)/2,0]} \,\ee^{(c+\dimZ)r/2} 
\int_0^r (r-y)^{[(c+\dimZ)/2-1,0]} \,y^{[(b+\dimV)/2,0]} \,\ee^{(b+\dimV)y/4}  \\
&\qquad\times \int_{-y}^y (y-|u|)^{[(b+\dimV)/2-1,0]} \, \ee^{\left(\frac{b}{4}+\frac{c}{2}\right) u} \, \omega(u) \,\dd u \,\dd y.
\end{split}\]

Now, in the case $\omega(u) = \omega^{-}_{s,\gamma,\tilde\gamma}(u) \defeq \ee^{su} \, |u|^{[\gamma,\tilde\gamma]} \,\chr_{\{ u \leq 1\}}$, we get
\begin{equation}\label{eq:Dminus}
\begin{split}
&D_{\omega^{-}_{s,\gamma,\tilde\gamma},b,c}(r) \\
&\simeq_{b,c} r^{[(c+\dimZ)/2,0]} \,\ee^{(c+\dimZ)r/2} 
\int_0^r (r-y)^{[(c+\dimZ)/2-1,0]} \,y^{[(b+\dimV)/2,0]} \,\ee^{(b+\dimV)y/4}  \\
&\qquad\times \int_{-y}^{\min\{1,y\}} (y-|u|)^{[(b+\dimV)/2-1,0]} \, |u|^{[\gamma,\tilde\gamma]} \, \ee^{\left(\frac{b}{4}+\frac{c}{2}+s\right) u} \,\dd u \,\dd y \\
&\simeq_{b,c} r^{[(c+\dimZ)/2,0]} \,\ee^{(c+\dimZ)r/2} 
\int_0^r (r-y)^{[(c+\dimZ)/2-1,0]} \,y^{[(b+\dimV)/2,0]} \,\ee^{(b+\dimV)y/4}  \\
&\qquad\times \int_0^y (y-u)^{[(b+\dimV)/2-1,0]} \, u^{[\gamma,\tilde\gamma]} \, \ee^{-\left(\frac{b}{4}+\frac{c}{2}+s\right) u} \,\dd u \,\dd y;
\end{split}
\end{equation}
in the last step, we used the splitting $\int_{-y}^{\min\{1,y\}} = \int_{-y}^0 + \int_{0}^{\min\{1,y\}}$ and the fact that, for $0 \leq u \leq \min\{1,y\}$, one has $\ee^{\pm\left(\frac{b}{4}+\frac{c}{2}+s\right) u} \simeq_{b,c} 1$, so the second summand can effectively be absorbed into the first one via the change of variables $u \mapsto -u$.

It is not difficult to prove that, if $\alpha,\beta>-1$, $\nu \geq 0$, $\sigma \in \RR$, then, for all $x \geq 0$,
\begin{equation}\label{eq:estintpe}
\int_0^x (x-t)^{[\alpha,0]} \, t^{[\beta,\nu]} \, \ee^{\sigma t} \,\dd t
\simeq_{\alpha,\beta,\nu,\sigma} \begin{cases}
x^{[\alpha+\beta+1,\nu]} \, \ee^{\sigma x} &\text{if } \sigma > 0,\\
x^{[\alpha+\beta+1,\nu+1]} &\text{if } \sigma = 0,\\
x^{[\alpha+\beta+1,0]} &\text{if } \sigma < 0.
\end{cases}
\end{equation}
By applying \eqref{eq:estintpe} twice (first to the integral in $u$ and then to that in $y$), we can estimate the expression in \eqref{eq:Dminus} and readily obtain that $D_{\omega^{-}_{s,\gamma,\tilde\gamma},b,c}(r) \simeq_{\varpi} \phi_\varpi^-(r)$, where $\phi_\varpi^-(r)$ is as in \eqref{eq:RintmD}, thus proving \eqref{eq:Rintm}.

On the other hand, in the case $\omega(u) = \omega^{0}_{s,\gamma,\tilde\gamma}(u) \defeq \ee^{su} \, |u|^{[\gamma,\tilde\gamma]} \,\chr_{\{ -1 \leq u \leq 1\}}$, we get
\[\begin{split}
&D_{\omega^{0}_{s,\gamma,\tilde\gamma},b,c}(r) \\
&\simeq_{b,c} r^{[(c+\dimZ)/2,0]} \,\ee^{(c+\dimZ)r/2} 
\int_0^r (r-y)^{[(c+\dimZ)/2-1,0]} \,y^{[(b+\dimV)/2,0]} \,\ee^{(b+\dimV)y/4}  \\
&\qquad\times \int_{-\min\{1,y\}}^{\min\{1,y\}} (y-|u|)^{[(b+\dimV)/2-1,0]} \, |u|^{[\gamma,\tilde\gamma]} \, \ee^{\left(\frac{b}{4}+\frac{c}{2}+s\right) u} \,\dd u \,\dd y \\
&\simeq_{\varpi} r^{[(c+\dimZ)/2,0]} \,\ee^{(c+\dimZ)r/2} 
\int_0^r (r-y)^{[(c+\dimZ)/2-1,0]} \,y^{[b+\dimV+\gamma,0]} \,\ee^{(b+\dimV)y/4}  \,\dd y,
\end{split}\]
and again an application of \eqref{eq:estintpe} gives that $D_{\omega^{0}_{s,\gamma,\tilde\gamma},b,c}(r) \simeq_\varpi \phi_\varpi^0(r)$, where $\phi_\varpi^0(r)$ is as in \eqref{eq:RintzD}, thus proving \eqref{eq:Rintz}.
\end{proof}

Below we collect some consequences of Lemma \ref{lem:intradial} that we shall use later.

\begin{corollary}\label{cor:radial}
For every nonnegative function $F\in C_{c}^{\infty}(S)$,
\begin{multline}\label{eq:radial456}
	\int_{S} a^{-\frac{1}{2}} \left(\left(1+a+\frac{|x|^{2}}{4}\right)|x| + |x||z|\right)\,\delta^{1/2}(\bx)F(|\bx|) \,\dd \rho(\bx)\\
	\simeq \int_{0}^{\infty} F(r) \,r^{\left[\dimN+1,0\right] } \,\ee^{(Q+2)r/2} \,\dd r,
\end{multline}
\begin{multline}\label{eq:radial4561}
	\int_{\{a\leq\ee\}} a^{-\frac{1}{2}}\left( \left(a+\frac{|x|^{2}}{4}\right)|x|+|x||z|\right) \,\delta^{1/2}(\bx)F(|\bx|) \,\dd \rho(\bx) \\
	\simeq \int_{0}^{\infty} F(r) \,r^{\left[\dimN+1,0\right]} \,\ee^{(Q/2+3/4)r} \,\dd r,
\end{multline}
\begin{multline}\label{eq:radial32}
	\int_{\{\frac{1}{\ee}\leq a\leq\ee\}} |1-a^{-1}| \left(1+a+\frac{|x|^{2}}{4}\right)\,\delta^{1/2}(\bx) F(|\bx|) \,\dd \rho(\bx)\\
	\simeq 	\int_{0}^{\infty} F(r) \,r^{\left[\dimN+1,0\right]} \,\ee^{(Q+1)r/2} \,\dd r,
\end{multline}
\begin{multline}\label{eq:radial31}
\int_{S} |1-a^{-1}| \left(1+a+\frac{|x|^{2}}{4}\right)\,\delta^{1/2}(\bx)F(|\bx|) \,\dd \rho(\bx) \\
\simeq \int_{0}^{\infty} F(r) \,r^{\left[\dimN+1,0\right]} \,\ee^{(Q+2)r/2} \,\dd r.
\end{multline}
\end{corollary}
\begin{proof}
Estimate \eqref{eq:radial456} follows by summing the instances of \eqref{eq:Rinta} with $(b,c,s,\gamma,\tilde\gamma)$ equal to $(1,0,-1/2,0,0)$, $(1,0,1/2,0,0)$, $(3,0,-1/2,0,0)$ and $(1,1,-1/2,0,0)$. Instead, estimate \eqref{eq:radial4561} follows by summing the instances of \eqref{eq:Rintm} corresponding to $(b,c,s,\gamma,\tilde\gamma)$ equal to $(1,0,1/2,0,0)$, $(3,0,-1/2,0,0)$ and $(1,1,-1/2,0,0)$.

Finally, notice that 
\[
|1-a^{-1}| \simeq \begin{cases}
\left| \log a\right|^{[1,0]} &\text{if } a \geq 1/\ee\\
\left|\log a\right|^{[1,0]} a^{-1} &\text{if } a \leq \ee.
\end{cases}
\]
Thus, \eqref{eq:radial32} follows from summing the instances of \eqref{eq:Rintz} where $(b,c,s,\gamma,\tilde\gamma)$ is equal to $(0,0,0,1,0)$, $(0,0,1,1,0)$ and $(2,0,0,1,0)$. Instead, \eqref{eq:radial31} follows by summing the instances of \eqref{eq:Rintm} where $(b,c,s,\gamma,\tilde\gamma)$ is equal to $(0,0,-1,1,0)$, $(0,0,0,1,0)$ and $(2,0,-1,1,0)$, together with the instances of \eqref{eq:Rintp} where $(b,c,s,\gamma,\tilde\gamma)$ is equal to $(0,0,0,1,0)$, $(0,0,1,1,0)$ and $(2,0,0,1,0)$.
\end{proof}

\section{Heat kernel estimates}\label{s:sectionheatkernel}

Let $h_{t}\defeq \krn_{\ee^{-t\Delta}}$ denote the heat kernel associated to $\Delta$ on $S$. The main purpose of this section is proving the weighted $L^1$-estimates for $h_t$ and its gradient stated in Theorem \ref{thm:heatkernel_intro}.
In the unweighted case $\varepsilon = 0$, the first estimate in \eqref{eq:heatkernel_intro} is a consequence of the general fact that $\|h_t\|_1 = 1$ for all $t>0$, while the second estimate is established in \cite[Proposition 4.7]{V07}. 
We shall exploit several ideas from \cite{V07} to prove the estimates for arbitrary $\varepsilon \geq 0$.

The operator $\Delta$ has a special relationship with the Laplace--Beltrami operator $\opLaBe$ associated with the Riemannian structure of $S$, see \cite[\S IX.1]{VSC92}, whence one obtains that
\[
\delta^{-1/2} h_t = \ee^{Q^2 t/4}  \krn_{\ee^{-t\opLaBe}},
\]
see also \cite[Proposition 2.3]{V07}. From the formulas and estimates for $\krn_{\ee^{-t\opLaBe}}$ obtained in \cite{ADY96} we can therefore deduce analogous results for $\delta^{-1/2} h_t$, which we state here.

If a function $f : S \to \CC$ is radial, with a slight abuse of notation we shall also write $f(r)$ with $r \geq 0$ to denote the value of $f$ at any point $\bx \in S$ with $|\bx| = r$.

\begin{proposition}\label{prp:ADJ}
The function $\delta^{-1/2} h_t$ is smooth and radial on $S$. Moreover, when $\dimZ$ is even,
\begin{equation}\label{eq:formulahteven}
(\delta^{-1/2} h_t)(r) 
= 2^{-\dimV-\frac{\dimZ}{2}} \, \pi^{-\frac{\dimN}{2}} \, \DfD_r^{\dimZ/2} \DfE_r^{\dimV/2} h_t^\RR(r),
\end{equation}
where $h_t^{\RR}(r) \defeq (4\pi t)^{-1/2} \,\ee^{-r^2/4t}$ is the heat kernel on $\RR$, while
\begin{equation}\label{eq:DfDE}
\DfD_r \defeq -\frac{1}{\sinh r}\partial_r, \quad \DfE_r \defeq -\frac{1}{\sinh(r/2)}\partial_r;
\end{equation}
when $\dimZ$ is odd, instead,
\begin{equation}\label{eq:formulahtodd}
(\delta^{-1/2} h_t)(r) = 2^{-\dimV-\frac{\dimZ}{2}} \, \pi^{-\frac{\dimN+1}{2}} \int_{r}^{\infty} \DfD_s^{(\dimZ+1)/2} \DfE_s^{\dimV/2} h_t^\RR(s) \,\dd\nu_r(s),
\end{equation}
where 
\[
\dd \nu_r(s) \defeq \frac{\sinh s}{(\cosh s-\cosh r)^{1/2}} \,\dd s.
\]
In addition, for any $\dimZ$, we have the following pointwise estimates:
\begin{align}
\label{eq:qt}
(\delta^{-1/2} h_t)(r) &\simeq t^{-3/2} \, (1+r)\left(1+\frac{1+r}{t}\right)^{(\dimN-2)/2}\ee^{-Qr/2-r^2/4t}, \\
\label{eq:nablaqt}
|\nabla(\delta^{-1/2}h_t)|(r) &\simeq t^{-3/2} \, r\left(1+\frac{1+r}{t}\right)^{\dimN/2}\ee^{-Qr/2}\ee^{-r^{2}/4t}.
\end{align}
\end{proposition}
\begin{proof}
Formulas \eqref{eq:formulahteven} and \eqref{eq:formulahtodd} are consequences of \cite[eq.\ (5.7) and (5.8)]{ADY96}. The estimates \eqref{eq:qt} and \eqref{eq:nablaqt} are contained in \cite[Theorem 5.9 and Corollary 5.49]{ADY96}.
\end{proof}

The pointwise estimates \eqref{eq:qt}-\eqref{eq:nablaqt} are enough to establish the first estimate in \eqref{eq:heatkernel_intro} for any $t$, as well as the second estimate for small $t$; in addition, we can obtain similar estimates with a faster decay for large $t$ by restricting the integral to a bounded subset of $S$.

\begin{proposition}\label{prp:heateasyestimates}
Let $\varepsilon \in [0,1)$ and $t_0 > 0$. Then
\begin{equation}\label{eq:htintest0}
\int_S |h_t(\bx)| \,\ee^{\varepsilon |\bx|^2/4t} \,\dd\rho(\bx) \lesssim_{\varepsilon} 1 \qquad \forall t>0
\end{equation}
and
\begin{equation}\label{eq:nablahtintestsmallt}
\int_S |\nabla h_t(\bx)| \,\ee^{\varepsilon |\bx|^2/4t} \,\dd\rho(\bx) \lesssim_{\varepsilon,t_0} t^{-1/2} \qquad\forall t \in (0,t_0].
\end{equation}
In addition, for all compact subsets $K \subseteq S$,
\begin{equation}\label{eq:htintestlocal}
\int_K [|h_t(\bx)| + |\nabla h_t(\bx)|] \,\ee^{\varepsilon |\bx|^2/4t} \,\dd\rho(\bx) \lesssim_{\varepsilon,t_0,K} t^{-3/2} \qquad\forall t \in [t_0,\infty).
\end{equation}
\end{proposition}
\begin{proof}
By the integration formula \eqref{eq:Rinta} and the pointwise bound \eqref{eq:qt}, the integral in \eqref{eq:htintest0} is controlled by
\begin{equation*}
\begin{split}
&\int_S |h_t(\bx)| \,\ee^{\varepsilon |\bx|^2/4t} \,\dd\rho(\bx) \\
&\simeq \int_{0}^{\infty} (\delta^{-1/2} h_{t})(r) \,\ee^{\varepsilon r^2/4t} \,r^{[\dimN,1]} \, \ee^{Qr/2} \,\dd r\\
&\simeq\int_{0}^{\infty} t^{-3/2} \left(1+\frac{1+r}{t}\right)^{(\dimN-2)/2} \ee^{-\theta r^{2}/t} r^{[\dimN,2]} \,\dd r 
= \int_0^{r_0} + \int_{r_0}^\infty
\end{split}
\end{equation*}
for any $r_0 \in \Rpos$, where $\theta \defeq (1-\varepsilon)/4 > 0$.
In the integral for $r \geq r_0$, we use the estimate $1+\frac{1+r}{t} \lesssim_{r_0} 1+\frac{r^2}{t}$ to obtain that
\begin{equation*}
\begin{split}
\int_{r_0}^\infty 
&\lesssim_{r_0} \int_{r_0}^{\infty} t^{-1/2} \, \frac{r^{2}}{t} \left(1+\frac{r^2}{t}\right)^{(\dimN-2)/2} \,\ee^{-\theta r^{2}/t} \,\dd r\\
&\leq \int_{0}^{\infty} s^{2}(1+s^{2})^{(\dimN-2)/2} \,\ee^{-\theta s^{2}} \,\dd s \lesssim_\varepsilon 1,
\end{split}
\end{equation*}
where we used the change of variables $s = t^{-1/2} \, r$.
In the integral for $r \leq r_0$, instead, we have $1 + \frac{1+r}{t} \simeq_{r_0} t^{[-1,0]}$, so
\begin{equation}\label{eq:extradecayht}
\int_0^{r_0}
\lesssim_{r_0} \int_{0}^{r_0} t^{[-(\dimN+1)/2,-3/2]}\,\ee^{-\theta r^{2}/t} \,r^{\dimN} \,\dd r
\lesssim_{\varepsilon,r_0} t^{[0,-3/2]};
\end{equation} 
for the last estimate, when $t$ is large one can simply bound the exponential term by $1$, while when $t$ is small one can use the change of variables $s=t^{-1/2} \, r$ to obtain a uniform bound in $t$.

This completes the proof of \eqref{eq:htintest0}; in addition, the decay $t^{-3/2}$ obtained for $r \leq r_0$ and large $t$ in \eqref{eq:extradecayht} also proves the estimate for $h_t$ in \eqref{eq:htintestlocal}.

We are left with proving the estimates for the gradient $\nabla h_t$. Here we observe that, from \eqref{eq:modularfunction} and \eqref{eq:leftinvariant} it readily follows that
\begin{equation}\label{eq:leftinvariantmod}
\vfX_{0} \delta^{1/2} = -\frac{Q}{2} \delta^{1/2}(\bx),\quad  
\vfX_{j} \delta^{1/2} = 0 \quad\text{for } j=1,\dots,\dimN\\
\end{equation}
thus $|\nabla \delta^{1/2}| = \frac{Q}{2} \delta^{1/2}$ and, by the Leibniz rule,
\begin{equation}\label{eq:nablaht_leibniz}
|\nabla h_t| \lesssim \delta^{1/2} |\nabla(\delta^{-1/2} h_t)| + |h_t| .
\end{equation}
Therefore, when proving the estimates for $\nabla h_t$ in \eqref{eq:nablahtintestsmallt} and \eqref{eq:htintestlocal}, we can effectively replace $|\nabla h_t|$ with $\delta^{1/2} |\nabla(\delta^{-1/2} h_t)|$, as the estimates for the remaining part $|h_t|$ have already been proved.

Now, from \eqref{eq:Rinta} and the pointwise bound \eqref{eq:nablaqt} we obtain that
\[\begin{split}
&\int_S \delta^{1/2}(\bx) \, |\nabla (\delta^{-1/2} h_t)(\bx)| \,\ee^{\varepsilon |\bx|^2/4t} \,\dd\rho(\bx) \\
&\simeq \int_{0}^{\infty} |\nabla (\delta^{-1/2} h_{t})|(r) \,\ee^{\varepsilon r^2/4t} \,r^{\left[\dimN,1\right]} \, \ee^{Qr/2} \,\dd r\\
&\simeq \int_0^\infty t^{-3/2}\,\ee^{-\theta r^2/t} \left(1+\frac{1+r}{t}\right)^{\dimN/2}  \, r^{[\dimN+1,2]} \,\dd r
= \int_{0}^{r_0} + \int_{r_0}^\infty.
\end{split}\]
In the part where $r \leq r_0$, arguing much as in \eqref{eq:extradecayht}, we obtain, for all $t>0$,
\begin{equation}\label{eq:extradecaynablaht}
\int_0^{r_0}
\simeq_{r_0} \int_0^{r_0} t^{[-(\dimN+3)/2,-3/2]} \,\ee^{-\theta r^2/t} \,r^{\dimN+1} \,\dd r
\lesssim_{\varepsilon,r_0} t^{[-1/2,-3/2]}.
\end{equation}
In the part where $r \geq r_0$, instead, if we additionally assume $t \leq t_0$ for some $t_0 \in \Rpos$, then 
we can estimate $\ee^{-\theta r^2/t} \leq \ee^{-\theta r^2/2t_0} \ee^{-\theta r_0^2/2t}$ and obtain
\[
\int_{r_0}^\infty 
\lesssim_{r_0,t_0} t^{-(\dimN+1)/2} \ee^{-\theta r_0/2t} \int_{r_0}^\infty \ee^{-\theta r^2/2t_0} \, r^{(\dimN+2)/2} \,\dd r
\lesssim_{\varepsilon,r_0,t_0,M} t^M
\]
for any $M \in \RR$ and $t \in (0,t_0]$. Combining these two estimates with the previous ones for $|h_t|$ completes the proof of \eqref{eq:nablahtintestsmallt}. Moreover, the extra decay $t^{-3/2}$ for large $t$ in \eqref{eq:extradecaynablaht}, together with \eqref{eq:extradecayht}, also completes the proof of \eqref{eq:htintestlocal}.
\end{proof}

In order to complete the proof of Theorem \ref{thm:heatkernel_intro}, in light of the estimates in Proposition \ref{prp:heateasyestimates}, it only remains to prove the weighted $L^1$-estimate for the gradient $|\nabla h_t(\bx)|$ for large $t$ and $\bx$. The approach used in Proposition \ref{prp:heateasyestimates}, based on \eqref{eq:nablaht_leibniz}, fails in this case, as $\|h_t\|_{L^1}$ does not decay for large $t$. Thus, in order to obtain the desired $t^{-1/2}$ decay, we cannot simply rely on the pointwise estimates \eqref{eq:qt} and \eqref{eq:nablaqt} as a black box, and we must exploit, much as in \cite{V07}, some subtler cancellations between the two terms.

We start with recalling some technical estimates from \cite{ADY96}.

\begin{lemma}
With the notation of Proposition \ref{prp:ADJ}, for any $p,q \in \NN$, we can write
\begin{equation}\label{eq:der_htR}
\DfD_r^q \DfE_r^p h_t^\RR(r) = \frac{\alpha_{p,q} \, r \, h_t^\RR(r)}{t \, \cosh^{p+2q} \frac{r}{2}} + Z_t^{p,q}(r),
\end{equation}
where $\alpha_{p,q} \in \Rpos$, while the function $Z_t^{p,q} : \RR \to \RR$ satisfies, for any $c \in (1,\infty)$,
\begin{equation}\label{eq:est_Htpq}
|Z_t^{p,q}(r)|+ |\partial_r Z_t^{p,q}(r)| \lesssim_{p,q,c} t^{-1} h_{ct}^\RR(r) \,\ee^{-(p+2q)r/2} \quad\forall r,t \geq 1.
\end{equation}
\end{lemma}
\begin{proof}
From \cite[eqs.\ (5.23)-(5.24)]{ADY96} (see also \cite[Lemma 4.4(ii)]{V07}) we see that
\[
\DfD_r^q \DfE_r^p \ee^{-r^2/4t} = \ee^{-r^2/4t} \left[ \frac{\alpha_{p,q} \, r}{t \,  \cosh^{p+2q} \frac{r}{2}} + \sum_{j=1}^{p+q} a_{p,q,j}(r) \, t^{-j} \right],
\]
where $\alpha_{p,q} \in \Rpos$, while, for all $k \in \NN$ and $r \geq 0$,
\begin{equation}\label{eq:est_apqj}
|\partial_r^k a_{p,q,j}(r)| \lesssim_{k,p,q} \begin{cases}
\ee^{-(p+2q)r/2} &\text{if } j =1,\\
(1+r)^j\, \ee^{-(p+2q)r/2} &\text{if } j \geq 2.
\end{cases}
\end{equation}
The identity \eqref{eq:der_htR} thus holds if we set
\[
Z_t^{p,q}(r) = \sum_{j=1}^{p+q} a_{p,q,j}(r) \, t^{-j} \, h_t^\RR(r).
\]
Now, by \eqref{eq:est_apqj},
\[\begin{split}
&|a_{p,q,j}(r) \, t^{-j} \, h_t^\RR(r)| + |\partial_r [a_{p,q,j}(r) \, t^{-j} \, h_t^\RR(r)]|\\[.2em]
&\lesssim [|a_{p,q,j}(r)| \, (1+r/t) + |\partial_r a_{p,q,j}(r)|] \, t^{-j} \, h_t^\RR(r) \\
&\lesssim_{p,q} \begin{cases} 
t^{-1} \, (1+r/t) \, h_t^\RR(r) \, \ee^{-(p+2q)r/2} &\text{if } j = 1,\\[.2em]
t^{-j} \, (1+r)^j \, (1+r/t) \, h_t^\RR(r) \, \ee^{-(p+2q)r/2} &\text{if } j \geq 2.
\end{cases}
\end{split}\]
If $r,t \geq 1$, then $1+r/t \leq 1+r/t^{1/2}$ and, for $j \geq 2$,
\[
t^{-j} (1+r)^j \simeq t^{-j} r^j = t^{-j/2} (r/t^{1/2})^j \leq t^{-1} (1+r/t^{1/2})^j;
\]
thus, for any $j \geq 1$,
\begin{multline*}
|a_{p,q,j}(r) \, t^{-j} \, h_t^\RR(r)| + |\partial_r [a_{p,q,j}(r) \, t^{-j} \, h_t^\RR(r)]| \\
\lesssim t^{-1} \, (1+r/t^{1/2})^{1+j} \, h_t^\RR(r) \,\ee^{-(p+2q)r/2} \lesssim_c t^{-1} \, h_{ct}^{\RR}(r) \,\ee^{-(p+2q)r/2}
\end{multline*}
and the estimate \eqref{eq:est_Htpq} follows.
\end{proof}

From the previous lemma and the formulas \eqref{eq:formulahteven}-\eqref{eq:formulahtodd}, we now derive another convenient representation for $h_t$.

\begin{lemma}
With the notation of Proposition \ref{prp:ADJ}, we can write
\begin{equation}\label{eq:newrepht}
(\delta^{-1/2} h_t)(r) = \frac{H_t^0(r)}{\cosh^Q \frac{r}{2}} + H_t^{1}(r),
\end{equation}
where, for any $c \in (1,\infty)$, the functions $H_t^0$ and $H_t^1$ satisfy
\begin{align}
\label{eq:est_Ht0}
t^{-1/2} \, |H_t^0(r)| + |\partial_r H_t^0(r)| &\lesssim_c t^{-1} \, h_{ct}^\RR(r) &&\forall t,r \geq 1,\\
\label{eq:est_Ht1}
|H_t^1(r)| + |\partial_r H_t^1(r)| &\lesssim_c t^{-1} \, h_{ct}^\RR(r) \, \ee^{-Qr/2} &&\forall t,r \geq 1.
\end{align}
\end{lemma}
\begin{proof}
In the case $\dimZ$ is even, by combining \eqref{eq:formulahteven} and \eqref{eq:der_htR} for $p=\dimV/2$ and $q=\dimZ/2$, we obtain that \eqref{eq:newrepht} holds true with
\[
H_t^0(r) = \frac{\alpha_{\dimV/2,\dimZ/2}}{2^{\dimV+\frac{\dimZ}{2}} \pi^{\frac{\dimN}{2}}} \frac{r \, h_t^\RR(r)}{t}, \qquad H_t^1(r) = \frac{1}{2^{\dimV+\frac{\dimZ}{2}} \pi^{\frac{\dimN}{2}}} Z_t^{\dimV/2,\dimZ/2}(r);
\]
in this case, \eqref{eq:est_Ht1} follows from \eqref{eq:est_Htpq}, while \eqref{eq:est_Ht0} is immediate.

When $\dimZ$ is odd, instead, by combining \eqref{eq:formulahtodd} and \eqref{eq:der_htR} for $p=\dimV/2$ and $q=(\dimZ+1)/2$, we obtain that \eqref{eq:newrepht} holds true with
\begin{align}
\label{eq:Ht0}
H_t^0(r) &= \frac{\alpha_{\dimV/2,(\dimZ+1)/2} \, \cosh^Q \frac{r}{2}}{2^{\dimV+\frac{\dimZ}{2}} \, \pi^{\frac{\dimN+1}{2}}}  \int_{r}^{\infty} \frac{s \, h_t^\RR(s)}{t \, \cosh^{Q+1} \frac{s}{2}} \,\dd\nu_r(s),\\
\label{eq:Ht1}
H_t^1(r) &= \frac{1}{2^{\dimV+\frac{\dimZ}{2}} \, \pi^{\frac{\dimN+1}{2}}}
 \int_{r}^{\infty} Z_t^{\dimV/2,(\dimZ+1)/2}(s) \,\dd\nu_r(s).
\end{align}

Now, from \eqref{eq:chchest} we see that, for $s \geq r \geq 1$,
\[
\cosh s - \cosh r \simeq \ee^s \, (s-r)^{[1,0]}, \quad \frac{\sinh s}{(\cosh s - \cosh r)^{1/2}} \simeq \ee^{s/2} (s-r)^{[-1/2,0]},
\]
whence, by \eqref{eq:est_Htpq} and \eqref{eq:Ht1}, for all $r,t \geq 1$,
\[\begin{split}
|H_t^1(r)| 
&\lesssim \int_{r}^{\infty} |Z_t^{\dimV/2,(\dimZ+1)/2}(s)| \,\dd\nu_r(s) \\
&\lesssim_c \int_{r}^{\infty} t^{-1} h_{ct}^\RR(s) \,\ee^{-(Q+1)s/2} \, \ee^{s/2} \,(s-r)^{[-1/2,0]} \,\dd s \\
&\leq t^{-1} h_{ct}^\RR(r) \, \ee^{-Qr/2} \int_{r}^{\infty} \,\ee^{-Q(s-r)/2} \,(s-r)^{[-1/2,0]} \,\dd s
\lesssim t^{-1} h_{ct}^\RR(r) \, \ee^{-Qr/2} ,
\end{split}\]
where we used that $h_{ct}$ is decreasing, and the change of variables $u=s-r$ in the integral to obtain a uniform bound in $r$.

Recall now from \cite[eq.\ (2.12)]{M23} that, for any smooth function $g:\Rpos\to \CC$ with sufficiently fast decay, and any $\ell \in \NN$ and $r > 0$,
\begin{equation}\label{eq:byparts}
	\DfD_r^\ell \int_r^{\infty} g(s) \,\dd\nu_r(s) = 
	\int_r^{\infty} \DfD_s^\ell g(s) \,\dd\nu_r(s).
\end{equation}
In particular, by \eqref{eq:Ht1} and \eqref{eq:byparts},
\[\begin{split}
\partial_r H^1_t(r) &= (-\sinh r) \DfD_r H^1_t(r) \\
&= \frac{\sinh r}{2^{\dimV+\frac{\dimZ}{2}} \, \pi^{\frac{\dimN+1}{2}}} 
\int_{r}^{\infty} \partial_s Z_t^{\dimV/2,(\dimZ+1)/2}(s) \,\frac{\dd s}{(\cosh s - \cosh r)^{1/2}}
\end{split}\]
and, much as before, by \eqref{eq:est_Htpq},
\[\begin{split}
|\partial_r H^1_t(r)| 
&\lesssim \sinh r \int_{r}^{\infty} |\partial_s Z_t^{\dimV/2,(\dimZ+1)/2}(s)| \,\frac{\dd s}{(\cosh s - \cosh r)^{1/2}} \\
&\lesssim_c \sinh r \int_{r}^{\infty} t^{-1} h_{ct}^\RR(s) \,\ee^{-(Q+1)s/2} \,\ee^{-s/2} \,(s-r)^{[-1/2,0]} \,\dd s \\
&\lesssim t^{-1} h_{ct}^\RR(r) \,\ee^{-(Q+2)r/2} \sinh r \int_{r}^{\infty} \,\ee^{-(Q+2)(s-r)/2} \,(s-r)^{[-1/2,0]} \,\dd s \\
&\lesssim t^{-1} h_{ct}^\RR(r) \,\ee^{-Qr/2}.
\end{split}\]
Combining the above estimates for $H_t^1(r)$ and $\partial_r H_t^1(r)$ proves \eqref{eq:est_Ht1}.

It remains to prove \eqref{eq:est_Ht0}. Much as in \cite[proof of Lemma 4.6]{V07}, notice that
\[
\frac{1}{\cosh^{Q+1} \frac{s}{2}} \frac{\sinh s}{(\cosh s-\cosh r)^{1/2}} = \frac{1}{\cosh^{Q} \frac{s}{2}} \frac{\sqrt{2} \sinh \frac{s}{2}}{(\cosh^2 \frac{s}{2}-\cosh^2\frac{r}{2})^{1/2}}.
\]
Therefore, by the change of variables $v = \frac{\cosh (s/2)}{\cosh (r/2)}$, we can write
\[
\int_{r}^{\infty} \frac{s \, h_t^\RR(s)}{t \, \cosh^{Q+1} \frac{s}{2}} \,\dd\nu_r(s) 
= \frac{2^{3/2}}{t \, \cosh^Q \frac{r}{2}} \int_1^\infty \tilde h_t^{\RR}(2\arccosh(v \cosh(r/2))) \,\frac{\dd v}{v^Q (v^2-1)^{1/2}}
\]
where $\tilde h_t^\RR(s) \defeq s \, h_t^\RR(s)$, and, by \eqref{eq:Ht0},
\begin{equation}\label{eq:Ht0new}
H_t^0(r) = \frac{\alpha_{\dimV/2,(\dimZ+1)/2}}{2^{\dimV+\frac{\dimZ}{2}} \, \pi^{\frac{\dimN+1}{2}}} \frac{2^{3/2}}{t} \int_1^\infty \tilde h_t^{\RR}(2\arccosh(v \cosh(r/2))) \,\frac{\dd v}{v^Q (v^2-1)^{1/2}}.
\end{equation}
As $\tilde h^\RR_t(s) \lesssim_c t^{1/2} \, h_{ct}^\RR(s)$, we thus see that
\[
|H_t^0(r)| \lesssim t^{-1} \int_1^\infty \tilde h_{t}^{\RR}(2\arccosh(v \cosh(r/2))) \,\frac{\dd v}{v^Q (v^2-1)^{1/2}} \lesssim_c t^{-1/2} \, h_{ct}^\RR(r),
\]
where we also used the monotonicity of $h_{ct}^\RR$.

In addition, from the elementary estimate $|\partial_s \tilde h_t^\RR(s)| \lesssim_c h_{ct}^\RR(s)$, we also see that
\[\begin{split}
&|\partial_r [\tilde h_t^\RR(2\arccosh(v \cosh(r/2)))]|\\
&= |(\partial_s\tilde h_t^\RR)(2\arccosh(v \cosh(r/2)))| \frac{v \sinh \frac{r}{2}}{(v^2 \cosh^2 \frac{r}{2} - 1)^{1/2}}\\
&\simeq |(\partial_s\tilde h_t^\RR)(2\arccosh(v \cosh(r/2)))| \lesssim_c h_{ct}^\RR(2\arccosh(v \cosh(r/2))) \leq h_{ct}^\RR(r)
\end{split}
\]
for $r,v \geq 1$. Thus, by \eqref{eq:Ht0new},
\[
|\partial_r H_t^0(r)| \lesssim t^{-1} \int_1^\infty |\partial_r[\tilde h_t^{\RR}(2\arccosh(v \cosh(r/2)))]| \,\frac{\dd v}{v^Q (v^2-1)^{1/2}} \lesssim_c t^{-1} \, h_{ct}^\RR(r),
\]
and the proof of \eqref{eq:est_Ht0} is complete.
\end{proof}

When proving large time estimates for $\nabla h_t$, the first summand in \eqref{eq:newrepht} is the most problematic, due to the worse decay in $t$. To deal with that term, we need a more precise spatial estimate for its gradient, which is captured in the statement below and is implicit in the proof of \cite[Lemma 4.5]{V07}.

\begin{lemma}
For all $\bx = (x,z,a) \in S$,
\begin{equation}\label{eq:est_nabladeltach}
|\nabla[\delta^{1/2}(\bx) \cosh^{-Q}(|\bx|/2)]| \lesssim \delta^{1/2}(\bx) \, a^{1/4} \,\ee^{-\left(\frac{Q}{2}+\frac{1}{4}\right) |\bx|}.
\end{equation}
\end{lemma}
\begin{proof}
Recall from \eqref{eq:distance} that
\begin{equation}\label{eq:chdist}
\cosh(|\bx|/2)=2^{-1} a^{-1/2} \left[(a+1+|x|^2/4)^2+|z|^2\right]^{1/2}.
\end{equation}
Thus, by \eqref{eq:modularfunction},
\[
\delta^{1/2}(\bx) \cosh^{-Q}(|\bx|/2)=2^{Q} \left[(a+1+|x|^2/4)^2+|z|^2\right]^{-Q/2}.
\]
Consequently, by \eqref{eq:leftinvariant}, we see that
\begin{align*}
|\vfX_{0}[\delta^{1/2}(\bx)\cosh^{-Q}(|\bx|/2)]|
&\lesssim\frac{a\left(a+1+|x|^{2}/4\right)}{\left[(a+1+|x|^{2}/4)^{2}+|z|^{2}\right]^{Q/2+1}},\\
|\vfX_{j}[\delta^{1/2}(\bx)\cosh^{-Q}(|\bx|/2))]|
&\lesssim\frac{a^{1/2}\left(a+1+|x|^{2}/4+|z|\right)|x|}{\left[(a+1+|x|^{2}/4)^{2}+|z|^{2}\right]^{Q/2+1}},\\
|\vfX_{\ell+\dimV}[\delta^{1/2}(\bx)\cosh^{-Q}(|\bx|/2))]|
&\lesssim\frac{a|z|}{\left[(a+1+|x|^{2}/4)^{2}+|z|^{2}\right]^{Q/2+1}} 
\end{align*}
for $j=1,\dots,\dimV$ and $\ell=1,\dots,\dimZ$.
Combining the above estimates gives
\[\begin{split}
|\nabla[\delta^{1/2}(\bx) \cosh^{-Q}(|\bx|/2)]| &\lesssim \frac{a^{1/2}}{(a+1+|x|^{2}+|z|)^{Q+1/2}} \\
&\simeq a^{1/2} a^{-Q/2-1/4} \cosh^{-Q-1/2}(|\bx|/2) \\
&\simeq a^{1/4} \delta^{1/2}(\bx) \,\ee^{-\left(\frac{Q}{2}+\frac{1}{4}\right) |\bx|},
\end{split}\]
where we used again \eqref{eq:chdist} and \eqref{eq:modularfunction}.
\end{proof}

We can finally complete the proof of Theorem \ref{thm:heatkernel_intro}.

\begin{proof}[Proof of Theorem \ref{thm:heatkernel_intro}]
In light of Proposition \ref{prp:heateasyestimates}, it only remains to prove that, for any $\varepsilon \in [0,1)$,
\[
\int_{|\bx| \geq 1} |\nabla h_t(\bx)| \,\ee^{\varepsilon |\bx|^2/4t} \,\dd\rho(\bx) \lesssim_\varepsilon t^{-1/2} \qquad\forall t \geq 1.
\]
Now, by \eqref{eq:newrepht}, we can write
\[
h_t(\bx) = \delta^{1/2}(\bx) \frac{H_t^0(|\bx|)}{\cosh^Q \frac{|\bx|}{2}} + \delta^{1/2}(\bx) H_t^{1}(|\bx|);
\]
therefore, for any $c \in (1,\infty)$, by the Leibniz rule,
\begin{equation}\label{eq:newptwestnablaht}
\begin{split}
|\nabla h_t(\bx)| 
&\leq \delta^{1/2}(\bx) \cosh^{-Q}(|\bx|/2) \, |(\partial_r H_t^0)(|\bx|)| \\
&\quad+ |\nabla[\delta^{1/2}(\bx) \cosh^{-Q}(|\bx|/2)]| \, |H_t^0(|\bx|)| \\
&\quad+ \delta^{1/2}(\bx) \,|(\partial_r H_t^{1})(|\bx|)| + |\nabla \delta^{1/2}(\bx)| \, |H_t^{1}(|\bx|)|\\
&\lesssim_c \delta^{1/2}(\bx) \, h_{ct}^\RR(|\bx|) \left[ t^{-1} \, \ee^{-Q|\bx|/2} + t^{-1/2} \, a^{1/4} \,\ee^{-(Q/2+1/4)|\bx|} \right] 
\end{split}
\end{equation}
for all $r,t \geq 1$,
where we used the estimates \eqref{eq:leftinvariantmod}, \eqref{eq:est_Ht0}, \eqref{eq:est_Ht1} and \eqref{eq:est_nabladeltach}, and the fact that $|\nabla[F(|\bx|)]| = |(\partial_r F)(|\bx|)|$ for a radial function $F(|\bx|)$.

Choose $c \in (1,1/\varepsilon)$, so that $\theta \defeq (1/c-\varepsilon)/4>0$.
From the pointwise estimate \eqref{eq:newptwestnablaht} and the integration formula
\eqref{eq:Rinta}
we finally deduce that, for $t \geq 1$,
\[\begin{split}
&\int_{|\bx| \geq 1} |\nabla h_t(\bx)| \,\ee^{\varepsilon |\bx|^2/4t} \, \dd\rho(\bx)\\
&\lesssim_\varepsilon t^{-3/2} \int_{|\bx| \geq 1} \delta^{1/2}(\bx) \, \ee^{-\theta |\bx|^2/t} \, \ee^{-Q|\bx|/2} \,\dd\rho(\bx) \\
&\quad+t^{-1} \int_{|\bx| \geq 1} \delta^{1/2}(\bx) \, \ee^{-\theta |\bx|^2/t} \, a^{1/4} \,\ee^{-(Q/2+1/4)|\bx|} \,\dd\rho(\bx) \\
&\simeq t^{-3/2} \int_1^\infty \ee^{-\theta r^2/t} \, r \,\dd r + t^{-1} \int_1^\infty \ee^{-\theta r^2/t} \,\dd r \lesssim_\varepsilon t^{-1/2},
\end{split}\]
as required.
\end{proof}

From Theorem \ref{thm:heatkernel_intro} and Proposition \ref{prp:heateasyestimates} we can also derive some estimates for second-order derivatives of the heat kernel, which we record here.
For simplicity we only present the unweighted version of the estimates, though a similar argument would also yield weighted variants.

\begin{proposition}\label{prp:higherheatestimates}
For all $t > 0$ and $j,k=0,\dots,\dimN$,
\begin{equation}\label{eq:XjXkstarht}
\int_S |\vfX_j(\vfX_k h_t)^*(\bx)| \,\dd\rho(\bx) \lesssim t^{-1};
\end{equation}
moreover, for all compact subsets $K \subseteq S$,
\begin{equation}\label{eq:XjXkhtloc}
\int_K |\vfX_j \vfX_k h_t(\bx)| \,\dd\rho(\bx) \lesssim t^{-1}.
\end{equation}
\end{proposition}
\begin{proof}
Notice that $h_t = h_{t/2} * h_{t/2}$ and $h_t^* = h_t$, so
\[
\vfX_j(\vfX_k h_t)^* = \vfX_j(h_{t/2} * \vfX_k h_{t/2})^* = (\vfX_k h_{t/2})^* * (\vfX_j h_{t/2})
\]
and therefore, by Young's inequality,
\[
\|\vfX_j(\vfX_k h_t)^*\|_1 \leq \|(\vfX_k h_{t/2})^*\|_1 \|(\vfX_j h_{t/2})\|_1 = \|\vfX_k h_{t/2}\|_1 \|\vfX_j h_{t/2}\|_1 \lesssim t^{-1},
\]
where the last estimate follows from Theorem \ref{thm:heatkernel_intro}.

Moreover, for any smooth function $f$ on $S$, from \eqref{eq:convinv} one sees that
\[
(\vfX_j f)^* = \vfX_j^\bullet f^*
\]
for certain right-invariant first-order differential operators $\vfX_j^\bullet$ on $S$; as the latter are differential operators with smooth (thus, locally bounded) coefficients, we have
\[
|(\vfX_j f)^*| \lesssim_K \sum_{\ell=0}^{\dimN} |\vfX_\ell f^*| + |f^*| \qquad\text{on } K
\]
pointwise for any compact subset $K \subseteq S$. As a consequence,
\[
|(\vfX_j \vfX_k h_t)^*| \lesssim_K \sum_{\ell=0}^{\dimN} |\vfX_\ell (\vfX_k h_t)^*| + |(\vfX_k h_t)^*| \qquad\text{on } K.
\]
Therefore, if $K^{-1} \defeq \{ \bx^{-1} \tc \bx \in K\}$, then
\[\begin{split}
\int_K |\vfX_j \vfX_k h_t(\bx)| \,\dd\rho(\bx) 
&=\int_{K^{-1}} |(\vfX_j \vfX_k h_t)^*(\bx)| \,\dd\rho(\bx) \\
&\lesssim_K \sum_{\ell=0}^{\dimN} \|\vfX_\ell (\vfX_k h_t)^*\|_1 + \int_{K} |\vfX_k h_t(\bx)| \,\dd\rho(\bx) \lesssim_K t^{-1},
\end{split}\]
where we applied the estimates \eqref{eq:XjXkstarht}, \eqref{eq:nablahtintestsmallt} and \eqref{eq:htintestlocal}.
\end{proof}

\section{\texorpdfstring{$L^{p}$}{Lp}-boundedness of the Riesz transforms for \texorpdfstring{$1<p\leq2$}{1<p<=2}.}\label{s:sectionpleq2}

From the heat kernel estimates proved in Section \ref{s:sectionheatkernel}, we can now deduce the $L^p$-boundedness for $p \in (1,2]$ of the Riesz transforms $\Rz_j$ given by \eqref{eq:Riesz}.

The proof is based on the following criterion to deduce $L^p$-boundedness of singular integral operators on Damek--Ricci spaces. This result was established in \cite{V07}, as a consequence of the fact that any Damek--Ricci space $S$ with the right Haar measure $\rho$ and the Riemannian distance $d$ is a Calder\'on--Zygmund space in the sense of \cite{HS03}. The statement below is actually a particular case of \cite[Theorem 4.1]{V07}, specialized to convolution operators (cf.\ \cite[Theorem 2.3 and Remark 2.4]{MV21}).

\begin{proposition}\label{prp:propcriteria}
Let $T=\sum_{n\in\ZZ}T_{n}$ be a linear operator bounded on $L^{2}(S)$, where:
\begin{enumerate}[label=(\roman*)]
\item the series converges in the strong topology of $L^{2}(S)$;
\item every $T_{n}$ is a convolution operator with kernel $k_{n}\in L^{1}(S)$;
\item there exists positive constants $a, A, \varepsilon$ and $c \neq 1$ such that
\begin{align}
\label{eq:criteria1}
\int_{S} |k_n(\bx)| \, (1+c^n|\bx|)^\varepsilon \,\dd\rho(\bx) &\leq A \quad \forall y\in S,\\
\label{eq:criteria3}
\int_{S} |\nabla k_{n}^{*}(\bx)| \,\dd\rho(\bx) &\leq Bc^{n}.
\end{align}
\end{enumerate}
Then $T$ extends from $L^{1}(S)\cap L^{2}(S)$ to an operator of weak type (1,1) and to an operator bounded on $L^{p}(S)$ for $1<p\leq 2$.
\end{proposition} 

We can now prove the $L^p$-boundedness for $p \in (1,2]$ of the Riesz transforms. The proof follows closely that of \cite[Theorem 1.1(i)]{MV21}.
  
\begin{theorem}\label{thm:Theoremboundedleq2}
The first-order Riesz transforms $\Rz_j = \vfX_j \Delta^{-1/2}$ are of weak type (1,1) and $L^{p}$-bounded for $1<p\leq2$ and $j=0,\dots,\dimN$.
\end{theorem}
\begin{proof}
The $L^2$-boundedness of $\Rz_j = \vfX_j \Delta^{-1/2}$ is an immediate consequence of the fact that, for any $f\in C_{c}^{\infty}(S)$,
\[
\|\vfX_j f \|_{L^{2}(S)} \lesssim \left\||\nabla f| \right\|_{L^{2}(S)} = \left\||\Delta^{1/2} f| \right\|_{L^{2}(S)}
\]
(cf., e.g., \cite[Proposition 4.1]{MV21}). To prove the remaining boundedness properties of $\Rz_j$, we shall apply Proposition \ref{prp:propcriteria}. For this, we decompose
\begin{equation}\label{eq:Rzsubheatdec}
\sqrt{\pi} \Rz_j = \sqrt{\pi} \vfX_j \Delta^{-1/2} 
= \sum_{n\in\ZZ}\int_{2^n}^{2^{n+1}} t^{-1/2} \vfX_j e^{-t\Delta} \,\dd t
\eqdef \sum_{n\in\ZZ} T_n,
\end{equation}
where the series converges in the strong operator topology on $L^{2}$.
Each $T_{n}$ is a convolution operator 
with kernel 
\[
   	k_{n}=\int_{2^{n}}^{2^{n+1}}t^{-1/2} \, \vfX_j h_{t} \,\dd t.
\]
Thus, to establish the weak type (1,1) and $L^{p}$-boundedness of this operator for $1<p\leq2$, it suffices to verify the kernel estimates \eqref{eq:criteria1} and \eqref{eq:criteria3}, as stated in Proposition \ref{prp:propcriteria}; in what follows, we verify these estimates with $c=2^{-1/2}$.
    
First, from Theorem \ref{thm:heatkernel_intro} we deduce that, for any $\varepsilon >0$,
\begin{equation}\label{eq:inequalitycriteria1}
\begin{split}
&\int_{S} |k_n(\bx)| \, (1+2^{-n/2}|\bx|)^\varepsilon \,\dd \rho(\bx) \\
&\lesssim_\varepsilon \int_{2^{n}}^{2^{n+1}} t^{-1/2} \int_{S} |\vfX_j h_{t}(\bx)| \,(1+t^{-1/2}|\bx|)^\varepsilon \,\dd\rho(\bx) \,\dd t\\
&\lesssim_\varepsilon \int_{2^{n}}^{2^{n+1}} t^{-1/2} \int_{S} |\nabla h_{t}(\bx)| \,\ee^{|\bx|^2/8t} \,\dd\rho(\bx) \,\dd t
\lesssim \int_{2^{n}}^{2^{n+1}} t^{-1}\,\dd t \lesssim 1.
\end{split}
\end{equation}
Next, from Proposition \ref{prp:higherheatestimates}, 
\begin{multline*}
\left\||\nabla k_{n}^{*}| \right\|_{L^{1}(S)}
\leq \sum_{\ell=0}^{\dimN} \|\vfX_\ell k_n^* \|_{L^1(S)}
\leq \sum_{\ell=0}^{\dimN} \int_{2^n}^{2^{n+1}} t^{-1/2} |\vfX_\ell (\vfX_j h_t)^*(\bx)| \,\dd\rho(\bx) \,\dd t
\\
\lesssim \int_{2^n}^{2^{n+1}} t^{-3/2} \,\dd t \simeq 2^{-n/2}.
\end{multline*}
This verifies conditions \eqref{eq:criteria1} and \eqref{eq:criteria3}, as required.
\end{proof}   

We cannot use the same approach to prove the $L^p$-boundedness of the $\Rz_j$ for $p>2$, i.e., that of the adjoint Riesz transforms $\Rz_j^*$ for $p <2$, due to the lack of global, large time estimates for second derivatives $\vfX_j \vfX_k h_t$ of the heat kernel analogous to \eqref{eq:XjXkstarht}. Thanks to the localized estimates \eqref{eq:XjXkhtloc}, however, we can prove the $L^p$-boundedness of the ``local parts'' of the adjoint Riesz transforms $\Rz_j^*$.
   
\begin{theorem}\label{thm:theoremlocalpart}
Let $\zeta \in C_{c}^{\infty}(S)$. Then, for $j=0,\dots,\dimN$, the convolution operators with kernels $\zeta \krn_{\Rz_{j}}$ and $\zeta \krn_{\Rz_{j}^{*}}$ are of weak type (1,1) and bounded on $L^{p}(S)$ for all $p\in (1,\infty)$.
\end{theorem}
\begin{proof}
Notice first that $(\zeta \krn_{\Rz_j})^* = \check\zeta \krn_{\Rz_j^*}$, where $\check\zeta(\bx) \defeq \overline{\zeta(\bx^{-1})}$ is also in $C^\infty_c(S)$. So, by duality, it is enough to establish $L^p$-boundedness for $p \leq 2$.

To prove this theorem, we apply Proposition \ref{prp:propcriteria}. Note that $\Rz_{j}$ is $L^{2}$-bounded, and so is its adjoint $\Rz_{j}^{*}$. Thus, by \cite[Lemma 4.1]{M23}, the convolution operators with kernels $\zeta \krn_{\Rz_{j}}$ and $\zeta \krn_{\Rz_{j}^{*}}$ are $L^2$-bounded too. We shall now apply Proposition \ref{prp:propcriteria} to obtain the remaining boundedness properties. We only present the argument for $\zeta \krn_{\Rz_{j}^{*}}$, as that for $\zeta \krn_{\Rz_{j}}$ is entirely analogous.

From the subordination formula \eqref{eq:Rzsubheatdec}, we deduce that
\[
\sqrt{\pi} \, \krn_{\Rz_{j}^{*}}= \sum_{n\in\ZZ}\int_{2^{n}}^{2^{n+1}}t^{-1/2} (\vfX_{j}h_{t})^* \,\dd t \eqdef \sum_{n\in\ZZ}u_{n}.
\]
According to Proposition \ref{prp:propcriteria}, it is enough to verify that \eqref{eq:criteria1} and \eqref{eq:criteria3} hold for the kernels $k_n = \zeta u_{n}$ with constant $c = 2^{-1/2}$.
For all $\varepsilon > 0$ and $n\in\ZZ$, we have
\[\begin{split}
\int_{S}|(\zeta u_{n})(\bx)| \, (1+2^{-n/2}|\bx|)^{\varepsilon} \,\dd\rho(\bx)
&\lesssim \int_{S}| u_{n}(\bx)| \,(1+2^{-n/2}|\bx|)^{\varepsilon} \,\dd\rho(\bx)\\
&=\int_{S}| u_{n}^{*}(\bx)| \,(1+2^{-n/2}|\bx|)^{\varepsilon} \,\dd\rho(\bx)
\lesssim_\varepsilon 1.
\end{split}\]
where the last inequality is given by \eqref{eq:inequalitycriteria1}. 
Moreover, $(\zeta u_{n})^{*}=\check{\zeta} u_{n}^{*}$ for some $\check{\zeta}\in C_{c}^{\infty}(S)$; thus, by the Leibniz rule, if we set $K = \supp \check\zeta$, then
\[\begin{split}
\left\| \nabla (\zeta u_{n})^{*}\right\|_{L^{1}(S)}
&\lesssim \left\| u_{n}^{*} \chr_K \right\|_{L^{1}(S)} +\left\| \nabla (u_{n}^{*}) \chr_K \right\|_{L^{1}(S)} \\
&\leq \int_{2^n}^{2^{n+1}} t^{-1/2} \int_{K} |\vfX_j h_t(\bx)| \,\dd\rho(\bx) \,\dd t \\
&\quad+\sum_{\ell=0}^{\dimN} \int_{2^n}^{2^{n+1}} t^{-1/2} \int_K |\vfX_\ell(\vfX_j h_t)(\bx)| \,\dd\rho(\bx) \,\dd t \\
&\lesssim \int_{2^n}^{2^{n+1}} t^{-3/2} \,\dd t \simeq 2^{-n/2},
\end{split}\]
where we used the estimates from Propositions \ref{prp:heateasyestimates} and \ref{prp:higherheatestimates}.
\end{proof}

In order to complete the proof of the $L^p$-boundedness for $p<2$ of the adjoint Riesz transforms $\Rz_j^*$, we can therefore concentrate on their ``parts at infinity'', which shall be the focus of the rest of the paper.

\section{Kernel asymptotics for the Riesz transforms}\label{s:sectionkernelasy}

Starting from the formula for the heat kernel $h_t = \krn_{\ee^{-t\Delta}}$ in Proposition \ref{prp:ADJ},
and arguing much as in \cite[Proposition 2.2]{M23},
by subordination we can obtain an explicit formula for the convolution kernel of $\Delta^{-1/2}$.

\begin{proposition}
The distribution $\krn_{\Delta^{-1/2}}$ coincides with a smooth function away from the origin, and $\delta^{-1/2} \krn_{\Delta^{-1/2}}$ is radial. Moreover, for all $r>0$, if $\dimZ$ is even, then 
\begin{equation}\label{eq:expressionkernel}
(\delta^{-1/2} \krn_{\Delta^{-1/2}})(r)
= 2^{-\dimV-\frac{\dimZ}{2}} \pi^{-\frac{\dimN+2}{2}} \DfD_r^{\dimZ/2} \DfE_r^{\dimV/2-1}\left(\frac{1}{r\sinh (r/2)}\right),
\end{equation}
while, if $\dimZ$ is odd, then 
\begin{equation}\label{eq:expressionkernel2}
\begin{split}
&(\delta^{-1/2} \krn_{\Delta^{-1/2}})(r) 
= 2^{-\dimV-\frac{\dimZ}{2}}\pi^{-\frac{\dimN+3}{2}}\int_{r}^{\infty} \DfD_s^{(\dimZ+1)/2} \DfE_s^{\dimV/2-1} \left(\frac{1}{s\sinh (s/2)}\right) \,\dd\nu_r(s)\\
&\quad=2^{-\dimV-\frac{\dimZ}{2}} \pi^{-\frac{\dimN+3}{2}} \DfD_r^{(\dimZ+1)/2} \int_{r}^{\infty} \DfE_s^{\dimV/2-1} \left(\frac{1}{s\sinh (s/2)}\right) \,\dd\nu_r(s).
\end{split}
\end{equation}
\end{proposition}
\begin{proof}
Notice that, at least formally,
\begin{equation}\label{eq:formularieszkernel}
    	\krn_{\Delta^{-1/2}}=\frac{1}{\sqrt{\pi}}\int_{0}^{\infty}h_{t} \,\frac{\dd t}{t^{1/2}}.
\end{equation}

Assume first $\dimZ$ even. Substituting the formula \eqref{eq:formulahteven} into the expression \eqref{eq:formularieszkernel}, we obtain
\begin{equation}\label{eq:Keven_prelim}
(\delta^{-1/2} \krn_{\Delta^{-1/2}})(r)=2^{-\dimV-\frac{\dimZ}{2}}\pi^{-\frac{\dimN+1}{2}}\int_{0}^{\infty} \DfD_r^{\dimZ/2} \DfE_r^{\dimV/2} h_t^\RR(r) \,\frac{\dd t}{t^{1/2}}.
\end{equation}
As $\dimV$ is always even, we have $\dimV/2\geq 1$. In particular, much as in \cite[fourth display equation on p.\ 9]{M23}, for all $r>0$,
\begin{equation}\label{eq:equation1}
\int_{0}^{\infty} \DfE_r h_{t}^{\RR}(r) \,{\frac{\dd t}{t^{1/2}}}
=\frac{1}{4\sqrt{\pi}}\frac{r}{\sinh(r/2)}\int_{0}^{\infty}\ee^{-r^{2}/4t} \,\frac{\dd t}{t^{2}}
=\frac{1}{\sqrt{\pi}}\frac{1}{r\sinh (r/2)}.
\end{equation}
By exchanging differentiation and integration in \eqref{eq:Keven_prelim} and using \eqref{eq:equation1}, we obtain the expression \eqref{eq:expressionkernel}. 
	 
Now consider the case where $\dimZ$ is odd. Plugging \eqref{eq:formulahtodd} into \eqref{eq:formularieszkernel}, we find
\begin{equation}\label{eq:Kodd_prelim}
(\delta^{-1/2} \krn_{\Delta^{-1/2}})(r)
= 2^{-\dimV-\frac{\dimZ}{2}} \pi^{-\frac{\dimN+2}{2}} 
\int_{0}^{\infty} \int_{r}^{\infty} \DfD_s^{(\dimZ+1)/2}\DfE_s^{\dimV/2} h_t^\RR(s) \,\dd\nu_r(s) \,\frac{\dd t}{t^{1/2}}.
\end{equation}
Interchanging the order of integration in \eqref{eq:Kodd_prelim} and using \eqref{eq:equation1} and \eqref{eq:byparts}, we finally obtain that
\[
\begin{aligned}
&(\delta^{-1/2} \krn_{\Delta^{-1/2}})(r)\\
&= 2^{-\dimV-\frac{\dimZ}{2}} \pi^{-\frac{\dimN+2}{2}} \DfD_r^{(\dimZ+1)/2}
\int_{r}^{\infty} \DfE_s^{\dimV/2-1}\int_{0}^{\infty} \DfE_s h_t^\RR(s) \,\frac{\dd t}{t^{1/2}} \,\dd\nu_r(s) \\
&= 2^{-\dimV-\frac{\dimZ}{2}} \pi^{-\frac{\dimN+3}{2}} \DfD_r^{(\dimZ+1)/2}
\int_{r}^{\infty} \DfE_s^{\dimV/2-1}\left( \frac{1}{s\sinh (s/2)}\right) \,\dd\nu_r(s),
\end{aligned}
\]
which is the second expression in \eqref{eq:expressionkernel2}; the first one follows by applying \eqref{eq:byparts}.
\end{proof}

A variant of the arguments in \cite[Proposition 2.3]{M23} and \cite[Proposition 6.3]{MP24} now allows us to obtain precise asymptotics at infinity for the convolution kernel of $\Delta^{-1/2}$ and its derivatives.

\begin{proposition}\label{prp:ThkL}
For all $r>0$, 
\begin{equation}\label{eq:Phi}
( \delta^{-1/2} \krn_{\Delta^{-1/2}})(r) = C_{\dimV,\dimZ}\Phi_{\dimV,\dimZ}(\cosh( r/2)),
\end{equation}
where
\begin{equation}\label{eq:PhiC}
C_{\dimV,\dimZ} \defeq 2^{-\dimV-\dimZ-1}\pi^{-\frac{\dimV+\dimZ+3}{2}},
\end{equation}
while $\Phi_{\dimV,\dimZ} : (1,\infty) \to \Rpos$ is smooth and satisfies
\begin{equation}\label{eq:Phi0}
\Phi_{\dimV,\dimZ}(X)
= \frac{\Gamma\left( \frac{\dimV}{4}+\frac{1}{2} \right) \Gamma\left( \frac{\dimV}{4}+\frac{\dimZ}{2} \right) }{X^{\frac{\dimV}{2}+\dimZ}\log X}\left(1+O\left(\frac{1}{\log X}\right) \right) \quad\text{as } X \to \infty
\end{equation}
and
\begin{equation}\label{eq:Phi1}
-\frac{1}{2X}\partial_{X} \Phi_{\dimV,\dimZ}(X)=\Phi_{\dimV,\dimZ+2}(X).
\end{equation}
\end{proposition}
\begin{proof}
Under the change of variables
\[
X = \cosh (r/2),
\]
we see that
\[
\sinh (r/2)=\sqrt{X^{2}-1}, \quad\sinh r=2X\sqrt{X^{2}-1}, \quad\cosh r=2X^{2}-1,
\]
and
\[
r=2\arccosh X=2\log(X+\sqrt{X^{2}-1}).
\]
Consequently,
\begin{equation}\label{eq:phi0}
\frac{2}{r\sinh(r/2)}=\frac{1}{\sqrt{X^2-1}\log(X+\sqrt{X^{2}-1})} \eqdef \phi_0(X),
\end{equation}
where $\phi_{0}$ is a smooth function on $(1,\infty)$.

Moreover, by the chain rule,
\[
\partial_X=\frac{2}{\sinh (r/2)}\partial_r,
\]
whence, by \eqref{eq:DfDE},
\begin{equation}\label{eq:DfTDE}
2\DfD_r=-\frac{1}{2X}\partial_X \eqdef \DfTD_X, \qquad 2\DfE_r = -\partial_X \eqdef \DfTE_X.
\end{equation}

When $\dimZ$ is even, from \eqref{eq:expressionkernel}, \eqref{eq:phi0} and \eqref{eq:DfTDE} we obtain
\[\begin{split}
( \delta^{-1/2} \krn_{\Delta^{-1/2}}) (2\arccosh X)=
2^{1-\dimV/2} \sqrt{\pi} \, C_{\dimV,\dimZ} \DfTD^{\dimZ/2}_X \DfTE^{\dimV/2-1}_X \phi_{0}(X),
\end{split}\]
where $C_{\dimV,\dimZ}$ is as in \eqref{eq:PhiC}; 
when $\dimZ$ is odd, from \eqref{eq:expressionkernel2} and the change of variables $Y=\cosh(s/2)$ we derive
\begin{multline*}
		(\delta^{-1/2} \krn_{\Delta^{-1/2}}) (2\arccosh X)\\
		=2^{1-\dimV/2} C_{\dimV,\dimZ} \int_{X}^{\infty} \DfTD_Y^{(\dimZ+1)/2} \DfTE_Y^{\dimV/2-1} \phi_{0}(Y)\, \frac{2Y}{(Y^{2}-X^{2})^{1/2}} \,\dd Y.
\end{multline*}
Thus, \eqref{eq:Phi} holds true by setting
\begin{equation}\label{eq:PhiDef}
\Phi_{\dimV,\dimZ}(X) \defeq \begin{cases}
2^{1-\dimV/2} \sqrt{\pi} \,\DfTD_X^{\dimZ/2} \DfTE_X^{\dimV/2-1} \phi_0(X), &\dimZ \text{ even},\\[.5em]
2^{1-\dimV/2} \int_X^\infty \DfTD_Y^{(\dimZ+1)/2} \DfTE_Y^{\dimV/2-1} \phi_0(Y) \frac{2Y}{\sqrt{Y^2-X^2}} \,\dd Y, &\dimZ \text{ odd}.
\end{cases}
\end{equation}
From the above expression, together with \eqref{eq:byparts}, it is clear that $\DfTD_X \Phi_{\dimV,\dimZ}(X) = \Phi_{\dimV,\dimZ+2}(X)$, whence \eqref{eq:Phi1} follows. In order to derive the asymptotics \eqref{eq:Phi0}, we shall now study the behaviour of $\DfTD_X^u \DfTE_X^v \phi_0(X)$ as $X \to \infty$ for any $u,v \in \NN$.

Our strategy is to expand $\phi_{0}(X)$ into a double series in powers of $1/X$ and $1/\log X$; namely, from \cite[last display of p.\ 11]{M23}, we know that there exist some coefficients $c_{m,\ell}\in\RR$, with $c_{0,0}=1$, such that
\begin{equation}\label{eq:series}
		\phi_0(X)=\frac{1}{X\log X}\sum_{\ell,m\geq0}c_{\ell,m}\frac{1}{X^\ell}\frac{1}{\log^mX},
\end{equation}
for sufficiently large $X$.
Let us now study the action of the differential operators $\DfTD_X = - \frac{1}{2X} \partial_{X}$ and $\DfTE_X = -\partial_X$ on the series. We observe that
\[
\DfTE_X \left(\frac{1}{X}\right) = \frac{1}{X^2}, \quad 
\DfTE_X \left(\frac{1}{\log X}\right) = \frac{1}{X\log^2X}.
\]
Therefore, termwise differentiation of \eqref{eq:series} yields that, for any $v\in\NN$,
\[
\DfTE_X^{v}\phi_0(X)
= \frac{1}{X^{v+1}\log X} \sum_{\ell,m \geq 0} c_{\ell,m}^{0,v} \frac{1}{X^\ell} \frac{1}{\log^m X},
\]
where the coefficients satisfy the recurrence relationh $c_{0,0}^{0,v+1} = (v+1) c_{0,0}^{0,v}$. As $c_{0,0}^{0,0}=c_{0,0}=1$, we conclude that $c_{0,0}^{0,v}=v!$.
Similarly, for $\DfTD_X$, we compute
\[
\DfTD_X \left(\frac{1}{X}\right) = \frac{1}{2X^3}, \quad 
\DfTD_X \left(\frac{1}{\log X}\right) = \frac{1}{2X^{2}\log^2X}
\]
and for any $u\in\NN$,
\[
\DfTD^{u}_X \DfTE^{v}_X \phi_0(X)
= \frac{1}{X^{v+2u+1}\log X} \sum_{\ell,m \geq 0} c_{\ell,m}^{u,v} \frac{1}{X^\ell} \frac{1}{\log^m X}
\]
with recursive structure $c_{0,0}^{u+1,v} = \frac{v+1+2u}{2} c_{0,0}^{u,v}$. Using $c_{0,0}^{0,v}=v!=\Gamma(v+1)$, we can express the leading coefficient explicitly: 
\[
c_{0,0}^{u,v}
=\frac{\Gamma(v/2+u+1/2)}{\Gamma(v/2+1/2)}\Gamma(v+1)
=\pi^{-1/2} 2^v \Gamma(v/2+1) \Gamma(v/2+u+1/2),
\]
where the last equality follows from the duplication formula for the Gamma function \cite[5.5.5]{DLMF}.
Putting everything together, we conclude that, for all $u,v \in \NN$, 
\[
2^{-v} \sqrt{\pi} \, \DfTD^{u}_X \DfTE^{v}_X \phi_{0}(X)
= \frac{ \Gamma\left(\frac{v+2}{2}\right) \Gamma\left( \frac{v+2u+1}{2}\right)}{X^{v+2u+1}\log X} \left( 1+O\left( \frac{1}{\log X}\right) \right) \quad\text{as } X \to \infty.
\]

Plugging the latter expression into \eqref{eq:PhiDef} immediately gives \eqref{eq:Phi0} in case $\dimZ$ is even. If instead $\dimZ$ is odd, we obtain
\begin{equation}\label{eq:integralmzodd1}
\Phi_{\dimV,\dimZ}(X) 
= \frac{2 \Gamma\left(\frac{\dimV}{4}+\frac{1}{2}\right) \Gamma\left( \frac{\dimV}{4}+\frac{\dimZ+1}{2}\right)}{\sqrt{\pi}}
 \int_X^\infty \frac{(Y^2-X^2)^{-\frac{1}{2}} \,\dd Y}{Y^{\frac{\dimV}{2}+\dimZ}\log Y} \left( 1+O\left( \frac{1}{\log X}\right) \right)
\end{equation}
and it just remains to estimate the latter integral. Much as in \cite[proof of Proposition 6.3]{MP24},
we apply the binomial series expansion for $\alpha>0$ and $\xi>\eta>0$:
\begin{equation}\label{eq:binomial}
(\xi-\eta)^{-\alpha} = \sum_{k\in\NN} \binom{\alpha+k-1}{k} \, \xi^{-\alpha-k} \,\eta^{k}
\end{equation}
(notice that all summands in \eqref{eq:binomial} are positive); thus,
\begin{equation}\label{eq:integralmzodd}
\begin{split}
\int_X^\infty \frac{(Y^2-X^2)^{-\frac{1}{2}} \,\dd Y}{Y^{\frac{\dimV}{2}+\dimZ}\log Y} 
&=\sum_{k\in\NN} \binom{k-\frac{1}{2}}{k} X^{2k} \int_{X}^{\infty} \,\frac{\dd Y}{Y^{2k+\frac{\dimV}{2}+\dimZ+1}\log Y}\\
&=\sum_{k\in\NN} \binom{k-\frac{1}{2}}{k} X^{2k} \int_{(2k+\frac{\dimV}{2}+\dimZ)\log X}^{\infty} \,\frac{\dd x}{x \,\ee^{x}} \\
&=
\sum_{k\in\NN} \frac{\binom{k-\frac{1}{2}}{k}}{2k+\frac{\dimV}{2}+\dimZ} \frac{1}{X^{\frac{\dimV}{2}+\dimZ}\log X}\left(1+O\left(\frac{1}{\log X }\right) \right),
\end{split}
\end{equation}
as $X\to \infty$,
where we used that $\int_A^\infty \frac{\dd x}{x \,\ee^x} = \frac{1}{A \,\ee^A} \left( 1+ O\left(\frac{1}{A} \right) \right)$ for $A > 0$.
Moreover, by \cite[\S 1.4, eq.\ (2)]{EMOT1},
\begin{equation}\label{eq:constantbecoms}
\sum_{k\in\NN} \frac{\binom{k-\frac{1}{2}}{k}}{2k+\frac{\dimV}{2}+\dimZ} 
= \frac{1}{2} \sum_{k\in\NN} \frac{(-1)^k \binom{-\frac{1}{2}}{k}}{k+\frac{\dimV}{4}+\frac{\dimZ}{2}} 
= \frac{\sqrt{\pi} \,\Gamma(\frac{\dimV}{4}+\frac{\dimZ}{2})}{2\Gamma(\frac{\dimV}{4}+\frac{\dimZ+1}{2})}.
\end{equation}
Plugging \eqref{eq:integralmzodd} and \eqref{eq:constantbecoms} into \eqref{eq:integralmzodd1}
eventually yields \eqref{eq:Phi0} for $\dimZ$ odd.
\end{proof}

By differentiating the expression for $\krn_{\Delta^{-1/2}}$ in Proposition \ref{prp:ThkL}, we now derive expressions for the convolution kernels of the Riesz transforms $\Rz_j = \vfX_j \Delta^{-1/2}$.

\begin{proposition}\label{prp:Thk}
With the notation of Proposition \ref{prp:ThkL},
for all $\bx = (x,z,a)\in S\setminus\{(0,0,1)\}$, if we set $r=|\bx|$ then
\begin{equation}\label{eq:expressionk0*}
\krn_{\Rz_{0}-\Rz_{0}^{*}}(\bx)
=-\frac{1}{2}C_{\dimV,\dimZ} a^{-Q/2} \left( 1-\frac{1}{a} \right) \left(1+a+\frac{|x|^{2}}{4}\right) \Phi_{\dimV,\dimZ+2}(\cosh(r/2)),
\end{equation}
and, for $j=1,\dots,\dimV$,
\begin{equation}\label{eq:expressionkiX}
\krn_{\Rz_j}(\bx)
= -\frac{C_{\dimV,\dimZ}}{4} a^{-(Q+1)/2} \left( \left(1+a+\frac{|x|^{2}}{4}\right) x_j+[x,e_j]\cdot z\right) \Phi_{\dimV,\dimZ+2}(\cosh(r/2)) ,
\end{equation}
while, for $j=1,\dots,\dimZ$,
\begin{equation}\label{eq:expressionkiZ}
\krn_{\Rz_{j+\dimV}}(\bx)
= -\frac{C_{\dimV,\dimZ}}{2} a^{-Q/2} z_{j} \Phi_{\dimV,\dimZ+2}(\cosh(r/2)) .
\end{equation}
\end{proposition}
\begin{proof}
For $j=0,\dots,\dimN$,
\begin{equation}\label{eq:kR01}
		\krn_{\Rz_{j}}
		=\vfX_{j} \krn_{\Delta^{-1/2}}
		=(\vfX_{j}\delta^{1/2}) \, \delta^{-1/2} \krn_{\Delta^{-1/2}}+ \delta^{1/2} \vfX_{j}(\delta^{-1/2} \krn_{\Delta^{-1/2}}).
\end{equation}
Now, \eqref{eq:Phi} expresses $\delta^{-1/2} \krn_{\Delta^{-1/2}}$ in terms of $\Phi_{\dimV,\dimZ}$; moreover, with the notation $r = |\bx|$,
from the chain rule and \eqref{eq:Phi1} we derive that
\begin{equation}\label{eq:X0Psi1}
\begin{split}
&\vfX_{j}(\Phi_{\dimV,\dimZ}(\cosh (r/2)))\\
&=\Phi'_{\dimV,\dimZ}(\cosh (r/2)) \,\vfX_{j}(\cosh(r/2))\\
&=\Phi'_{\dimV,\dimZ}(\cosh (r/2)) \,\frac{1}{2\cosh(r/2)} \vfX_{j}(\cosh^{2}(r/2))\\
&=-\Phi_{\dimV,\dimZ+2}(\cosh(r/2)) \,\vfX_{j}(\cosh^{2}(r/2)).
\end{split}
\end{equation}
In addition, from the formulas \eqref{eq:modularfunction}, \eqref{eq:leftinvariant} and \eqref{eq:distance}, it is easily seen that
\begin{equation}\label{eq:leftinvariantcosh}
\begin{aligned}
\vfX_{0}(\cosh^{2}(r/2))&=-\cosh^{2}(r/2)+\frac{1}{2}\left(1+a+\frac{|x|^{2}}{4}\right),\\
\vfX_{j}(\cosh^{2}(r/2))&=\frac{a^{-1/2}}{4}\left( \left(1+a+\frac{|x|^{2}}{4}\right) x_{j}+[x,e_{j}]\cdot z\right)  &&\text{for } j=1,\dots,\dimV,\\
\vfX_{j+\dimV}(\cosh^{2}(r/2))&=\frac{z_{j}}{2} &&\text{for } j=1,\dots,\dimZ.
\end{aligned}
\end{equation}

Thus, by combining equations \eqref{eq:Phi}, \eqref{eq:kR01}, \eqref{eq:X0Psi1}, \eqref{eq:leftinvariantmod} and \eqref{eq:leftinvariantcosh}, we immediately obtain \eqref{eq:expressionkiX} and \eqref{eq:expressionkiZ}, together with
\begin{multline}\label{eq:kR03}
\krn_{\Rz_{0}}(\bx)
=-\frac{Q}{2}C_{\dimV,\dimZ}\,a^{-Q/2}\Phi_{\dimV,\dimZ}(\cosh (r/2))\\
-\frac{1}{2}C_{\dimV,\dimZ}a^{-Q/2}\left(1+a+\frac{|x|^{2}}{4}-2\cosh^{2}(r/2)\right) \Phi_{\dimV,\dimZ+2}(\cosh(r/2)).
\end{multline}
Now, by using the fact that $(\delta^{1/2}f)^{*}(\bx)=\delta^{1/2}(\bx) \overline{f(\bx^{-1})}$, we also get
\begin{multline}\label{eq:kR04}
			\krn_{\Rz_{0}^{*}}(\bx) = \krn_{\Rz_0}^*(\bx) 
			=-\frac{Q}{2}C_{\dimV,\dimZ}\,a^{-Q/2}\Phi_{\dimV,\dimZ}(\cosh (r/2))\\
			-\frac{1}{2}C_{\dimV,\dimZ}a^{-Q/2}\left(\frac{1}{a}(1+a+\frac{|x|^{2}}{4})-2\cosh^{2}(r/2)\right) \Phi_{\dimV,\dimZ+2}(\cosh(r/2)),
\end{multline}
and \eqref{eq:expressionk0*} follows by subtracting \eqref{eq:kR03} from \eqref{eq:kR04}.
\end{proof}

We shall say that a function $F \in L^1_\loc(S \setminus \{(0,0,1)\})$ is \emph{integrable at infinity} if $F$ is integrable on the complement of any neighbourhood of the identity element.

We can now use the integration formulas from Corollary \ref{cor:radial} to simplify the expressions for the kernels of the Riesz transforms obtained in Proposition \ref{prp:Thk}, by identifying a main term and a remainder, where the latter is integrable at infinity. The following result may be compared to \cite[Proposition 3.4]{M23}.

For a function $F:N\to\CC$ and $\lambda>0$, as in \eqref{eq:sect_dil}, we shall use the notation 
\begin{equation*}
F_{(\lambda)}(x,z)
\defeq \lambda^{-Q} F(\dil_{1/\lambda} (x,z))
=\lambda^{-Q} F(\lambda^{-1/2}x, \lambda^{-1}z).
\end{equation*}
This notation will be of convenience in presenting the following expressions.

\begin{proposition}\label{prp:TheoremkernelRiesz}
We can write
\begin{align}
\label{eq:kRz0skew}
\krn_{\Rz_{0}-\Rz^{*}_{0}}& = - \tilde C_{\dimV,\dimZ} (\tilde{K}_{0}+K_{0}) +p_{0},\\
\label{eq:kRzj}
\krn_{\Rz^{*}_{j}} &= - \tilde C_{\dimV,\dimZ} K_{j}+p_{j} \qquad\text{for } j=1,\dots,\dimN,
\end{align}
where
\begin{equation}\label{eq:tildeC}
\tilde C_{\dimV,\dimZ} \defeq
 2^{1-\frac{\dimV}{2}} \pi^{-\frac{\dimV+\dimZ+3}{2}} \Gamma\left( \frac{\dimV}{4}+\frac{1}{2}\right) \Gamma\left( \frac{\dimV}{4}+\frac{\dimZ}{2}+1\right),
\end{equation}
the $p_{j}$ are integrable at infinity for $j=0,\dots,\dimN$, while 
\begin{align}
\label{eq:tK0}
\tilde{K}_{0}(\bx) &= r_{0}(x,z) \,\frac{\chr_{\{a\geq \ee\}}+\chr_{\{a\leq \frac{1}{\ee}\}}}{\log a},\\
\label{eq:K0}
K_{0}(\bx) &= \left[ (r_{0})_{(a)}(x,z)-r_{0}(x,z)\right] \frac{\chr_{\{a\geq \ee\}}}{\log a},\\
\label{eq:Kj}
K_{j}(\bx) &= r_{j}(x,z) \,\frac{\chr_{\{a\leq \frac{1}{\ee}\}}}{\log a} \qquad\text{for } j=1,\dots,\dimN,
\end{align}
and
\begin{align}
\label{eq:r0}
r_{0}(x,z)&= \frac{ 1+\frac{1}{4}|x|^{2} }{H(x,z)^{\frac{Q+2}{2}}},\\
\label{eq:rjV}
r_{j}(x,z)&= \frac{1}{2} \frac{\left( 1+\frac{1}{4}|x|^{2}\right) x_{j}-[x,e_{j}] \cdot z }{H(x,z)^{\frac{Q+2}{2}}} &\text{for } j=1,\dots,\dimV,\\
\label{eq:rjZ}
r_{j+\dimV}(x,z)&= \frac{z_{j}}{H(x,z)^{\frac{Q+2}{2}}} &\text{for } j=1,\dots,\dimZ,
\end{align}
and
\begin{equation}\label{eq:Hnorm}
H(x,z) \defeq \left( 1+\frac{1}{4} |x|^{2}\right)^{2}+|z|^{2}.
\end{equation}
\end{proposition}
\begin{proof}
Recall that $Q=\frac{\dimV}{2}+\dimZ$ and notice that, by \eqref{eq:PhiC} and \eqref{eq:tildeC},
\[
C_{\dimV,\dimZ}\Gamma \left( \frac{\dimV}{4}+\frac{1}{2}\right) \Gamma\left( \frac{\dimV}{4}+\frac{\dimZ}{2}+1\right)
= 2^{-Q-2} \tilde C_{\dimV,\dimZ}.
\]
Thus, by Propositions \ref{prp:ThkL} and \ref{prp:Thk}, we obtain that, for $j=1,\dots,\dimZ$,
\begin{equation}\label{eq:RzZprelim}
\begin{split}
	\krn_{\Rz_{j+\dimV}}(\bx)
	&=-\frac{C_{\dimV,\dimZ}}{2} a^{-Q/2} z_{j} \Phi_{\dimV,\dimZ+2}(\cosh(r/2))\\
	&=-\frac{\tilde C_{\dimV,\dimZ}}{2^{Q+3}}\frac{\delta^{\frac{1}{2}}(\bx) \, z_{j} }{\cosh^{Q+2}\frac{r}{2} \, \log \cosh\frac{r}{2}} (1+O(r^{-1}) ),
\end{split}
\end{equation}
as $r \to \infty$, where we use the notation $r = |\bx|$.
From \eqref{eq:Rinta} we see that the contribution from the Big-O term in \eqref{eq:RzZprelim} is integrable at infinity, as
\[
	\int_{r \geq 1} \frac{\delta^{\frac{1}{2}}(\bx) \,|z| \,r^{-1}}{\cosh^{Q+2} \frac{r}{2} \, \log \cosh \frac{r}{2}} \,\dd\rho(\bx)
	\simeq \int_{1}^{\infty} \ee^{-\frac{Q+2}{2} r} r^{-2} \ee^{\frac{Q+2}{2}r} \,\dd r
	< \infty.
\]
Similarly, the main term in \eqref{eq:RzZprelim} for $a\leq \ee$ is integrable at infinity by \eqref{eq:Rintm}:
\[
\begin{aligned}
	\int_{r\geq 1} \chr_{\{a\leq \ee\}} \frac{\delta^{\frac{1}{2}}(\bx) \,|z| \,r^{-1}}{\cosh^{Q+2} \frac{r}{2} \, \log \cosh \frac{r}{2}}  \,\dd\rho(\bx)
	&\simeq \int_{1}^{\infty} \ee^{-\frac{Q+2}{2}r} r^{-1} \ee^{\frac{Q+1}{2}r} \,\dd r < \infty.
\end{aligned}
\]
This proves that $\krn_{\Rz_{j+\dimV}}(\bx)$ differs from
\begin{equation}\label{eq:RzZbetter}
-\frac{\tilde C_{\dimV,\dimZ}}{2^{Q+3}}\frac{\delta^{\frac{1}{2}}(\bx) \, z_{j} }{\cosh^{Q+2}\frac{r}{2} \, \log \cosh\frac{r}{2}} \chr_{\{a \geq \ee\}} 
\end{equation}
by a term which is integrable at infinity.

Now, from \eqref{eq:distance}, if $a\geq \ee$, then
\begin{equation}\label{eq:coshr2approximation}
\begin{split}
\cosh\frac{r}{2}&=\frac{a^{\frac{1}{2}}}{2}\left[ \left( 1+a^{-1}+\frac{1}{4}\left| \frac{x}{a^{\frac{1}{2}}}\right| ^{2}\right) ^{2}+\left| \frac{z}{a}\right| ^{2}\right] ^{\frac{1}{2}}\\
&=\frac{a^{\frac{1}{2}}}{2}\left[ \left( 1+\frac{1}{4}\left| \frac{x}{a^{\frac{1}{2}}}\right| ^{2}\right) ^{2}+\left| \frac{z}{a}\right| ^{2}+a^{-2}+2a^{-1}\left( 1+\frac{1}{4}\left| \frac{x}{a^{\frac{1}{2}}}\right| ^{2}\right)\right] ^{\frac{1}{2}}\\
&=\frac{a^{\frac{1}{2}}}{2}H(\dil_{1/a}(x,z))^{\frac{1}{2}}(1+O(a^{-1})),
\end{split}
\end{equation}
where we are using the notation \eqref{eq:dil} and \eqref{eq:Hnorm};
additionally,
\[
\log\cosh \frac{r}{2}=\frac{1}{2}\log a +\frac{1}{2}\log H(\dil_{1/a}(x,z))-\log 2+O(a^{-1}), 
\]
hence,
\begin{equation}\label{eq:logcoshr2approximation}
\begin{split}
	\frac{1}{\log \cosh \frac{r}{2}}=&\frac{1}{\frac{1}{2}\log a}-\frac{\frac{1}{2}\log H(\dil_{1/a}(x,z))-\log 2+O(a^{-1})}{\frac{1}{2}\log a\log\cosh\frac{r}{2}}\\
	=&\frac{1}{\frac{1}{2}\log a}\left( 1+O\left(\frac{1+\log H(\dil_{1/a}(x,z))}{\log a} \right) \right);
\end{split}
\end{equation}
here we have used the fact that,  by \eqref{eq:est_d_a}, $\log\cosh\frac{r}{2} \simeq r \geq \log a \geq 1$ for $a\geq \ee$.
From \eqref{eq:coshr2approximation} and \eqref{eq:logcoshr2approximation} we obtain in particular that
\begin{equation}\label{eq:chlchapprox}
\frac{1}{\cosh^{Q+2}\frac{r}{2} \, \log\cosh \frac{r}{2}} = \frac{2^{Q+3} a^{-Q/2-1}}{H(\dil_{1/a}(x,z))^{\frac{Q+2}{2}} \, \log a} \left( 1+O\left(\frac{1+\log H(\dil_{1/a}(x,z))}{\log a} \right) \right).
\end{equation}
By using \eqref{eq:chlchapprox}, the term \eqref{eq:RzZbetter} can be rewritten as
\begin{equation}\label{eq:mainitem1}
-\tilde C_{\dimV,\dimZ} \frac{a^{-Q} ( z_{j}/a) }{H(\dil_{1/a}(x,z))^{\frac{Q+2}{2}} \log a}\chr_{\{a\geq \ee\}}
\left( 1+O\left(\frac{1+\log H(\dil_{1/a}(x,z))}{\log a} \right) \right) .	
\end{equation}
Now, the contribution from the Big-O item in \eqref{eq:mainitem1} is integrable, as
\begin{multline*}
\int_{\ee}^{\infty} \int_{N} \frac{a^{-Q}|z/a| }{H(\dil_{1/a}(x,z))^{\frac{Q+2}{2}} \log a} \frac{1+\log H(\dil_{1/a}(x,z))}{\log a}\,\dd x \,\dd z \,\frac{\dd a}{a}\\
=\int_{1}^{\infty} \int_{N} \frac{ |z| }{u^{2}}\frac{ 1+\log\left( \left( 1+\frac{1}{4}|x| ^{2}\right) ^{2}+|z| ^{2}\right)}{\left[ \left( 1+\frac{1}{4}|x| ^{2}\right) ^{2}+|z| ^{2}\right] ^{\frac{Q+2}{2}}} \,\dd x \,\dd z \,\dd u<\infty.
\end{multline*}
For the remaining part in \eqref{eq:mainitem1}, we take adjoints and obtain that
$\krn_{\Rz^{*}_{j+\dimV}}=\krn_{\Rz_{j+\dimV}}^{*}$ differs from 
\[
-\tilde C_{\dimV,\dimZ} \frac{z_{j}}{H(x,z)^{\frac{Q+2}{2}}} \frac{\chr_{\{a\leq \frac{1}{\ee}\}}}{\log a}
\]
by a term which is integrable at infinity. In light of \eqref{eq:rjZ}, this proves \eqref{eq:kRzj} for $j=\dimV+1,\dots,\dimV+\dimZ$.

When $j=1,\dots,\dimV$, from Propositions \ref{prp:ThkL} and \ref{prp:Thk}, we obtain that
\begin{equation}\label{eq:KjVprelim}
\begin{split}
\krn_{\Rz_{j}}(\bx)&=-\frac{C_{\dimV,\dimZ}}{4}a^{-(Q+1)/2}\left( (1+a+\frac{|x|^{2}}{4})x_j+[x,e_j]\cdot z\right) \Phi_{\dimV,\dimZ+2}(\cosh(r/2))\\
&=-\frac{\tilde C_{\dimV,\dimZ}}{2^{Q+4}}\frac{\delta^{\frac{1}{2}}(\bx) \, a^{-1/2}\left( (1+a+\frac{|x|^{2}}{4}) x_j+[x,e_j]\cdot z\right)}{\cosh^{Q+2}\frac{r}{2} \,\log \cosh\frac{r}{2}} (1+O(r^{-1}) ).
\end{split}
\end{equation}
From \eqref{eq:radial456} we see that the contribution from the Big-O item in \eqref{eq:KjVprelim} is integrable at infinity, as
\[
	\int_{r \geq 1}\frac{\delta^{\frac{1}{2}}(\bx) \, a^{-1/2}\left( (1+a+\frac{|x|^{2}}{4})|x|+|z||x|\right) }{\cosh^{Q+2}\frac{r}{2} \,\log \cosh\frac{r}{2}} \frac{1}{r} \,\dd\rho(\bx)
	\simeq\int_{1}^{\infty}\ee^{-\frac{Q+2}{2}r}r^{-2} \,\ee^{\frac{Q+2}{2}r} \,\dd r.
\]
Now, in the main term in \eqref{eq:KjVprelim}, we split the factor $(1+a+\frac{|x|^{2}}{4}) x_j+[x,e_j]\cdot z$ as the sum of $x_j$ and $(a+\frac{|x|^{2}}{4}) x_j+[x,e_j]\cdot z$.
The part corresponding to the first summand $x_j$ is integrable at infinity by
\eqref{eq:Rinta}, as
\[
\int_{r\geq 1} \frac{\delta^{\frac{1}{2}}(\bx) \, a^{-1/2} |x|}{\cosh^{Q+2} \frac{r}{2} \,\log \cosh\frac{r}{2}} \,\dd\rho(\bx)	\simeq \int_{1}^{\infty} \ee^{-\frac{Q+2}{2}r} r^{-1} \,\ee^{\frac{Q+1}{2}r} \,\dd r.
\]
As for the part corresponding to the second summand, by \eqref{eq:radial4561} it is integrable at infinity when restricted to $a\leq \ee$, as
\[
\int_{r \geq 1} \chr_{\{a\leq \ee\}} \frac{\delta^{\frac{1}{2}}(\bx) \, a^{-1/2}\left( (a+\frac{|x|^{2}}{4})|x|+|z||x|\right) }{\cosh^{Q+2}\frac{r}{2} \,\log \cosh\frac{r}{2}} \,\dd \rho(\bx)
\simeq\int_{1}^{\infty} \ee^{-\frac{Q+2}{2}r}r^{-1} \ee^{(\frac{Q}{2}+\frac{3}{4})r} \,\dd r.
\]
Gathering these facts together, we conclude that $\krn_{\Rz_j}(\bx)$ differs from
\[
-\frac{\tilde C_{\dimV,\dimZ}}{2^{Q+4}}\frac{\delta^{\frac{1}{2}}(\bx) \, a^{-1/2}\left( (a+\frac{|x|^{2}}{4}) x_j+[x,e_j]\cdot z\right)}{\cosh^{Q+2}\frac{r}{2} \,\log \cosh\frac{r}{2}} \chr_{\{a \geq \ee\}}
\]
by a term that is integrable at infinity. By \eqref{eq:chlchapprox}, the latter expression can be rewritten as
\begin{equation}\label{eq:mainitem2}
-\frac{\tilde C_{\dimV,\dimZ}}{2}\frac{a^{-Q} L_j(\dil_{1/a}(x,z))}{H(\dil_{1/a}(x,z))^{\frac{Q+2}{2}} \,\log a} \chr_{\{a \geq \ee\}} \left( 1+O\left(\frac{1+\log H(\dil_{1/a}(x,z))}{\log a} \right) \right),
\end{equation}
where
\[
L_j(x,z) \defeq (1+|x|^2/4)x_j+[x,e_j] \cdot z.
\]
Now, the contribution from the Big-O item in \eqref{eq:mainitem2} is integrable, as
\begin{multline*}
\int_{\ee}^{\infty}\int_{N} \frac{a^{-Q} |L_j(\dil_{1/a}(x,z))|}{H(\dil_{1/a}(x,z))^{\frac{Q+2}{2}} \log a}
\frac{1+\log H(\dil_{1/a}(x,z))}{\log a}\,\dd x \,\dd z \,\frac{\dd a}{a}\\
\lesssim \int_{1}^{\infty}\int_{N} \frac{ \left( 1+\frac{1}{4}|x| ^{2}\right) |x| +|x||z| }{ u^{2}} \frac{1+\log\left( \left( 1+\frac{1}{4}|x|^{2}\right)^{2}+|z|^{2}\right)}{\left[ \left( 1+\frac{1}{4}|x|^{2}\right)^{2}+|z|^{2}\right]^{\frac{Q+2}{2}}} \,\dd x \,\dd z \,\dd u<\infty.
\end{multline*}
For the remaining part in \eqref{eq:mainitem2}, we take adjoints and obtain that $\krn_{\Rz^*_j} = \krn_{\Rz_j}^*$ differs from
\begin{equation*}
-\frac{\tilde C_{\dimV,\dimZ}}{2} \frac{ \left( 1+\frac{1}{4}|x| ^{2}\right)x_{i}-[x,e_{i}] \cdot z }{\left[ \left( 1+\frac{1}{4}|x|^{2}\right)^{2}+|z|^{2}\right]^{\frac{Q+2}{2}}}\frac{\chr_{\{a\leq \frac{1}{\ee}\}}}{\log a}
\end{equation*}
by a term which is integrable at infinity. In light of \eqref{eq:rjV}, this proves \eqref{eq:kRzj} for $j=1,\dots,\dimV$.

In a similar way we can discuss the kernel $\krn_{\Rz_{0}-\Rz_{0}^{*}}$. From Proposition \ref{prp:Thk} and Proposition \ref{prp:ThkL} we obtain that
\begin{equation*}
\krn_{\Rz_{0}-\Rz_{0}^{*}}(\bx)
= -\frac{\tilde C_{\dimV,\dimZ}}{2^{Q+3}} \frac{\delta^{\frac{1}{2}}(\bx) \,(1-a^{-1}) (1+a+\frac{|x|^{2}}{4})}{\cosh^{Q+2}\frac{r}{2} \,\log\cosh\frac{r}{2}}(1+O(r^{-1})).
\end{equation*}
From Corollary \ref{cor:radial} (formulas \eqref{eq:radial31} and \eqref{eq:radial32}), we then see that
$\krn_{\Rz_{0}-\Rz_{0}^{*}}$ differs from
\begin{equation}\label{eq:R0part}
-\frac{C_{\dimV,\dimZ}}{2^{Q+3}} \frac{\delta^{\frac{1}{2}}(\bx) \,( 1-a^{-1}) (1+a+\frac{|x|^{2}}{4})}{\cosh^{Q+2}\frac{r}{2}\, \log\cosh\frac{r}{2}}\chr_{\{a\leq\frac{1}{\ee}\}\cup\{ a\geq\ee\} }.
\end{equation}
by a term which is integrable at infinity.

We first restrict our attention to the region where $a\geq \ee$. If $a\geq\ee$, then 
\[
( 1-a^{-1}) (1+a+\frac{|x|^{2}}{4})=a\left( 1+\frac{1}{4}\left| \frac{x}{a^{\frac{1}{2}}}\right|^{2} \right) (1+O(a^{-1}));
\]
combined with \eqref{eq:chlchapprox}, this allows us to rewrite the part of \eqref{eq:R0part} where $a \geq \ee$ as 
\[
-\tilde C_{\dimV,\dimZ} \frac{a^{-Q}\left( 1+\frac{1}{4}\left| \frac{x}{a^{\frac{1}{2}}}\right|^{2} \right)}{H(\dil_{1/a}(x,z))^{\frac{Q+2}{2}}\log a}\chr_{\{a\geq \ee\}}
\left( 1+O\left(\frac{1+\log H(\dil_{1/a}(x,z))}{\log a} \right) \right).
\]
The contribution from the Big-O item in the above expression is integrable, as
\begin{multline*}
\int_{a}^{\infty} \int_{N} \frac{a^{-Q} \left( 1+\frac{1}{4}\left| \frac{x}{a^{\frac{1}{2}}}\right|^{2} \right)}{H(\dil_{1/a}(x,z))^{\frac{Q+2}{2}} \, \log a} 
\frac{1+ \log H(\dil_{1/a}(x,z))}{\log a} \,\dd x \,\dd z \,\frac{\dd a}{a} \\
=\int_{1}^{\infty} \int_{N} \frac{1+\frac{1}{4}|x|^{2}}{u^2} 
\frac{ 1+\log\left( \left( 1+\frac{1}{4}|x|^{2}\right)^{2}+|z|^{2}\right)}{\left( \left( 1+\frac{1}{4}|x| ^{2}\right)^{2}+|z|^{2}\right)^{\frac{Q+2}{2}}} \,\dd x \,\dd z \, \dd u< \infty.
\end{multline*}
Thus, in the region where $a\geq \ee$, $\krn_{\Rz_{0}-\Rz^{*}_{0}}(\bx)$ differs from 
\[
-\tilde C_{\dimV,\dimZ} \frac{a^{-Q}\left( 1+\frac{1}{4}\left| \frac{x}{a^{\frac{1}{2}}}\right|^{2} \right)}{H(\dil_{1/a}(x,z))^{\frac{Q+2}{2}}}\frac{\chr_{\{a\geq \ee\}}}{\log a}
\]
by a term which is integrable at infinity.
As $\krn_{\Rz_{0}-\Rz^{*}_{0}}^* = -\krn_{\Rz_{0}-\Rz^{*}_{0}}$, by taking adjoints we deduce that, in the region where $a\leq\frac{1}{\ee}$, $\krn_{\Rz_{0}-\Rz^{*}_{0}}(\bx)$ differs from 
\[
-\tilde C_{\dimV,\dimZ} \frac{ 1+\frac{1}{4}| x |^{2} }{H(x,z)^{\frac{Q+2}{2}}} \frac{\chr_{\{a\leq 1/\ee\}}}{\log a}
\]
by a term which is integrable at infinity. By summing the two expressions above, we finally deduce that 
$\krn_{\Rz_{0}-\Rz^{*}_{0}}(\bx)$ differs from 
\[
-\tilde C_{\dimV,\dimZ} \left[\frac{a^{-Q}\left( 1+\frac{1}{4}\left| \frac{x}{a^{\frac{1}{2}}}\right|^{2} \right)}{H(\dil_{1/a}(x,z))^{\frac{Q+2}{2}}}\frac{\chr_{\{a\geq \ee\}}}{\log a}
+\frac{ 1+\frac{1}{4}\left| x \right|^{2} }{H(x,z)^{\frac{Q+2}{2}}} \frac{\chr_{\{a\leq 1/\ee\}}}{\log a}\right]
\]
by a term which is integrable at infinity. In light of \eqref{eq:K0}, \eqref{eq:tK0} and \eqref{eq:r0}, this proves \eqref{eq:kRz0skew}.
\end{proof}

\begin{proposition}\label{prp:rjXjH}
The functions $r_j$ for $j=1,\dots,\dimN$ from Proposition \ref{prp:TheoremkernelRiesz} can be rewritten as
\begin{equation}\label{eq:rjXj}
r_j = Q^{-1} (\vfXN_j H^{-Q/2})^*
\end{equation}
where $H$ is as in \eqref{eq:Hnorm}, and ${}^*$ denotes involution on $N$.
\end{proposition}
\begin{proof}
From \eqref{eq:Hnorm} we see that
\[\begin{aligned}
\nabla_x [H(x,z)^{-Q/2}] &= -\frac{Q}{2} H(x,z)^{-(Q+2)/2} \left( 1+\frac{1}{4} |x|^{2}\right) x,\\
\nabla_z [H(x,z)^{-Q/2}] &= -\frac{Q}{2} H(x,z)^{-(Q+2)/2} \, 2z.
\end{aligned}\]
So from \eqref{eq:leftinvariantN} we deduce that
\[\begin{aligned}
\vfXN_j H^{-Q/2}(x,z) &= -Q \frac{\left( 1+\frac{1}{4} |x|^{2}\right) x_j + [x,e_j] \cdot z}{2 \, H(x,z)^{(Q+2)/2}}  &\text{for } j=1,\dots,\dimV,\\
\vfXN_{j+\dimV} H^{-Q/2}(x,z) &= -Q \frac{z_j}{H(x,z)^{(Q+2)/2}} &\text{for } j=1,\dots,\dimZ.
\end{aligned}\]
By taking adjoints,
\[\begin{aligned}
(\vfXN_j H^{-Q/2})^*(x,z) &= Q \frac{\left( 1+\frac{1}{4} |x|^{2}\right) x_j - [x,e_j] \cdot z}{2 \, H(x,z)^{(Q+2)/2}}  &\text{for } j=1,\dots,\dimV,\\
(\vfXN_{j+\dimV} H^{-Q/2})^*(x,z) &= Q \frac{z_j}{H(x,z)^{(Q+2)/2}} &\text{for } j=1,\dots,\dimZ,
\end{aligned}\]
which, in light of \eqref{eq:rjV} and \eqref{eq:rjZ}, gives \eqref{eq:rjXj}.
\end{proof}

\section{Reduction to the parts at infinity}\label{s:sectionlocalpart}

\begin{remark}\label{rem:productmeasure}
By formula \eqref{eq:haar}, the measure space $(S,\dd\rho)$ is the product of the measure spaces $(N,\dd x \,\dd z)$ and $(\Rpos,\dd a/a)$. In particular, we can think of the Lebesgue space $L^p(S)$ as $L^p(N \times \Rpos)$. In addition, the change of variables $a = \ee^u$ is a Lie group isomorphism between $\Rpos$ (with multiplication) and $\RR$ (with addition), and can be used to identify the measure space $(\Rpos,\dd a/a)$ with $(\RR,\dd u)$. These considerations will be central in what follows.
\end{remark}

In this section, we shall show that the $L^p$-boundedness properties of the adjoint Riesz transforms $\Rz_j^*$ can be reduced to analogous properties of the convolution operators with the kernels $\tilde K_0,K_0,K_1,\dots,K_{\dimN}$ introduced in Proposition \ref{prp:TheoremkernelRiesz}.

Among those, the kernel $\tilde K_0$ is the easiest to deal with (cf.\ \cite[Proposition 4.4]{M23}).

\begin{proposition}\label{prp:propsitiontildeK0}
Let $\tilde{K}_{0}$ be as in Proposition \ref{prp:TheoremkernelRiesz}. Then the operator 
\[
	f\mapsto f*\tilde{K}_{0}
\]
is of weak type (1,1) and bounded on $L^{p}(S)$ for all $p\in(1,\infty)$.
\end{proposition}
\begin{proof}
Recall from \eqref{eq:tK0} that $\tilde K_0(\bx) = r_0(x,z) J(\log a)$, where $J(u)\defeq \frac{\chr_{\{|u|\geq 1\}}}{u}$.
From \eqref{eq:convSN} we then see that
\begin{equation*}
(f*\tilde K_0)^{[a]} = \int_{\Rpos} (f^{[a']} *_N (r_0)_{(a)}) \, J(\log a -\log a') \,\frac{\dd a'}{a'} .
\end{equation*}
In other words, if we use the notation \eqref{eq:sect_dil} and
\[
	f_{[(x,z)]}(u) \defeq f(x,z,\ee^{u}),
\]
then $f*\tilde{K}_{0}=ABf$, where the operators $A$ and $B$ are given by
\[
	(Bf)^{[a]} = f^{[a]} *_{N} (r_{0})_{[a]}, \qquad (Af)_{[(x,z)]} = f_{[(x,z)]} *_{\RR} J.
\]

Now, from \eqref{eq:r0} we see that $r_0 \in L^1(N)$, and the scaling $F \mapsto F_{(\lambda)}$ preserves the $L^1(N)$-norm, thus from Remark \ref{rem:productmeasure} and Fubini's theorem it readily follows that $B$ is $L^p(S)$-bounded for all $p \in [1,\infty]$.
	
On the other hand, $J$ is the truncation of a Calder\'on--Zygmund kernel on $\RR$, thus the corresponding convolution operator is of weak type (1,1) and bounded on $L^{p}(\RR)$ for $p\in(1,\infty)$, see \cite[Chapter I, Section 7.1]{S95}. As a consequence (see Remark \ref{rem:productmeasure}), $A$ is of weak type (1,1) and bounded on $L^{p}(S)$ for $p\in(1,\infty)$.
\end{proof}

In light of the previous proposition and Theorem \ref{thm:theoremlocalpart}, we can now connect the $L^p$-boundedness properties of $\Rz_j$ with those of the convolution operator by $K_j$.

\begin{proposition}\label{prp:reductioninfinity}
Let $p \in (1,2]$ and $j \in \{0,\dots,\dimN\}$. The convolution operator $f \mapsto f * K_j$ is $L^p(S)$-bounded if and only if $\Rz_j^*$ is.
\end{proposition}
\begin{proof}
Consider first $j \geq 1$.
Let $\zeta \in C^\infty_c(S)$ be identically one on a neighbourhood of the identity element. Then we can decompose
\[
\Rz_j^* f = f * \krn_{\Rz_j^*} = f * (\zeta \krn_{\Rz_j^*}) + f * ((1-\zeta) \krn_{\Rz_j^*}).
\]
Now, $f \mapsto \zeta \krn_{\Rz_j^*}$ is $L^{p}$-bounded by Theorem \ref{thm:theoremlocalpart}. Moreover, by Proposition \ref{prp:TheoremkernelRiesz}, we can write
\[
(1-\zeta) \krn_{\Rz_j^*} = K_j - \zeta K_j + (1-\zeta) p_j,
\]
where $K_j \in L^1_\loc(S)$, while $p_j$ is integrable at infinity. As $\zeta$ is compactly supported and $1-\zeta$ vanishes in a neighbourhood of the identity, we have $-\zeta K_j + (1-\zeta) p_j \in L^1(G)$, thus the corresponding convolution operator is $L^p$-bounded. This shows that $\Rz_j^*$ is $L^p$-bounded if and only if $f \mapsto f * K_j$ is.

In the case $j=0$, we first notice that, as $\Rz_0$ is $L^p$-bounded by Theorem \ref{thm:Theoremboundedleq2}, we have that $\Rz_0^*$ is $L^p$-bounded if and only if the difference $\Rz_0 - \Rz_0^*$ is.
So, much as before, we write
\[
(\Rz_0 - \Rz_0^*) f = f * (\zeta \krn_{\Rz_0}) - f * (\zeta \krn_{\Rz_0^*}) + f * ((1-\zeta) \krn_{\Rz_0 - \Rz_0^*}).
\]
Now, the $L^p$-boundedness of $f \mapsto f * (\zeta \krn_{\Rz_0})$ and $f \mapsto f * (\zeta \krn_{\Rz_0^*})$ follows from Theorem \ref{thm:theoremlocalpart}. On the other hand, by Proposition \ref{prp:TheoremkernelRiesz},
\[
(1-\zeta) \krn_{\Rz_0 - \Rz_0^*} = K_0 + \tilde K_0 - \zeta (K_0 + \tilde K_0) + (1-\zeta) p_0,
\]
where $p_0$ is integrable at infinity, while $K_0, \tilde K_0 \in L^1_\loc(S)$; therefore $- \zeta (K_0 + \tilde K_0) + (1-\zeta) p_0 \in L^1(S)$ and the corresponding convolution operator is $L^p$-bounded. As $f \mapsto f * \tilde K_0$ is $L^p$-bounded by Proposition \ref{prp:propsitiontildeK0}, we conclude that $\Rz_0 - \Rz_0^*$ is $L^p$-bounded if and only if $f \mapsto f * K_0$ is.
\end{proof}

Recall from \eqref{eq:Kj} and Proposition \ref{prp:rjXjH} that, for $j=1,\dots,\dimN$, the kernels $K_j$ have the form
\[
K_j(\bx) = r_j(x,z) \,\frac{\chr_{\{a\leq \frac{1}{\ee}\}}}{\log a}, \qquad\text{where } r_j = Q^{-1} (\vfXN_j H^{-Q/2})^*
\]
and $H$ is given in \eqref{eq:Hnorm}. We shall now reduce the $L^p$-boundedness properties for the corresponding convolution operators to that of two convolution operators defined in terms of the sub-Laplacian $\opLN$ on $N$; this reduction exploits the $L^p$-boundedness of the Riesz transforms associated with $\opLN$ and may be compared to that in \cite[Proposition 3.5]{MP24bis}.

\begin{proposition}\label{prp:RzNfactorisation}
Define
\begin{equation}\label{eq:Kbullet}
\begin{aligned}
K_{\lie{v}}(\bx) &\defeq r_{\lie{v}}(x,z) \,\frac{\chr_{\{a\leq \frac{1}{\ee}\}}}{\log a}, \qquad\text{where } r_{\lie{v}} \defeq \opLN^{1/2} H^{-Q/2},\\
K_{\lie{z}}(\bx) &\defeq r_{\lie{z}}(x,z) \,\frac{\chr_{\{a\leq \frac{1}{\ee}\}}}{\log a}, \qquad\text{where } r_{\lie{z}} \defeq \opLN H^{-Q/2}.
\end{aligned}
\end{equation}
Let $p \in (1,\infty)$.
\begin{enumerate}[label=(\roman*)]
\item If $f \mapsto f * K_{\lie{v}}$ is $L^p(S)$-bounded, then the operators $f \mapsto f * K_j$ are $L^p(S)$-bounded for $j=1,\dots,\dimV$.
\item If $f \mapsto f * K_{\lie{z}}$ is $L^p(S)$-bounded, then the operators $f \mapsto f * K_j$ are $L^p(S)$-bounded for $j=\dimV+1,\dots,\dimV+\dimZ$.
\end{enumerate}
\end{proposition}
\begin{proof}
Much as in the proof of Proposition \ref{prp:propsitiontildeK0}, we can write $f * K_j = A_{\lie{n}} B_j f$, $f * K_{\lie{v}} = A_{\lie{n}} B_{\lie{v}} f$ and $f * K_{\lie{z}} = A_{\lie{n}} B_{\lie{z}} f$, where
\begin{align*}
(A_{\lie{n}} f)_{[(x,z)]} &= f_{[(x,z)]} *_{\RR} J_{\lie{n}}, &
(B_j f)^{[a]} &= f^{[a]} *_{N} (r_j)_{(a)}, \\
(B_{\lie{v}} f)^{[a]} &= f^{[a]} *_{N} (r_{\lie{v}})_{(a)}, &
(B_{\lie{z}} f)^{[a]} &= f^{[a]} *_{N} (r_{\lie{z}})_{(a)},
\end{align*}
and $J_{\lie{n}}(u) = \frac{\chr_{\{u \leq -1\}}}{u}$.

Now, for $j=1,\dots,\dimV$,
\[
Q r_j = (\vfXN_j H^{-Q/2})^* = (\vfXN_j \opLN^{-1/2} \opLN^{1/2} H^{-Q/2})^* = (r_{\lie{v}} *_N \krn_{\Rz_j^N})^* = \krn_{\Rz_j^N}^* *_N r_{\lie{v}},
\]
where $\Rz_j^N \defeq \vfXN_j \opLN^{-1/2}$ is a first-order Riesz transform
relative to the sub-Laplacian $\opLN$ on $N$, and we used that $r_{\lie{v}}^* = r_{\lie{v}}$ (as $H^* = H$ and $\opLN$ is self-adjoint). As a consequence,
\[
(Q B_j f)^{[a]} = f^{[a]} *_N \krn_{\Rz_j^N}^* *_N (r_{\lie{v}})_{(a)} = ((\Rz_j^N)^* f^{[a]} ) *_N (r_{\lie{v}})_{(a)},
\]
where we also used that $(\krn_{\Rz_j^N}^*)_{(a)} = \krn_{\Rz_j^N}^*$ by the $0$-homogeneity of $\Rz_j^N = \vfXN_j \opLN^{-1/2}$.

In other words, for $j=1,\dots,\dimV$, we can write 
\begin{equation}\label{eq:RzNfactorisation_v}
A_{\lie{n}} B_j = Q^{-1} A_{\lie{n}} B_{\lie{v}} ((\Rz_j^N)^* \otimes \id),
\end{equation}
where
\[
((\Rz_j^N)^* \otimes \id) f)^{[a]} = (\Rz_j^N)^* f^{[a]}.
\]
In much the same way, for $j=\dimV+1,\dots,\dimV+\dimZ$, we obtain that
\begin{equation}\label{eq:RzNfactorisation_z}
A_{\lie{n}} B_j = Q^{-1} A_{\lie{n}} B_{\lie{z}} ((\Rz_j^N)^* \otimes \id),
\end{equation}
where $\Rz_j^N \defeq \vfXN_j \opLN^{-1}$ is a second-order Riesz transform relative to $\opLN$ on $N$; indeed, as $\lie{z} = [\lie{v},\lie{v}]$, we can write $\vfXN_j$ as a linear combination of products $\vfXN_k \vfXN_\ell$ with $k,\ell \in \{1,\dots,\dimV\}$.

Now, Riesz transforms on $N$ of arbitrary order are known to be $L^q(N)$-bounded for all $q \in (1,\infty)$ \cite{CG84,tERS97,F75,LV85}. Thus, if $A_{\lie{n}} B_\lie{v}$ is $L^p(S)$-bounded, from \eqref{eq:RzNfactorisation_v} we deduce that the operators $A_{\lie{n}} B_j$ for $j=1,\dots,\dimV$ are $L^p(S)$-bounded too. Similarly, if $A_{\lie{n}} B_\lie{z}$ is $L^p(S)$-bounded, from \eqref{eq:RzNfactorisation_z} we deduce that the operators $A_{\lie{n}} B_j$ for $j=\dimV+1,\dots,\dimZ$ are $L^p(S)$-bounded too.
\end{proof}

By Propositions \ref{prp:reductioninfinity} and \ref{prp:RzNfactorisation}, the problem of the $L^p$-boundedness for $p \in (2,\infty)$ of the Riesz transforms $\Rz_j$ is thus reduced to that of the $L^p$-boundedness for $p \in (1,2)$ of the convolution operators by the kernels $K_0,K_{\lie{v}},K_{\lie{z}}$ from \eqref{eq:K0} and \eqref{eq:Kbullet}. 

Crucially, the functions $r_0,r_\lie{v},r_\lie{z}$ on $N$ appearing in the definition of $K_0,K_{\lie{v}},K_{\lie{z}}$ are radial in the first-layer variable $x$; so, as we shall see in Section \ref{s:sectionmultiplierapproach}, the convolution operators with kernels $r_0,r_\lie{v},r_\lie{z}$ on $N$ are in the joint functional calculus for the sub-Laplacian $\opLN$ and the central derivatives $-i\nabla_z$ on $N$.

This is essentially what will allow us to consider the convolution operators by $K_0,K_{\lie{v}},K_{\lie{z}}$ on $S$ as belonging to an operator-valued joint functional calculus for $(\opLN,-i \nabla_z)$, and eventually deduce their $L^p$-boundedness from a suitable operator-valued spectral multiplier theorem, which we prove in Section \ref{s:sectionmultipliertheorem}.

\section{An operator-valued multiplier theorem for the joint functional calculus of a system of commuting self-adjoint operators}\label{s:sectionmultipliertheorem}

This section is devoted to proving a multivariable analogue to \cite[Theorem 7.7]{MP24}.
Roughly speaking, we are going to prove that, if for a pairwise strongly commuting system of self-adjoint operators we have a scalar-valued spectral multiplier theorem of Mihlin--H\"ormander type on $L^p$ for some $p \in (1,\infty)$, then we also have an analogous operator-valued spectral multiplier theorem.

The present section is essentially self-contained and independent of the rest of the paper. In the next section, we shall apply the operator-valued spectral multiplier theorem proved here to the system $(\opLN,-i\nabla_z)$ made of the sub-Laplacian and the central derivatives on the H-type group $N$. Our proof of the multiplier theorem, however, applies more generally to a system $(L_1,\dots,L_{\dimX})$ of commuting self-adjoint operators on $L^2(X)$ for some measure space $X$, and we present the result in this greater generality.

Central to our discussion is the notion of R-boundedness for a family of operators, whose definition we recall here for the reader's convenience, referring to \cite[Definition 8.1.1(1) and Remark 8.1.2]{HvNVW17} for additional details.

\begin{definition}
Let $Z$ be a Banach space and let $\cT \subseteq \LinBdd(Z)$. We say that the family of operators $\cT$ is \emph{R-bounded} if there exists a constant $C>0$ such that for any finite sequences $(T_{m})_{m=1}^{\bar m}$ in $\cT$ and $(z_{m})_{m=1}^{\bar m}$ in $Z$, 
\[
\int_{0}^{1} \left\| \sum_{m=1}^{\bar m} \varepsilon_{m}(t) T_{m} z_{m} \right\|_{Z} \,\dd t\leq C\int_{0}^{1} \left\| \sum_{m=1}^{\bar m} \varepsilon_{m}(t) z_{m} \right\|_{Z} \,\dd t;
\]
here $(\varepsilon_{m})_m$ is an independent sequence of Rademacher functions on $[0,1]$. The smallest constant $C$ for which the above inequality holds is denoted by $\RBd_{Z}(\cT)$.
\end{definition}

We recall the definition of the local scale-invariant Sobolev norm on $\RR^{\dimX}$ that we shall use in our discussion.
Namely, for $F : \RR^{\dimX} \to \CC$, $s \geq 0$, $q \in [1,\infty]$, we set
\begin{equation}\label{eq:sloc}
\left\|F \right\|_{L^{q}_{s,\sloc}} \defeq \sup_{t>0}\left\| F(t\cdot) \,\phi\right\|_{L^{q}_{s}(\RR^{\dimX})}, 
\end{equation}
where $L^{q}_{s}(\RR^{\dimX})$ is the $L^{q}$-Sobolev space of order $s$, while $\phi\in C_{c}^{\infty}(\dot\RR^{\dimX})$ is a nonnegative cutoff such that $\bigcup_{t>0} t \cdot \{x \in \RR^{\dimX} \tc \phi(x) \neq 0\} = \dot\RR^{\dimX}$.
It is easily seen that different choices of the cutoff $\phi$ lead to equivalent norms, and moreover
\[
\|F\|_{L^q_{s,\sloc}} \lesssim_{q,s,k} \max_{|\alpha| \leq k} \sup_{\xi \in \dot\RR^\dimX} |\partial_\xi^\alpha F(\xi)| \lesssim_{q,s',k} \|F\|_{L^q_{s',\sloc}}
\]
for any $k\in \NN$ and $s,s' \in [0,\infty)$ with $s \leq k < s' - \dimX/q$; for additional details, see, e.g., \cite[Section 2.4]{M10}.

We can now state our assumption on the system of operators $(L_1,\dots,L_{\dimX})$.

\begin{assumption}\label{ass:fctcal}
$\vec L \defeq (L_{1},\dots,L_{\dimX})$ is pairwise strongly commuting system of self-adjoint operators on $L^{2}(X)$, where $(X,\mu)$ is a $\sigma$-finite measure space. Moreover, for some $p \in (1,\infty)$ and $\thr \in [0,\infty)$, the system of operators $L$ admits a bounded $L_{s,\sloc}^{\infty}$-functional calculus on $L^{p}(X)$ for all $s>\thr$: in other words, for all $s > \thr$ and for all bounded Borel functions $F : \RR^{\dimX} \to \CC$ such that $\|F\|_{L_{s,\sloc}^\infty} < \infty$, the operator $F(\vec L) = F(L_1,\dots,L_{\dimX})$, initially defined on $L^2(X)$, extends to a bounded operator on $L^p(X)$, and
\begin{equation}\label{eq:fctcal}
\|F(\vec L)\|_{\LinBdd(L^{p}(X))} \lesssim_{\vec L,p,s} \left\| F\right\|_{L^{\infty}_{s,\sloc}}.
\end{equation}
\end{assumption}

By applying \eqref{eq:fctcal} with $F = \chr_{\{0\}}$, we see that, under Assumption \ref{ass:fctcal}, the operators $L_1,\dots,L_{\dimX}$ are jointly injective, i.e., the spectral measure of $\{0\}$ in the joint spectrum of $L_1,\dots,L_{\dimX}$ is zero. In particular, the operator $F(\vec L)$ only depends on the restriction of $F : \RR^{\dimX} \to \CC$ to $\dot\RR^{\dimX}$.

The existence of a $L^\infty_{s,\sloc}$-bounded functional calculus on $L^p$ for a system $\vec L$ of commuting self-adjoint operators on $L^2$ can often be deduced as a consequence of suitable ``heat kernel bounds''; we refer, e.g., to \cite[Section 6]{M17} for a more detailed discussion with examples and further references to the literature. For the clarity of the presentation, the definition in \eqref{eq:sloc} of the scale-invariant Sobolev norm used here involves isotropic dilations on $\dot\RR^{\dimX}$; however, one could consider, as in \cite[Theorem 6.1]{M17}, more general nonisotropic dilations.

We now want to make sense of an operator-valued functional calculus for $\vec L = (L_1,\dots,L_{\dimX})$.
To this purpose, we recall from \cite[Definition 7.4]{MP24} some definitions about operator-valued functions (in \cite{MP24} they are presented for $\dimX=1$ only, but the multivariate extension is straightforward).

Throughout this section, we assume that $(Y,\nu)$ is a $\sigma$-finite measure space which is \emph{separable}, i.e., such that $L^1(Y)$ (equivalently, any $L^p(Y)$ with $p \in [1,\infty)$) is a separable Banach space.

\begin{definition}
Let $U \subseteq \dot\RR^{\dimX}$ be open and let $M: U \to \LinBdd(L^{2}(Y))$.
\begin{enumerate}[label=(\alph*)]
\item We say that $M$ is measurable function if the ``matrix coefficients''
\begin{equation}\label{eq:coefficient}
\xi \mapsto \left\langle M(\xi)f, g \right\rangle
\end{equation}
are measurable on $U$ for all $f,g\in L^{2}(Y)$. In this case, as $L^2(Y)$ is separable, the function 
\[
\xi \mapsto \left\| M(\xi) \right\|_{\LinBdd(L^{2}(Y))}
\]
is also measurable.
\item We say that $M$ is of class $C^{k}$ for $k\in\NN$, if the ``matrix coefficients'' \eqref{eq:coefficient} are of class $C^{k}$ on $U$ for all $f,g\in L^{2}(Y)$. In this case, from the Banach--Steinhaus theorem, one can deduce that, for any $\xi \in U$ and $|\alpha| \leq k$, there exists operators in $\LinBdd(L^{2}(Y))$, denoted by $\partial_{\xi}^{\alpha} M(\xi)$, such that 
\[
\langle\partial_\xi^\alpha M(\xi)f,g\rangle=\partial_\xi^\alpha \langle M(\xi)f,g\rangle \quad\forall f,g\in L^2(Y).
\]
\end{enumerate}
\end{definition}

Following \cite[Section 7]{MP24}, we now define the operator-valued functional calculus for a system $\vec L = (L_1,\dots,L_{\dimX})$ of operators satisfying Assumption \ref{ass:fctcal}.
By the spectral theorem, there exist a measure space $\Omega$, a measurable function $\ell : \Omega \to \dot\RR^{\dimX}$, whose image is contained in the joint spectrum of $L_1,\dots,L_{\dimX}$, and a unitary operator $\Upsilon : L^{2}(X) \to L^{2}(\Omega)$, such that for any bounded Borel function $F:\dot\RR^{\dimX} \to \CC$,
\[
\Upsilon(F(\vec L)f)(\omega) = F(\ell(\omega)) (\Upsilon f)(\omega), \quad f \in L^{2}(X).
\]
We can now make sense of $M(\vec L)$ as a bounded operator on $L^2(X \times Y)$ for any measurable and bounded operator-valued function $M : \dot\RR^{\dimX} \to \LinBdd(L^{2}(Y))$, where $Y$ is a separable $\sigma$-finite measure space. 
Namely, by means of the unitary operator $\widetilde{\Upsilon} \defeq \Upsilon \otimes \id : L^{2}(X\times Y) \to L^{2}(\Omega\times Y)$, we define
\begin{equation}\label{eq:ovfc}
\widetilde{\Upsilon}(M(\vec L)f)(\omega,y) = \left(M(\ell(\omega)) (\widetilde{\Upsilon}f)(\omega,\cdot)\right)(y),\quad f \in L^{2}(X\times Y).
\end{equation}
We refer to \cite[Remark 7.6]{MP24} for additional details on the above definition. We remark in particular that, in the case $M(\xi) = F(\xi) \, \id$ for some scalar-valued $F$, we have $M(\vec L) = F(\vec L) \otimes \id$; in this sense, the operator-valued calculus extends the scalar-valued one.

Under Assumption \ref{ass:fctcal}, we know that, for a bounded scalar-valued function $F$ on $\dot\RR^{\dimX}$, the operator $F(\vec L)$, initially defined on $L^2$, extends to a bounded operator on $L^p$ under suitable smoothness conditions on $F$. The purpose of this section is establishing an analogous result, Theorem \ref{thm:theoremmultiplier} below, for the operators $M(\vec L)$ corresponding to operator-valued functions $M$. With a slight abuse of notation, we shall not distinguish between an $L^2$-bounded operator and its $L^p$-bounded extension, when the latter exists.

The proof of Theorem \ref{thm:theoremmultiplier} will make fundamental use of the two following lemmas.
The first one is a Littlewood--Paley decomposition adapted to the joint functional calculus for $L_1,\dots,L_{\dimX}$; while the result below is likely well known to experts (cf., e.g., \cite{KW16} for the case of a single operator, or \cite[Lemmas 5.5 and 5.6]{M12} for related results for multiple operators), we could not find a reference that we could directly cite, so we include a proof for the reader's convenience.

\begin{lemma}\label{lem:lemma1}
Under Assumption \ref{ass:fctcal}, if $\eta\in C_{c}^{\infty}(\dot\RR^{\dimX})$ is nonnegative and satisfies
\[
\sum_{m\in\ZZ} \eta(2^{m}\xi) \simeq 1
\]
for $\xi \in \dot\RR^{\dimX}$, then there holds 
\begin{equation*}
\left\|f\right\|_{L^{p}(X\times Y)}
\simeq_{\vec L,\eta,p} \int_{0}^{1} \left\|\sum_{m\in\ZZ}\varepsilon_{m}(t) (\eta(2^{m} \vec L) \otimes \id) f\right\|_{L^{p}(X\times Y)} \,\dd t.
\end{equation*}
for all $f\in L^{p}(X\times Y)$, where $(\varepsilon_{m})_{m\in\ZZ}$ is an independent sequence of Rademacher functions on $[0,1]$.
\end{lemma}
\begin{proof}
As $p \in (1,\infty)$, \cite[Theorem 3.2.23]{HvNVW16} and Fubini's theorem give that 
\[\begin{split}
&\int_{0}^{1} \left\| \sum_{m\in\ZZ} \varepsilon_{m}(t)(\eta(2^{m} \vec L) \otimes \id)f \right\|_{L^{p}(X\times Y)} \,\dd t\\
&\simeq_p \left(\int_{0}^{1} \left\| \sum_{m\in\ZZ} \varepsilon_{m}(t) (\eta(2^{m} \vec L) \otimes \id)f \right\|^{p}_{L^{p}(X \times Y)} \,\dd t \right)^{\frac{1}{p}} \\
&= \left(\int_{Y} \int_{0}^{1} \left\| \sum_{m\in\ZZ} \varepsilon_{m}(t) \eta(2^{m} \vec L)f(\cdot,y) \right\|^{p}_{L^{p}(X)} \,\dd t \,\dd y \right)^{\frac{1}{p}}.
\end{split}\]
Similarly,
\begin{equation*}
\left\|f\right\|_{L^{p}(X\times Y)}
= \left(\int_{Y}\left\|f(\cdot,y)\right\|^{p}_{L^{p}(X)} \,\dd y \right)^{\frac{1}{p}}.
\end{equation*}
So the estimate for arbitrary $Y$ can be reduced to the particular case where $Y=\CC$; namely, we need to prove that, for $f\in L^{p}(X)$, there holds 
\begin{equation}\label{eq:reduction2}
\left\|f\right\|_{L^{p}(X)}
\simeq_{\vec L,\eta,p} \left( \int_{0}^{1} \left\| \sum_{m\in\ZZ} \varepsilon_{m}(t) \eta(2^{m} \vec L) f\right\|^{p}_{L^{p}(X)} \,\dd t \right)^{\frac{1}{p}}.
\end{equation} 

The inequality $\gtrsim$ in \eqref{eq:reduction2} follows immediately from the fact that
\begin{equation*}
\sup_{t \in [0,1]} \left\| \sum_{m\in\ZZ}\varepsilon_{m}(t) \eta(2^{m} \vec L) \right\|_{\LinBdd(L^{p}(X))}
\lesssim_{\vec L,\eta,p} 1,
\end{equation*}
which in turn follows from Assumption \ref{ass:fctcal} and the fact that 
\[
\sup_{t \in [0,1]} \left\| \sum_{m \in \ZZ} \varepsilon_m(t) \eta(2^{m} \cdot) \right\|_{L^\infty_{s,\sloc}} < \infty
\]
for any $s \geq 0$.

To prove the reverse inequality $\lesssim$ in \eqref{eq:reduction2}, we notice that, under our assumptions, we have $\sum_{m \in \ZZ} \eta^2(2^m \xi) \simeq 1$ for all $\xi \in \dot\RR^{\dimX}$, whence it readily follows that
\[
\left\| \frac{1}{\sum_{m \in \ZZ} \eta^2(2^{m} \cdot)} \right\|_{L^\infty_{s,\sloc}} < \infty,
\]
and therefore, by Assumption \ref{ass:fctcal},
\begin{equation}\label{eq:reductionfnorm}
\left\|f\right\|_{L^{p}(X)}
\lesssim_{\vec L,\eta,p} \left\| \sum_{m\in\ZZ} \eta^{2}(2^{m} \vec L) f \right\|_{L^{p}(X)}.
\end{equation}
Moreover, by Fubini's theorem and Khintchine's inequality \cite[Corollary 3.2.24]{HvNVW16},
\begin{equation}\label{eq:equivalence}
\begin{split}
&\left( \int_{0}^{1} \left\| \sum_{m\in\ZZ} \varepsilon_{m}(t) \eta(2^{m} \vec L)f \right\|^{p}_{L^{p}(X)} \,\dd t \right) ^{\frac{1}{p}}\\
&= \left\| \left( \int_{0}^{1} \left| \sum_{m\in\ZZ} \varepsilon_{m}(t) \eta(2^{m} \vec L)f \right|^{p} \,\dd t \right)^{\frac{1}{p}}\right\| _{L^{p}(X)}\\
&\simeq_p \left\| \left( \sum_{m\in\ZZ} \left| \eta(2^{m} \vec L)f \right|^{2} \right)^{\frac{1}{2}} \right\|_{L^{p}(X)}
= \left\| \left( \eta(2^{m} \vec L)f \right)_{m\in\ZZ} \right\|_{L^{p}(X;\ell^{2})}.
\end{split}
\end{equation} 
Combining with \eqref{eq:reductionfnorm}, we are thus reduced to proving that
\begin{equation*}
\left\|\sum_{m\in\ZZ}\eta^{2}(2^{m} \vec L) f\right\|_{L^{p}(X)}
\lesssim_{\vec L,\eta,p}
\left\| \left( \eta(2^{m} \vec L)f\right)_{m\in\ZZ} \right\|_{L^{p}(X;\ell^{2})}
\end{equation*}
for all $f \in L^p(X)$, which is a particular case of the estimate
\begin{equation}\label{eq:reduction3}
\left\|\sum_{m\in\ZZ} \eta(2^{m} \vec L) f_{m} \right\|_{L^{p}(X)}
\lesssim_{\vec L,\eta,p}
\left\|(f_{m})_{m\in\ZZ} \right\|_{L^{p}(X;\ell^{2})}
\end{equation}
for all $(f_m)_{m \in \ZZ} \in L^p(X;\ell^2)$.
By duality, \eqref{eq:reduction3} reduces to showing that 
\begin{equation}\label{eq:reduction4}
\left\| \left( \eta(2^{m} \vec L)f \right)_{m\in\ZZ} \right\|_{L^{p'}(X;\ell^{2})}
\lesssim_{\eta} \|f\|_{L^{p'}(X)}
\end{equation}
for all $f \in L^{p'}(X)$, where $p'$ is the conjugate exponent to $p$.

We now observe that, as $L_1,\dots,L_{\dimX}$ are self-adjoint, for any $F : \dot\RR^{\dimX} \to \CC$,
\[
\|F(\vec L) \|_{\LinBdd(L^{p'}(X))}
= \|\overline{F}(\vec L) \|_{\LinBdd(L^{p}(X))}, 
\]
while $\|F\|_{L^\infty_{s,\sloc}} = \|\overline{F}\|_{L^\infty_{s,\sloc}}$; thus, $\vec L$ also satisfies the analogue of Assumption \ref{ass:fctcal} with $p'$ in place of $p$, and therefore all the previous estimates hold for the conjugate exponent $p'$ too. In particular, by applying \eqref{eq:equivalence} with $p'$ in place of $p$, we see that \eqref{eq:reduction4} reduces to
\begin{equation*}
\left( \int_{0}^{1} \left\| \sum_{m\in\ZZ} \varepsilon_{m}(t) \eta(2^{m} \vec L)f \right\|^{p'}_{L^{p'}(X)} \,\dd t \right)^{\frac{1}{p'}} 
\lesssim_{\vec L,\eta,p} \| f \|_{L^{p'}(X)},
\end{equation*} 
i.e., the $p'$ version of the already established inequality $\gtrsim$ in \eqref{eq:reduction2}.
\end{proof}

The second one concerns the R-boundedness of a family of operators in the joint functional calculus for $L_1,\dots,L_{\dimX}$. The proof is inspired, among others, by some ideas from \cite[proof of Theorem 6.1]{KW18}.

\begin{lemma}\label{lem:lemma2}
Let $\chi\in C_{c}^{\infty}(\dot\RR^{\dimX})$.
Set $E_{k}(\xi) \defeq \chi(\xi) \,\ee^{ik\cdot\xi}$ and $E_{k,m}(\xi) \defeq E_k(2^m \xi)$ for $m \in \ZZ$, $k\in\ZZ^{\dimX}$. 
Under Assumption \ref{ass:fctcal}, if $\sigma>\dimX+1+\thr$, then
\[
\RBd_{L^{p}(X)}\{(1+|k|)^{-\sigma} E_{k,m}(\vec L) \tc m\in\ZZ, \, k\in\ZZ^{\dimX}\}<\infty.
\]
\end{lemma}
\begin{proof}
Let $S^{\dimX-1}$ be the unit sphere in $\RR^{\dimX}$. If we define $H_k : S^{\dimX-1} \times \RR \to \CC$ by
\[
H_k(\omega,s) \defeq E_k(e^s \omega),
\]
then $H_k$ is smooth and compactly supported. By the Fourier inversion formula,
\[
H_{k}(\omega,s)=\frac{1}{2\pi}\int_{\RR}\hat{H}_{k}(\omega,u) \,\ee^{isu} \,\dd u,
\]
where $\hat H_k(\omega,\cdot)$ is the Fourier transform of $H_k(\omega,\cdot)$.
Therefore, for all $\xi \in \dot\RR^{\dimX}$,
\[
E_{k}(\xi)= H_k(\xi/|\xi|,\log |\xi|) =  \frac{1}{2\pi}\int_{\RR}\hat{H}_{k}(\xi/|\xi|,u) \,|\xi|^{iu} \,\dd u.
\]
Thus,
\[\begin{split}
(1+|k|)^{-\sigma} E_{k,m}(\vec L)
&=\frac{1}{2\pi} \int_{\RR} 2^{imu} \, (1+|k|)^{-\sigma} \hat{H}_{k}(\vec L/|\vec L|,u) \,|L|^{iu} \,\dd u\\
&=\int_{\ZZ^{\dimX}}\int_{\RR} \delta_{k\ell} \, 2^{imu} \, (1+|\ell|)^{-\sigma} F(\ell,u,\vec L) \,\frac{\dd u}{2\pi} \,\dd\#(\ell),
\end{split}\]
where $F(\ell,u,\xi) \defeq \hat{H}_{\ell}(\xi/|\xi|,u) \,|\xi|^{iu}$ and $\dd\#$ is the counting measure on $\ZZ^{\dimX}$. 
By applying \cite[Theorem 8.5.4]{HvNVW17} with $S= \left(\RR\times\ZZ^{\dimX},\frac{\dd u}{2\pi} \,\dd\#(\ell)\right)$, since 
\[
\left\| (u,\ell) \mapsto \delta_{k\ell} \, 2^{imu}\right\|_{L^{\infty}(S)}\leq 1 \qquad\forall k\in\ZZ^{\dimX}, \,m\in\ZZ,
\]
we can conclude that 
\begin{equation}\label{eq:reduction6}
\begin{split}
&\RBd_{L^{p}(X)}\{(1+|k|)^{-\sigma} E_{k,m}(L) \tc m\in\ZZ,\, k\in\ZZ^{\dimX}\} \\
&\lesssim_{p}\int_{\ZZ^{\dimX}}\int_{\RR}(1+|\ell|)^{-\sigma} \| F(\ell,u,\vec L)\|_{\LinBdd(L^{p}(X))} \,\frac{\dd u}{2\pi} \,\dd\#(\ell)\\
&\lesssim_{\vec L,p,s}\int_{\ZZ^{\dimX}}\int_{\RR}(1+|\ell|)^{-\sigma} \| F(\ell,u,\cdot) \|_{L_{s,\sloc}^{\infty}(\RR^{\dimX})} \,\dd u \,\dd\#(\ell)
\end{split}
\end{equation}
for any $s > \thr$, where the last inequality follows from Assumption \ref{ass:fctcal}.

In order to show that the last integral in \eqref{eq:reduction6} is finite, it remains to estimate $\| F(\ell,u,\cdot) \|_{L_{s,\sloc}^{\infty}(\RR^{\dimX})}$. By the Leibniz rule,
\begin{equation}\label{eq:Leibniz}
\partial_{\xi}^{\alpha} F(\ell,u,\xi)=\sum_{\alpha'+\alpha''=\alpha}C_{\alpha',\alpha''}\left(\partial_{\xi}^{\alpha'} \hat{H}_{\ell}(\xi/|\xi|,u)\right) \left(\partial_{\xi}^{\alpha''}|\xi|^{iu}\right)
\end{equation}
for some $C_{\alpha',\alpha''} > 0$, and clearly
\begin{equation}\label{eq:easypart}
\left|\partial_{\xi}^{\alpha''}|\xi|^{iu}\right|\lesssim_{\alpha''}(1+|u|)^{|\alpha''|}|\xi|^{-|\alpha''|}.
\end{equation}
On the other hand,
\[
\partial_{\xi}^{\alpha'} \hat{H}_{\ell}(\xi/|\xi|,u)
=\int_{\RR}\partial_{\xi}^{\alpha'} ( E_{\ell}(\ee^{s} \xi/|\xi|) ) \, \ee^{-isu} \,\dd s,
\]
so, by integrations by parts, for any $r \in \NN$,
\[
(iu)^{r}\partial_{\xi}^{\alpha'} \hat{H}_{\ell}(\xi/|\xi|,u)
=\int_{\RR}\partial_{s}^{r}\partial_{\xi}^{\alpha'} ( E_{\ell}(\ee^{s} \xi/|\xi|) ) \,\ee^{-isu} \,\dd s.
\]
As $E_\ell(\zeta) = \chi(\zeta) \ee^{i \ell \cdot \zeta}$ and due to the support of $\chi$,
\[
\left| |u|^{r}\partial_{\xi}^{\alpha'} \hat{H}_{\ell}(\xi/|\xi|,u)\right|
\lesssim \sup_{s\lesssim1} \left|\partial_{s}^{r}\partial_{\xi}^{\alpha'} ( E_{\ell}(\ee^{s} \xi/|\xi|) ) \right| \lesssim_{\alpha',r}|\xi|^{-|\alpha'|}(1+|\ell|)^{r+|\alpha'|},
\]
whence also, for any $r \in \NN$,
\begin{equation*}
\left| \partial_{\xi}^{\alpha'} \hat{H}_{\ell}(\xi/|\xi|,u) \right| 
\lesssim_{\alpha',r} |\xi|^{-|\alpha'|}(1+|\ell|)^{r+|\alpha'|}(1+|u|)^{-r}.
\end{equation*}
If we apply the last estimate with $r+|\alpha''|$ in place of $r$, and combine it with \eqref{eq:Leibniz} and \eqref{eq:easypart}, we conclude that 
\[
|\xi|^{|\alpha|} \left|\partial_{\xi}^{\alpha} F(\ell,u,\xi) \right| \lesssim_{\alpha,r} (1+|\ell|)^{r+|\alpha|}(1+|u|)^{-r}
\]
for all $\alpha\in\NN^{\dimX}$ and $r\in\NN$. By interpolation, we thus deduce that
\begin{equation}\label{eq:sslocnorm}
\left\| F(\ell,u,\cdot) \right\|_{L_{s,\sloc}^{\infty}(\RR^{\dimX})}
\lesssim_{s,r}(1+|\ell|)^{r+s} (1+|u|)^{-r},
\end{equation} 
for all $s,r\in [0,\infty)$.

By \eqref{eq:sslocnorm}, the last integral in \eqref{eq:reduction6} is bounded by a multiple of
\[
\sum_{\ell\in\ZZ^{\dimX}} \int_{\RR} (1+|k|)^{-\sigma+r+s} (1+|u|)^{-r} \,\dd u,
\]
which is finite whenever $r>1$ and $\sigma>\dimX+r+s$. In particular, for any $\sigma>\dimX+1+\thr$, we can choose appropriate values of $r>1$ and $s>\thr$ and deduce from \eqref{eq:reduction6} the claimed R-boundedness.
\end{proof}

Thanks to Lemma \ref{lem:lemma1} and Lemma \ref{lem:lemma2}, we can now follow closely the proof of \cite[Theorem 7.1]{MP24} to obtain the following result.

\begin{theorem}\label{thm:theoremmultiplier}
Under Assumption \ref{ass:fctcal}, fix an integer $B>2\dimX+1+\thr$. Let $M:\dot\RR^{\dimX} \to \LinBdd(L^{2}(Y))$ be measurable and bounded.
Assume that $M$ is of class $C^{B}$ and 
\begin{equation}\label{eq:assumption1}
C_{p,B}(M) 
\defeq \RBd_{L^{p}(Y)}\{ |\xi|^{|\alpha|} \partial^{\alpha}_{\xi}M(\xi) \tc |\alpha|\leq B, \ \xi\in\dot\RR^{\dimX} \}<\infty.		
\end{equation}
Then, $M(\vec L)$ is a bounded operator on $L^{p}(X\times Y)$ with 
\begin{equation*}
\| M(\vec L) \|_{\LinBdd(L^{p}(X\times Y))} \lesssim_{B,p,\vec L} C_{p,B}(M).
\end{equation*}
\end{theorem}
\begin{proof}
Fix a nonnegative cutoff $\eta \in C^\infty_c(\dot\RR^{\dimX})$ such that $\supp \eta \subseteq \{ \xi \tc 1/2 \leq |\xi| \leq 2 \}$ and $\sum_{m \in \ZZ} \eta(2^m \cdot) = 1$ on $\dot\RR^{\dimX}$.
If we set 
\begin{equation}\label{eq:Mm}
M_{m}(\xi) \defeq M(2^{-m}\xi)\eta(\xi),
\end{equation}
then we can decompose 
\[
M(\xi)=\sum_{m\in\ZZ}M_{m}(2^{m}\xi), \quad\xi\in\dot\RR^{\dimX},
\]
with convergence in the weak operator topology of $\LinBdd(L^{2}(Y))$.
Following the proof of \cite[Theorem 7.7]{MP24}, using Lemma \ref{lem:lemma1} in place of \cite[Lemma 7.3]{MP24}, we obtain that
\[
\|M(\vec L)\|_{\LinBdd(L^{p}(X\times Y))}
\lesssim_{p} \RBd_{L^{p}(X\times Y)}\{ M_{m}(2^{m}\vec L) \}_{m\in\ZZ}.
\]

To prove the R-boundedness of $\{M_{m}(2^{m} \vec L)\}_{m\in\ZZ}$ on $L^{p}(X\times Y)$, 
we introduce the Fourier coefficients of the functions $\xi \mapsto M_m(\xi)$, which are supported in $\supp \eta \subseteq (-\pi,\pi)^{\dimX}$ by construction.
Namely, define $\hat{M}_{m}(k)\in \LinBdd(L^{2}(Y))$ for $k\in\ZZ^{\dimX}$ by
\[
\hat{M}_{m}(k) \defeq \frac{1}{(2\pi)^{\dimX}}\int_{-\pi}^{\pi}\dots\int_{-\pi}^{\pi}M_{m}(\xi) \,\ee^{-ik\cdot\xi} \,\dd\xi.
\]
Then, for $k\in\ZZ^{\dimX}$, integration by parts shows that
\begin{equation}\label{eq:coefficient1}
(ik)^{\alpha}\hat{M}_{m}(k)=\frac{1}{(2\pi)^{\dimX}}\int_{-\pi}^{\pi}\dots\int_{-\pi}^{\pi}\partial^{\alpha}_{\xi}M_{m}(\xi) \,\ee^{-ik\cdot\xi} \,\dd\xi.
\end{equation}
Moreover, if $\chi \in C^\infty_c(\dot\RR^{\dimX})$ is supported in $(-\pi,\pi)^{\dimX}$ and such that $\chi = 1$ on $\supp \eta$, then we can write
\[
M_{m}(\xi)=\sum_{k\in\ZZ^{\dimX}}\chi(\xi)\ee^{ik\cdot\xi}\hat{M}_{m}(k),\quad\xi\in\dot\RR^{\dimX},
\]
with convergence in the weak operator topology of $\LinBdd(L^{2}(Y))$; therefore, if the functions $E_{k,m}$ are defined as in Lemma \ref{lem:lemma2}, we obtain that
\begin{equation*}
M_{m}(2^{m}\vec L)
= \sum_{k\in\ZZ^{\dimX}} E_{k,m}(\vec L) \otimes \hat{M}_{m}(k)
\end{equation*}
in the weak operator topology of $\LinBdd(L^{2}(X\times Y))$. 
Much as in the proof of \cite[Theorem 7.7]{MP24},
from this decomposition we deduce that
\[\begin{split}
&\RBd_{L^{p}(X\times Y)}\{M_{m}(2^{m} \vec L)\}_{m\in\ZZ}\\
&\lesssim_\varepsilon \RBd_{L^{p}(X)}\{(1+|k|)^{-\sigma} E_{k,m}(\vec L)\}_{m\in\ZZ,k\in\ZZ^{\dimX}}
\, \RBd_{L^{p}(Y)}\{(1+|k|)^{B}\hat{M}_{m}(k)\}_{m\in\ZZ, k\in\ZZ^{\dimX}},
\end{split}\]
whenever $B-\sigma = \dimX+\varepsilon$ for some $\varepsilon > 0$ (here we are exploiting \cite[Propositions 8.1.24 and 8.1.19(3)]{HvNVW17} and the fact that $\sum_{k \in \ZZ^{\dimX}} (1+|k|)^{-\dimX-\varepsilon} < \infty$). As $B > \thr+2\dimX+1$ by assumption, we can take $\varepsilon>0$ such that $B=\thr+2\dimX+1+2\varepsilon$, and correspondingly $\sigma=\thr+\dimX+1+\varepsilon$.

With this choice of $\sigma$, the R-boundedness of $\{(1+|k|)^{-\sigma} E_{k,m}(\vec L)\}_{m\in\ZZ, k\in\ZZ^{\dimX}}$ on $L^p(X)$ follows from Lemma \ref{lem:lemma1}. To conclude, it only remains to show the R-boundedness of $\{(1+|k|)^{B}\hat{M}_{m}(k)\}_{m\in\ZZ,k\in\ZZ^{\dimX}}$ on $L^p(Y)$, which is a consequence of the smoothness assumption \eqref{eq:assumption1} on $M$.

Indeed, arguing as in the proof of \cite[Theorem 7.7]{MP24}, from \eqref{eq:coefficient1} we deduce that
\[
\RBd_{L^{p}(Y)}\{|k^{\alpha}| \hat{M}_{m}(k) \tc m\in\ZZ, \, k\in\ZZ^{\dimX}\}
\leq \RBd_{L^{p}(Y)}\{\partial_{\xi}^{\alpha} M_{m}(\xi) \tc \xi\in\dot\RR^{\dimX}, \, m\in\ZZ\}
\]
for all $\alpha \in \NN^{\dimX}$, and therefore also
\[\begin{split}
&\RBd_{L^{p}(Y)} \{(1+|k|)^B \hat{M}_{m}(k)\tc m\in\ZZ,\,k\in\ZZ^{\dimX}\} \\
&\lesssim_B \RBd_{L^{p}(Y)} \{\partial_{\xi}^{\alpha} M_{m}(\xi) \tc \xi\in\dot\RR^{\dimX}, \, m\in\ZZ, \, |\alpha|\leq B\} \\
&\lesssim_{\eta,B} \RBd_{L^{p}(Y)} \{|\xi|^{|\beta|}\partial_{\xi}^{\beta}M(\xi) \tc |\beta|\leq B, \, \xi\in\dot\RR^{\dimX}\} = C_{p,B}(M),
\end{split}\]
where the last inequality follows from the fact that, by \eqref{eq:Mm},
\[
\partial_\xi^\alpha M_m(\xi)=\sum_{\beta\leq \alpha}\binom{\alpha}{\beta}\frac{\partial^{(\alpha-\beta)}_{\xi}\eta(\xi)}{|\xi|^{|\beta|}}\left( |\xi|^{|\beta|}\partial_\xi^{\beta} M\right)(2^{-m}\xi).
\]
and by the Kahane contraction principle \cite[Proposition 3.2.10]{HvNVW16}. As $C_{p,B}(M)$ is finite by \eqref{eq:assumption1}, we are done.
\end{proof}

In the particular case where $Y=\RR^{\dimY}$ with the Lebesgue measure for some $\dimY \in \Npos$, the smoothness condition \eqref{eq:assumption1} on the multiplier $M$ in Theorem \ref{thm:theoremmultiplier}, which is expressed in terms of R-boundedness, can be replaced with a somewhat more concrete condition involving uniform weighted $L^2$-boundedness with respect to weights in the Muckenhoupt class $A_2(\RR^{\dimY})$. Let $[w]_{A_{2}}$ denote the $A_{2}$-characteristic of a weight $w \in A_{2}(\RR^{\dimY})$. The proof of the following corollary, based on \cite[Theorem 8.2.6]{HvNVW17}, is analogous to that of \cite[Corollary 7.8]{MP24}.

\begin{corollary}\label{cor:corollary}
Under Assumption \ref{ass:fctcal}, fix an integer $B>2\dimX+1+\thr$. Let $M:\dot\RR^{\dimX} \to \LinBdd(L^{2}(\RR^{\dimY}))$ be measurable and bounded.
Assume that $M$ is of class $C^{B}$ and there exists a nondecreasing function $\psi:[1,\infty)\to[0,\infty)$ such that, for all $w\in A_{2}(\RR^{\dimY})$,
\begin{equation}\label{eq:assumption_A2}
\max_{\alpha \tc |\alpha| \leq B} \sup_{\xi \in \dot\RR^{\dimX}} \left\||\xi|^{|\alpha|} \partial_{\xi}^{\alpha} M(\xi)\right\|_{\LinBdd(L^{2}(w))}
\leq \psi([w]_{A_{2}}).
\end{equation}
Then, $M(\vec L)$ is a bounded operator on $L^{p}(X\times\RR^{\dimY})$.
\end{corollary}

The following remark discusses a weakening of the smoothness assumption on the multiplier, in case additional information is available on the joint spectrum of $(L_1,\dots,L_{\dimX})$. Here a subset $\Gamma$ of $\dot\RR^{\dimX}$ is said to be \emph{conic} if $t \cdot \Gamma \subseteq \Gamma$ for all $t>0$.

\begin{remark}\label{rem:conicspectrum}
Under Assumption \ref{ass:fctcal}, assume further that the joint $L^2$-spectrum of $(L_1,\dots,L_{\dimX})$ is contained in $\{0\} \cup \Gamma$ for some closed conic subset $\Gamma$ of $\dot\RR^{\dimX}$. Then, the smoothness conditions \eqref{eq:assumption1} and \eqref{eq:assumption_A2} can be replaced by analogous conditions, where $\xi$ is restricted to some conic neighbourhood $U \subseteq \dot\RR^{\dimX}$ of $\Gamma$. Indeed, if the operator-valued function $M : U \to \LinBdd(L^2(Y))$ satisfies such a restricted smoothness assumption, then
\[
M(\vec L) = \tilde M(\vec L),
\]
where $\tilde M(\xi) \defeq \zeta(\xi) M(\xi)$, while $\zeta \in C^\infty(\dot\RR^{\dimX})$ is any $0$-homogeneous cutoff with $\supp \zeta \subseteq U$ and $\zeta|_{\Gamma} \equiv 1$. One can then apply Theorem \ref{thm:theoremmultiplier} or Corollary \ref{cor:corollary} to $\tilde M$ in order to deduce the $L^p$-boundedness of $M(\vec L)$.
\end{remark}

\section{Boundedness of the parts at infinity}\label{s:sectionmultiplierapproach}

\subsection{Convolution operators as operator-valued spectral multipliers}\label{ss:convKopvalmult}

As discussed in Section \ref{s:sectionlocalpart}, the proof of Theorem \ref{thm:maintheorem} is now reduced to proving the $L^p(S)$-boundedness for $p \in (1,2)$ of the convolution operators with the kernels $K_0,K_{\lie{v}},K_{\lie{z}}$ from \eqref{eq:K0} and \eqref{eq:Kbullet}.
To this purpose, we shall write those operators as $M_0(\opLN,-i\nabla_z),M_{\lie{v}}(\opLN,-i\nabla_z),M_{\lie{z}}(\opLN,-i\nabla_z)$ for suitable operator-valued functions $M_0,M_{\lie{v}},M_{\lie{z}}$, and deduce their $L^p$-boundedness from the operator-valued spectral multiplier theorem of Section \ref{s:sectionmultipliertheorem}.

Let us first discuss how the theory of Section \ref{s:sectionmultipliertheorem} applies to the joint functional calculus for $(\opLN,-i\nabla_z)$, thus proving Theorem \ref{thm:opval_intro}.

Recall from Remark \ref{rem:productmeasure} that $S = N \times \Rpos$ as measure spaces, and that moreover the change of variables $a = \ee^u$ allows one to identify $\Rpos$ and $\RR$; in particular, we can define the class of weights $A_2(\Rpos)$ on $\Rpos$ by ``transplanting'' the classical Muckenhoupt class $A_2(\RR)$ via that change of variables.

In addition, if $\opLN$ and $\nabla_z$ denote, as before, the sub-Laplacian and the vector of central derivatives on $N$, then
\[
(\opLN,-i\nabla_z) = (\opLN,-i\partial_{z_1},\dots,-i\partial_{z_{\dimZ}})
\]
is a commuting system of self-adjoint operators on $L^2(N)$, whose joint spectrum is contained in 
\[
\{ (\lambda,\mu) \in \RR \times \RR^{\dimZ} \tc \lambda \geq \tfrac{\dimV}{2} |\mu|\},
\]
see \cite[Remark 2.1]{MRS96}.
Furthermore, by \cite[Corollary 2.4]{MRS96} (see also \cite[Propositions 3.9 and 4.1, and Theorem 4.6]{M12})
the system $(\opLN,-i\nabla_z)$ satisfies Assumption \ref{ass:fctcal} for any $p \in (1,\infty)$ with $\thr = \dimN/2$.

Thus, for any bounded measurable function $M : \Rpos \times \RR^{\dimZ} \to \LinBdd(L^2(\Rpos))$, we can consider the operator $M(\opLN,-i\nabla_z)$ defined by \eqref{eq:ovfc} as a bounded operator on $L^2(S)$; moreover, from Theorem \ref{thm:theoremmultiplier} and Remark \ref{rem:conicspectrum}, we deduce Theorem \ref{thm:opval_intro},
giving a sufficient condition for the $L^p(S)$-boundedness of such operators $M(\opLN,-i\nabla_z)$. We shall actually use the following version of the result, corresponding to Corollary \ref{cor:corollary}.

\begin{corollary}\label{cor:applytheorem}
Let $\kappa \in (0,\dimV/2)$ and $\Omega_\kappa \defeq \{ (\lambda,\mu) \in \Rpos \times \RR^{\dimZ} \tc \kappa |\mu| \leq \lambda \}$. Fix an integer $B > (\dimV+5\dimZ)/2+3$. If $M : \Rpos \times \RR^{\dimZ} \to \LinBdd(L^2(\Rpos))$ is of class $C^B$ and there exists a nondecreasing function $\psi : [1,\infty) \to [0,\infty)$ such that
\[
\max_{\alpha \tc |\alpha| \leq B} \sup_{(\lambda,\mu) \in \Omega_\kappa} \left\|\lambda^{|\alpha|} \partial_{(\lambda,\mu)}^{\alpha}M(\lambda,\mu)\right\|_{\LinBdd(L^{2}(w))}\leq \psi([w]_{A_{2}}) \qquad \forall w \in A_2(\Rpos),
\]
then $M(\opLN,-i\nabla_z)$ is bounded on $L^p(S)$ for all $p \in (1,\infty)$.
\end{corollary}

Now, from \eqref{eq:K0} and \eqref{eq:Kbullet} we have
\begin{gather*}
K_0(\bx) = [(r_0)_{(a)}(x,z) - r_0(x,z)] J_0(\log a), \\
K_{\lie{v}}(\bx) = r_{\lie{v}}(x,z) J_{\lie{n}}(\log a), \qquad
K_{\lie{z}}(\bx) = r_{\lie{z}}(x,z) J_{\lie{n}}(\log a)
\end{gather*}
for certain functions $r_0,r_{\lie{v}},r_{\lie{z}}$ on $N$ and $J_0,J_{\lie{n}}$ on $\RR$, where in particular
\begin{equation}\label{eq:J0Jbullet}
J_0(u) = \frac{\chr_{\{u\geq 1\}}}{u}, \qquad J_{\lie{n}}(u) = \frac{\chr_{\{u \leq -1\}}}{u}.
\end{equation}
As we shall see, one can write
\[
r_0 = \krn_{F_0(\opLN,-i\nabla_z)}, \qquad r_{\lie{v}} = \krn_{F_{\lie{v}}(\opLN,-i\nabla_z)}, \qquad r_{\lie{z}} = \krn_{F_{\lie{z}}(\opLN,-i\nabla_z)}
\]
for suitable bounded smooth functions $F_0,F_{\lie{v}},F_{\lie{z}} : \Rpos \times \RR^{\dimZ} \to \CC$ (see Propositions \ref{prp:r0rbulletF0Fbullet} and \ref{prp:theoremtildeHtildeestimate} below).
As a consequence, from \eqref{eq:convSN} and \eqref{eq:ovfc} we see that, for $\ell \in \{\lie{v},\lie{z}\}$, at least formally,
\[\begin{split}
(f * K_\ell)^{[\ee^u]} 
&= \int_{\RR} f^{[\ee^{u'}]} *_N (\krn_{F_\ell(\opLN,-i\nabla_z)})_{(\ee^{u'})} J_{\lie{n}}(u-u') \,\dd u' \\
&= \int_{\RR} F_\ell(\ee^{u'}\opLN,\ee^{u'}(-i\nabla_z))  f^{[\ee^{u'}]}  J_{\lie{n}}(u-u') \,\dd u' \\
&= (M_\ell(\opLN,-i\nabla_z) f)^{[\ee^u]},
\end{split}\]
where $M_\ell(\lambda,\mu)$ is the integral operator acting on functions $\phi : \Rpos \to \CC$ by
\begin{equation}\label{eq:Mbullet}
M_\ell(\lambda,\mu) \phi(\ee^u) = \int_\RR F_\ell(\ee^{u'} \lambda,\ee^{u'} \mu) J_{\lie{n}}(u-u') \phi(\ee^{u'}) \,\dd u'.
\end{equation}
In a similar way, we see that
\[\begin{split}
(f * K_0)^{[\ee^u]} 
&= \int_{\RR} f^{[\ee^{u'}]} *_N ((\krn_{F_0(\opLN,-i\nabla_z)})_{(\ee^{u})} - (\krn_{F_0(\opLN,-i\nabla_z)})_{(\ee^{u'})}) J_0(u-u') \,\dd u' \\
&= \int_{\RR} (F_0(\ee^{u}\opLN,\ee^{u}(-i\nabla_z)) - F_0(\ee^{u'}\opLN,\ee^{u'}(-i\nabla_z)))  f^{[\ee^{u'}]}  J_0(u-u') \,\dd u' \\
&= (M_0(\opLN,-i\nabla_z) f)^{[\ee^u]},
\end{split}\]
where $M_0(\lambda,\mu)$ is the integral operator acting on functions $\phi : \Rpos \to \CC$ by
\begin{equation}\label{eq:M0}
M_0(\lambda,\mu) \phi(\ee^u) = \int_\RR (F_0(\ee^{u} \lambda,\ee^{u} \mu) - F_0(\ee^{u'} \lambda,\ee^{u'} \mu)) J_0(u-u') \phi(\ee^{u'}) \,\dd u'.
\end{equation}

In order to check that $M_0,M_{\lie{v}},M_{\lie{z}}$ satisfy the smoothness assumptions of Corollary \ref{cor:applytheorem}, we shall obtain sufficiently explicit expressions for the functions $F_0,F_{\lie{v}},F_{\lie{z}}$, allowing us eventually to derive suitable estimates for their derivatives.

\subsection{Gelfand transform on H-type groups}\label{ss:gelfand}

We briefly recall some facts about the Gelfand transform of radial functions on H-type groups that we shall use in our discussion; for more details, we refer to \cite{DR921,MRS96}.

A function $f$ on the H-type group $N$ is said to be \emph{radial} if it has the form $f(x,z) = \tilde f(|x|,z)$ for some function $\tilde f$ and all $(x,z) \in N$. The space $L^1_{\rad}(N)$ of integrable radial functions on $N$, with convolution and involution, is a commutative Banach $*$-subalgebra of $L^1(N)$.

The Gelfand transform $\Gf f$ of a function $f \in L^1_\rad(N)$ is essentially given (up to a negligible set of characters) by
\begin{equation*}
\Gf f(\mu,\ell)=\binom{\ell+\frac{\dimV}{2}-1}{\ell}^{-1} \int_{N} f(x,z) \, L_{\ell}^{\frac{\dimV}{2}-1}(|\mu||x|^{2}/2)\, \ee^{-|\mu||x|^{2}/4} \, \ee^{-i\mu\cdot z} \,\dd x \,\dd z
\end{equation*}
for $\ell \in \NN$ and $\mu \in \RR^{\dimZ}$, where
\[
L_\ell^a(t) \defeq \frac{1}{\ell!} \, \ee^t \, t^{-a} \frac{d^\ell}{dt^\ell}(\ee^{-t} \,t^{\ell+a}),\quad t\in\RR,
\]
is the Laguerre polynomial of type $a>-1$ and degree $\ell$. For the Gelfand transform we have a Plancherel formula (see \cite[eq.\ (1.1)]{MRS96}),
\begin{equation}\label{eq:Gf_Plancherel}
\|f\|_{L^2(N)}^2 = (2\pi)^{-Q} \int_{\RR^{\dimZ}} \sum_{\ell \in \NN} |\Gf f(\mu,\ell)|^2 \binom{\ell + \frac{\dimV}{2}-1}{\ell} \,|\mu|^\frac{\dimV}{2} \,\dd\mu,
\end{equation}
valid at least for radial $f \in L^1 \cap L^2(N)$,
by means of which $\Gf$ extends to an isometric isomorphism from $L^2_{\rad}(G)$ to $L^2(\RR^{\dimZ} \times \NN, (2\pi)^{-Q} \binom{\ell + \frac{\dimV}{2}-1}{\ell}\,|\mu|^\frac{\dimV}{2} \,\dd\mu \, \dd\#(\ell))$; correspondingly, we have
an inversion formula (see \cite[eq.\ (1.2)]{MRS96}),
\[
f(x,z) = (2\pi)^{-Q} \int_{\RR^{\dimZ}} \sum_{\ell \in \NN} \Gf f(\mu,\ell) \, L_\ell^{\frac{\dimV}{2}-1}(|\mu||x|^2/2) \,\ee^{-|\mu| |x|^2/4} \,\ee^{i\mu\cdot z} \,|\mu|^\frac{\dimV}{2} \,\dd\mu,
\]
valid at least for $f \in L^1_\rad(N)$ in the Schwartz class.

Radial functions on $N$ are intimately related with the joint functional calculus for $(\opLN,-i\nabla_z)$; indeed, the convolution kernel $\krn_{F(\opLN,-i\nabla_z)}$ of an operator $F(\opLN,-i\nabla_z)$ in the joint functional calculus is radial, and
\begin{equation}\label{eq:Gffun}
\Gf \krn_{F(\opLN,-i\nabla_z)}(\mu,\ell) = F((2\ell+\dimV/2)|\mu|,\mu),
\end{equation}
see \cite[eq.\ (2.1)]{MRS96}. In light of this, the Plancherel formula \eqref{eq:Gf_Plancherel} corresponds to that for the joint functional calculus for $(\opLN,-i\nabla_z)$ \cite[Theorem 3.10 and Section 5.3]{M11}.
In addition, we recall from \cite[eq.\ (5.2)]{MRS96} the effect on the Gelfand transform of multiplication by the weight $|x|^2$:
\begin{multline}\label{eq:Gfwgt}
\Gf [(x,z) \mapsto |x|^2 f(x,z)](\mu,\ell) \\
= \frac{2}{|\mu|} \left[ \left(2\ell+\frac{\dimV}{2} \right)\Gf f(\mu,\ell)-\ell \Gf f(\mu,\ell-1)-\left(\ell+\frac{\dimV}{2} \right) \Gf f(\mu,\ell+1) \right].
\end{multline}

We now exploit some formulas from \cite{RT20} to obtain an explicit expression for the Gelfand transform of certain radial functions, which may be compared with the formula for the fundamental solution for $\opLN$ \cite{K80}.

\begin{lemma}\label{lem:lemmaPsi}
Let $H$ be given by \eqref{eq:Hnorm},
and let $s \in \CC$ be such that $\Re s> -1$. If
\[
\Psi_{s}(x,z) \defeq H(x,z)^{-(Q+s)/2},
\]
then
\begin{equation}\label{eq:psi_kernel}
\Psi_s = \krn_{\Xi_s(\opLN,-i\nabla_z)},
\end{equation}
where
\begin{equation}\label{eq:definitionHs}
\Xi_s(\lambda,\mu) \defeq 
\frac{2^{\dimV} \pi^{\frac{\dimN}{2}}\Gamma\left(\frac{\dimV}{4}+\frac{s}{2}\right)}{\Gamma\left( \frac{Q+s}{2}\right) \Gamma\left( \frac{\dimV}{2}+s\right) }
\int_{0}^{\infty} \left( \frac{|\mu|}{\sinh(t|\mu|)}\right)^{s+1} \ee^{-|\mu|\coth(t|\mu|)} \,\ee^{-t\lambda} \,\dd t;
\end{equation}
in particular,
\begin{equation}\label{eq:PSIFOURIER}
\Gf \Psi_{s}(\mu,\ell) = \Xi_s\left( (2\ell+\frac{\dimV}{2})|\mu|,\mu\right).
\end{equation}
\end{lemma}
\begin{proof}
From the above definitions it readily follows that $\Psi_s \in L^2(N)$ and that the right-hand side of \eqref{eq:PSIFOURIER} is in $L^2(\RR^{\dimZ} \times \NN, (2\pi)^{-Q}\binom{\ell + \frac{\dimV}{2}-1}{\ell}\,|\mu|^\frac{\dimV}{2} \,\dd\mu \, \dd\#(\ell))$ whenever $\Re s >-1$ (cf.\ Lemma \ref{lem:Propositionderivativeestimate}). So, by \eqref{eq:Gffun} and the Plancherel formula \eqref{eq:Gf_Plancherel}, both sides of \eqref{eq:psi_kernel} are in $L^2(N)$ for $\Re s> -1$, and they are (weakly) analytic in $s$. Hence, by analytic continuation, it is enough to prove \eqref{eq:psi_kernel} for $s \in (0,1)$.

Let $q_{t} = \krn_{\ee^{-t\opLN}}$ be the heat kernel associated to $\opLN$.
We recall the well-known Mehler-type formula for the partial Fourier transform of $q_t$ along $\lie{z}$ (see, e.g., \cite[Proposition 4]{MM16} or \cite[eq.\ (3.4)]{RT20}):
\begin{equation*}
\int_{\RR^{\dimZ}} q_t(x,z) \, \ee^{-i\mu\cdot z} \,\dd z 
= (4\pi)^{-\dimV/2} \left(\frac{|\mu|}{\sinh(t|\mu|)}\right)^{\dimV/2} \ee^{-\frac{1}{4}|\mu|\coth(t|\mu|)|x|^2}
\end{equation*}
for $x \in \RR^{\dimV}$ and $\mu\in \RR^{\dimZ}$.
Much as in \cite[eq.\ (3.5)]{RT20}, for $t>0$ and $s \in (0,1)$, we consider functions $p_{t,s} : \RR^{\dimZ} \to \CC$ whose Fourier transforms are given by
\[
\widehat p_{t,s}(\mu) =	\int_{\RR^{\dimZ}} p_{t,s}(z) \, \ee^{-i\mu\cdot z} \,\dd z = (4\pi)^{-s-1} \left(\frac{|\mu|}{\sinh(t|\mu|)}\right)^{s+1} \ee^{-\frac{1}{4}|\mu|\coth(t|\mu|)}
\]
for $\mu\in \RR^{\dimZ}$ (notice that $p_{t,s}$ here corresponds to $p_{t,s}(1,\cdot)$ in \cite{RT20}). 
By \cite[eq.\ (3.6) and Corollary 3.6]{RT20} applied with $\rho=1$, we deduce that
\begin{equation*}
C_{\dimV,\dimZ,s} ((1+|x|^2)^2+16|z|^2)^{-(Q+s)/2} = \int_{0}^{\infty} \int_{\RR^{\dimZ}} p_{t,s}(w) \, q_{t}(x,z-w) \,\dd w \,\dd t.
\end{equation*}
where
\[
C_{\dimV,\dimZ,s}=\frac{4^{\dimZ-1}}{\pi^{\dimN/2+1+s}}\frac{\Gamma(\dimV/2+s) \,\Gamma((Q+s)/2)}{\Gamma(\dimV/4+s/2)}.
\]
In other words,
\begin{equation}\label{eq:int_convN}
C_{\dimV,\dimZ,s} ((1+|x|^2)^2+16|z|^2)^{-(Q+s)/2} = \int_{0}^{\infty} (q_{t} *_N (\delta_0 \otimes p_{t,s}))(x,z) \,\dd t.
\end{equation}
As $q_t$ is the convolution kernel of $\ee^{-t\opLN}$, while $\delta_0 \otimes p_{t,s}$ is the convolution kernel of $\widehat p_{t,s}(-i\nabla_z)$ on $N$, the right-hand side of \eqref{eq:int_convN} is the convolution kernel of the operator
\[	
\int_{0}^{\infty} \widehat p_{t,s}(-i\nabla_z) \, \ee^{-t\opLN} \,\dd t.
\]
Through the automorphic scaling $(x,z) \mapsto \dil_{1/4} (x,z)$ we thus obtain that
\[
4^{-Q} C_{\dimV,\dimZ,s} ((1+|x|^2/4)^2+|z|^2)^{-(Q+s)/2} = \krn_{\Theta_s(\opLN,-i\nabla_z)}(x,z),
\]
where
\[\begin{split}
\Theta_s(\lambda,\mu) 
&\defeq \int_0^\infty \widehat p_{t,s}(4\mu) \,\ee^{-4t\lambda} \,\dd t \\
&= 4^{-1} \pi^{-s-1} \int_0^\infty \left(\frac{|\mu|}{\sinh(t|\mu|)}\right)^{s+1} \ee^{-|\mu|\coth(t|\mu|)} \, \ee^{-t\lambda} \,\dd t
\end{split}\]
This proves \eqref{eq:psi_kernel}, and \eqref{eq:PSIFOURIER} follows from \eqref{eq:Gffun}.
\end{proof}

\begin{lemma}\label{lem:Psiconvolutionkernel}
With the notation of Lemma \ref{lem:lemmaPsi}, we have
\[
(1+|x|^{2}/4)\Psi_s(x,z) = \krn_{\tilde \Xi_s{(\opLN,-i\nabla_z)}}(x,z),	
\]
where
\begin{equation}\label{eq:definitiontildeHs}
\tilde \Xi_s(\lambda,\mu) \defeq \Xi_s(\lambda,\mu) - \lambda \Xi^{(2)}_s(\lambda,\mu) - \frac{\dimV}{4} \Xi_s^{(1)}(\lambda,\mu)
\end{equation}
and 
\begin{equation}\label{eq:H123}
\begin{aligned}
\Xi^{(1)}_s(\lambda,\mu) &\defeq \int_{-1}^{1} \partial_\lambda \Xi_s(\lambda + 2v |\mu|,\mu) \,\dd v,\\
\Xi^{(2)}_s(\lambda,\mu) &\defeq \int_{-1}^{1} \partial_\lambda^{2} \Xi_s(\lambda + 2v |\mu|,\mu) \, (1-|v|) \,\dd v,
\end{aligned}
\end{equation}
\end{lemma}
\begin{proof}
By \eqref{eq:PSIFOURIER} and \eqref{eq:definitionHs}, we can think of $\Gf \Psi_s(\mu,\ell) = \Xi_s((2\ell+\dimV/2)|\mu|,\mu)$ as being defined for noninteger values of $\ell$ too. Thus, by \eqref{eq:Gfwgt} and the fundamental theorem of calculus, for all $\ell \in \NN$ and $\mu \in \dot\RR^{\dimZ}$,
\[\begin{split}
&\frac{|\mu|}{2} \Gf (|x|^{2}\Psi_{s})(\mu,\ell)\\
&=\ell \left( \Gf \Psi_s(\mu,\ell)-\Gf \Psi_s(\mu,\ell-1)\right) +\left(\ell+\frac{\dimV}{2}\right)
\left( \Gf\Psi_s(\mu,\ell)-\Gf\Psi_s(\mu,\ell+1 )\right) \\
&=\ell \int_0^1 \partial_\ell \Gf \Psi_s(\mu,\ell-1+w) \,\dd w- \left(\ell+\frac{\dimV}{2}\right) \int_0^1 \partial_\ell \Gf\Psi_s(\mu,\ell+w) \,\dd w \\
&=-\left(\ell+\frac{\dimV}{4}\right) \int_0^1 \int_0^1 \partial_\ell^2 \Gf \Psi_s(\mu,\ell-v+w) \,\dd w \,\dd v-\frac{\dimV}{4} \int_{-1}^1 \partial_\ell \Gf\Psi_s(\mu,\ell+w) \,\dd w \\
&=-\left(\ell+\frac{\dimV}{4}\right) \int_{-1}^1 \partial_\ell^2 \Gf \Psi_s(\mu,\ell+u) \, (1-|u|) \,\dd u -\frac{\dimV}{4} \int_{-1}^1 \partial_\ell \Gf\Psi_s(\mu,\ell+w) \,\dd w .
\end{split}\]
By plugging the expression $\Gf \Psi_s(\mu,\ell) = \Xi_s((2\ell+\dimV/2)|\mu|,\mu)$ into the previous formula, we finally get, by \eqref{eq:H123},
\[\begin{split}
\frac{1}{4} \Gf (|x|^{2}\Psi_{s})(\mu,\ell)
&=-\left(2\ell+\frac{\dimV}{2}\right) |\mu| \int_{-1}^1 \partial_\lambda^2 \Xi_s((2\ell+2u+\dimV/2)|\mu|,\mu) \, (1-|u|) \,\dd u \\
&\quad -\frac{\dimV}{4} \int_{-1}^1 \partial_\lambda \Xi_s((2\ell+2w+\dimV/2)|\mu|,\mu) \,\dd w \\
&= \left. \left[-\lambda \Xi_s^{(2)}(\lambda,\mu) - \frac{\dimV}{4} \Xi_s^{(1)}(\lambda,\mu)\right]\right|_{\lambda = (2\ell+\dimV/2)|\mu|},
\end{split}\]
thus also, by \eqref{eq:definitiontildeHs},
\[
\Gf ((1+|x|^{2}/4)\Psi_{s})(\mu,\ell)
= \tilde \Xi_s((2\ell+\dimV/2)|\mu|,\mu),
\]
and the conclusion follows by \eqref{eq:Gffun}.
\end{proof}

Thanks to the previous lemmas, we can now write explicit expressions for the functions $F_0,F_{\lie{v}},F_{\lie{z}}$ discussed in Section \ref{ss:convKopvalmult}.

\begin{proposition}\label{prp:r0rbulletF0Fbullet}
We can write the functions $r_0,r_{\lie{v}},r_{\lie{z}}$ from \eqref{eq:r0} and \eqref{eq:Kbullet} as
\[
r_0 = \krn_{F_0(\opLN,-i\nabla_z)}, \qquad r_{\lie{v}} = \krn_{F_{\lie{v}}(\opLN,-i\nabla_z)}, \qquad r_{\lie{z}} = \krn_{F_{\lie{z}}(\opLN,-i\nabla_z)},
\]
where, with the notation of Lemmas \ref{lem:lemmaPsi} and \ref{lem:Psiconvolutionkernel},
\begin{equation}\label{eq:F0Fbullet}
F_0(\lambda,\mu) = \tilde \Xi_2(\lambda,\mu), \quad 
F_{\lie{v}}(\lambda,\mu) = \lambda^{1/2} \Xi_0(\lambda,\mu), \quad
F_{\lie{z}}(\lambda,\mu) = \lambda \Xi_0(\lambda,\mu).
\end{equation}
\end{proposition}
\begin{proof}
From \eqref{eq:r0} we know that
\[
r_0(x,z) = (1+|x|^{2}/4) \Psi_2(x,z),
\]
so the expression for $F_0$ is a direct consequence of Lemma \ref{lem:Psiconvolutionkernel}. On the other hand, from \eqref{eq:Kbullet} and Lemma \ref{lem:lemmaPsi} we have
\[
r_{\lie{v}} = \opLN^{1/2} \Psi_0 = \opLN^{1/2} \krn_{\Xi_0(\opLN,-i\nabla_z)} = \krn_{\opLN^{1/2} \Xi_0(\opLN,-i\nabla_z)},
\]
whence we deduce the expression for $F_{\lie{v}}$, and a similar computation gives the expression for $F_{\lie{z}}$.
\end{proof}

\begin{remark}\label{rem:notradial}
Arguing as in the proof of Proposition \ref{prp:r0rbulletF0Fbullet}, from \eqref{eq:rjXj} on could also deduce that $r_{j+\dimV} = \krn_{F_j(\opLN,-i\nabla_z)}$, where $F_j(\lambda,\mu) = -Q^{-1} i \mu_j \Xi_0(\lambda,\mu)$, for $j=1,\dots,\dimZ$. However, such representation is not possible for the $r_j$ with $j=1,\dots,\dimV$, because they are not radial, see \eqref{eq:rjV}. The reduction in Proposition \ref{prp:RzNfactorisation} is thus essential to treat the Riesz transforms $\Rz_j$ for $j=1,\dots,\dimV$ with our approach.
\end{remark}

\subsection{Estimates for the operator-valued symbols}\label{ss:symbol_estimates}

On the basis of Proposition \ref{prp:r0rbulletF0Fbullet} and the discussion in Section \ref{ss:convKopvalmult}, we can write the convolution operators with kernels $K_0,K_{\lie{v}},K_{\lie{z}}$ as the operator-valued spectral multipliers $M_0(\opLN,-i\nabla_z)$, $M_{\lie{v}}(\opLN,-i\nabla_z)$ and $M_{\lie{z}}(\opLN,-i\nabla_z)$, whose operator-valued symbols $M_0,M_{\lie{v}},M_{\lie{z}}$ are given in \eqref{eq:Mbullet} and \eqref{eq:M0}, where the functions $F_0,F_{\lie{v}},F_{\lie{z}}$ are as in \eqref{eq:F0Fbullet}.

We are now ready to present the main result of this section, which concerns the weighted $L^{2}$-boundedness of the integral operators $M_{0}(\lambda,\mu),M_{\lie{v}}(\lambda,\mu),M_{\lie{z}}(\lambda,\mu)$ from \eqref{eq:Mbullet} and \eqref{eq:M0}, as well as their $(\lambda,\mu)$-derivatives.

\begin{proposition}\label{prp:propositionaboutkjk0}
The operator-valued functions $M_{0},M_{\lie{v}},M_{\lie{z}}$ are of class $C^{\infty}$ on $\Rpos \times \RR^{\dimZ}$. Moreover, for all $\alpha \in \NN^{1+\dimZ}$, $\kappa \in \Rpos$, $w\in A_{2}(\RR^+)$ and $\ell \in \{0,\lie{v},\lie{z}\}$,
\begin{equation}\label{eq:inequalityaboutkjk0h0}
		\max_{|\alpha|\geq0}\sup_{\substack{(\lambda,\mu)\in\RR^{\dimZ+1}\setminus \{0\}\\\lambda\geq \kappa|\mu|}}\left\|\lambda^{|\alpha|}\partial_{(\lambda,\mu)}^{\alpha} M_\ell(\lambda,\mu)\right\|_{\opLN(L^{2}(w))}
		\lesssim_{[w]_{A_{2}}} 1.
\end{equation}
\end{proposition}
\begin{proof}
We shall derive this estimate from \cite[Proposition 5.6]{M23}; we point out that the result in \cite{M23} applies to operators on $L^2(\RR)$, so we must use the change of variables $a = \ee^u$ as in Remark \ref{rem:productmeasure} to apply it to the operators on $L^2(\Rpos)$ considered here.

Now, for $\ell \in \{\lie{v},\lie{z}\}$, in light of \eqref{eq:Mbullet} and \eqref{eq:J0Jbullet}, the operators $\lambda^{|\alpha|} \partial_{(\lambda,\mu)}^\alpha M_\ell(\lambda,\mu)$ are integral operators on $L^2(\Rpos)$ with kernels
\[
(\ee^u,\ee^{u'}) \mapsto F_\ell^\alpha(\ee^{u'} \lambda,\ee^{u'} \mu) \,\frac{\chr_{\{u \leq u'-1\}}}{u-u'},
\]
where $F_\ell^\alpha(\lambda,\mu) \defeq \lambda^{|\alpha|} \partial_{(\lambda,\mu)}^{\alpha} F_\ell(\lambda,\mu)$;
hence, by \cite[Proposition 5.6(ii)]{M23}, to establish the estimate \eqref{eq:inequalityaboutkjk0h0}, it is enough to show that, for some $\delta >0$,
\begin{equation*}
|F_\ell^\alpha(\lambda,\mu)| \lesssim_{\alpha,\kappa} \min\{\lambda,\lambda^{-1}\}^{\delta} \qquad\text{when } \lambda \geq \kappa|\mu|,
\end{equation*}
which follows from Proposition \ref{prp:theoremtildeHtildeestimate} below.

Similarly, in light of \eqref{eq:M0} and \eqref{eq:J0Jbullet}, the operators $\lambda^{|\alpha|} \partial_{(\lambda,\mu)}^\alpha M_0(\lambda,\mu)$ are integral operators with kernels
\[
(\ee^u,\ee^{u'}) \mapsto 
[F_0^\alpha(\ee^{u} \lambda,\ee^{u} \mu) - F_0^\alpha(\ee^{u'} \lambda,\ee^{u'} \mu)] \, \frac{\chr_{\{u \geq u'+1\}}}{u-u'},
\]
where $F_0^\alpha(\lambda,\mu) \defeq \lambda^{|\alpha|} \partial_{(\lambda,\mu)}^{\alpha} F_0(\lambda,\mu)$.
Thus, in order to obtain \eqref{eq:inequalityaboutkjk0h0}, by \cite[Proposition 5.6(i)]{M23}, it is enough to prove that, for some $\delta>0$,
\begin{align*}
\left|F_0^\alpha(\lambda,\mu) \right| &\lesssim_{\alpha,\kappa} (1+\lambda)^{-\delta} &&\text{when } \lambda \geq \kappa|\mu|,\\
|F_0^\alpha(\lambda,\mu)- F_0^\alpha(\lambda',\mu')| &\lesssim_{\alpha,\kappa} |(\lambda,\mu)-(\lambda',\mu')|^{\delta} &&\text{when } \lambda \geq \kappa|\mu|, \ \lambda' \geq \kappa|\mu'|,
\end{align*}
which also follows from Proposition \ref{prp:theoremtildeHtildeestimate} below.
\end{proof}

\begin{proposition}\label{prp:theoremtildeHtildeestimate}
Let $F_0,F_\lie{v},F_{\lie{z}}$ be as in \eqref{eq:F0Fbullet}.
Let $\kappa \in \Rpos$ and $c \in (0,1)$. For any $\alpha\in \NN^{1+\dimZ}$ and $(\lambda,\mu) \in \Rpos \times \RR^{\dimZ}$ with $\lambda\geq \kappa|\mu|$, we have 
\begin{equation}\label{eq:l12H0est}
\begin{aligned}
\left|\lambda^{|\alpha|}\partial_{(\lambda,\mu)}^\alpha F_{\lie{v}}(\lambda,\mu) \right|
&\lesssim_{c,\kappa,\alpha} \lambda^{1/2} \log(\ee+1/\lambda) \, \ee^{-2c\sqrt{\lambda+|\mu|}}, \\
\left|\lambda^{|\alpha|}\partial_{(\lambda,\mu)}^\alpha F_{\lie{z}}(\lambda,\mu) \right|
&\lesssim_{c,\kappa,\alpha} \lambda \log(\ee+1/\lambda) \, \ee^{-2c\sqrt{\lambda+|\mu|}}
\end{aligned}
\end{equation}
and
\begin{equation}\label{eq:tH2est}
\left|\lambda^{|\alpha|}\partial_{(\lambda,\mu)}^\alpha F_0(\lambda,\mu) \right|
\lesssim_{c,\kappa,\alpha} \begin{cases}
\ee^{-2c\sqrt{\lambda+|\mu|}} &\text{if } |\alpha|=0,\\
\lambda \log(\ee+1/\lambda) \, \ee^{-2c\sqrt{\lambda+|\mu|}} &\text{if } |\alpha|=1,\\
\lambda \, \ee^{-2c\sqrt{\lambda+|\mu|}} &\text{if } |\alpha|\geq 2.
\end{cases}
\end{equation}
In addition, for all $\varepsilon \in (0,1)$ and $\alpha \in \NN^{1+\dimZ}$, the function
$(\lambda,\mu) \mapsto \lambda^{|\alpha|}\partial_{(\lambda,\mu)}^\alpha F_0(\lambda,\mu)$ is $\varepsilon$-H\"older continuous on $\{ (\lambda,\mu) \in \Rpos \times \RR^{\dimZ} \tc \lambda \geq \kappa|\mu|\}$.
\end{proposition}

The proof of the above proposition will be presented after a few preliminary lemmas. The first one concerns differentiation of radial extensions of even functions, and will be useful to estimate $\mu$-derivatives of expressions, such as \eqref{eq:H123}, which also involve $|\mu|$.

\begin{lemma}\label{lem:propositionderivative}
Let $n\in\NN$. Let $f\in C^{n}(\RR)$ be even, and define $F : \RR^{\dimWZX} \to \CC$ by $F(\mu)=f(|\mu|)$, $\mu\in\RR^{\dimWZX}$. Then 
$F\in C^{n}(\RR^{\dimWZX})$. Moreover, for all $\mu \in \RR^{\dimWZX}$ and $\alpha\in\NN^{\dimWZX}$ with $|\alpha|\leq n$,
\begin{equation}\label{eq:derivative2}
\left| \partial_{\mu}^{\alpha}F(\mu)\right| 
\lesssim_{\dimWZX,|\alpha|} \int_{0}^{1}|f^{(|\alpha|)}(s|\mu|)| \,\dd \mu_{|\alpha|}(s)\leq \sup_{s\in[0,1]}|f^{(|\alpha|)}(s|\mu|)|,
\end{equation}
for some probability measure $\mu_{|\alpha|}$ on $[0,1]$, and more precisely
\begin{equation}\label{eq:derivative1}
|\partial_{\mu}^{\alpha}F(\mu)|\leq|f^{(|\alpha|)}(|\mu|)| \qquad\text{if } |\alpha|\leq 1.
\end{equation}
\end{lemma}
\begin{proof}
The case $n=0$ is trivial, so we may assume $n>0$.
	
Let us first prove that $F\in C^{n}(\RR^{\dimWZX})$. As $f$ is even, $f^{(k)}(0)=0$ for any odd $k \leq n$. Thus we can write $f=P+R$, where 
\[
P(t)\defeq \sum_{\substack{0\leq k\leq n\\k \text{ even}}} \frac{f^{(k)}(0)}{k!} t^{k}
\]
is the Taylor polynomial of $f$ at $0$ of order $n$, and the remainder $R\in C^{n}(\RR)$ satisfies $R^{(k)}(0)=0$ for $k=0,\dots,n$. In particular, 
\begin{equation}\label{eq:smalloremainder}
R^{(k)}(t)=o(t^{n-k}) \text{ as } t\to 0 \quad \text{for } k=0,\dots,n.
\end{equation}
Correspondingly, $F(\mu)=f(|\mu|)=P(|\mu|)+R(|\mu|)$, where 
\[
P(|\mu|)=\sum_{\substack{0\leq k\leq n\\k \text{ even}}}\frac{f^{(k)}(0)}{k!} |\mu|^{k}
\]
is a polynomial in $|\mu|^2$, so $\mu\mapsto P(|\mu|)$ is in $C^\infty(\RR^{\dimWZX})$, and it just remains to show that $\mu\mapsto R(|\mu|)$ is in $C^{n}(\RR^{\dimWZX})$.

On the other hand, as $|\cdot|$ is smooth off the origin, $R(|\cdot|) \in C(\RR^{\dimWZX})\cap C^{n}(\dot\RR^{\dimWZX})$.
We now show that $\partial_{\mu}^{\alpha}[R(|\mu|)]\to 0$ as $\mu\to 0$ whenever $|\alpha| \leq n$. When $\alpha = 0$, this is clear from \eqref{eq:smalloremainder}.
On the other hand, if $1\leq |\alpha|\leq n$, by the Chain and Leibniz rules,
\begin{equation}\label{eq:der_abs}
\partial_{\mu}^{\alpha}[R(|\mu|)]=\sum_{k=1}^{|\alpha|}Q_{\alpha,k}(\mu)R^{(k)}(|\mu|),
\end{equation}
for suitable functions $Q_{\alpha,k} \in C^{\infty}(\dot\RR^{\dimWZX})$ homogeneous of degree $k-|\alpha|$; thus, again by \eqref{eq:smalloremainder},
\[
\partial_{\mu}^{\alpha}[R(|\mu|)] = \sum_{k=1}^{|\alpha|} O(|\mu|^{k-|\alpha|}) \, o(|\mu|^{n-k}) = o(|\mu|^{n-|\alpha|})\to 0\text{ as } \mu\to 0.
\]
This proves that $\mu\mapsto \partial_{\mu}^{\alpha}[R(|\mu|)]$ extends to a continuous function on the whole $\RR^{\dimWZX}$ for any $\alpha\in\NN^{\dimWZX}$ with $|\alpha|\leq n$, and consequently $R(|\cdot|) \in C^n(\RR^{\dimWZX})$.
	
It remains to prove the estimates \eqref{eq:derivative2} and \eqref{eq:derivative1}. Clearly they hold for $\alpha=0$, so we may assume $|\alpha|>0$, and moreover by continuity it is enough to check them for $\mu \neq 0$.
	
As $\partial_{\mu_{j}}F(\mu)=\frac{\mu_{j}}{|\mu|}f'(|\mu|)$ for $j=1,\dots,\dimWZX$, the estimate \eqref{eq:derivative1} is clear. 
To prove \eqref{eq:derivative2}, let $\alpha \in \NN^{\dimWZX}$ be such that $1 \leq m \defeq |\alpha|\leq n$, consider the polynomial
\[
P_{m}(t) \defeq \sum_{\substack{0 \leq k < m \\ k \text{ even}}} \frac{f^{(k)}(0)}{k!}t^{k}
\]
and set $R_{m} \defeq f-P_{m}$. As before, $R_{m}\in C^n(\RR)$ and $R_{m}^{(k)}(0)=0$ for $0\leq k<m$. Moreover $P_{m}^{(m)} \equiv 0$, thus $R^{(m)}_{m} = f^{(m)}$. Therefore, for $0 \leq k < m$, the Taylor expansion of $R_m^{(k)}$ with integral remainder at $0$ of order $m-k-1$ is given by
\begin{equation}\label{eq:RNk}
\begin{split}
R_{m}^{(k)}(t) 
&= \frac{1}{(m-k-1)!} \int_0^t R_m^{(m)}(u) \, (t-u)^{m-k-1} \,\dd u\\
&= t^{m-k} \int_0^1 f^{(m)}(st) \frac{(1-s)^{m-k-1}}{(m-k-1)!} \,\dd s.
\end{split}
\end{equation}
Now, $F(\mu) = f(|\mu|) = P_{m}(|\mu|) + R_{m}(|\mu|)$, where $P_{m}(|\mu|)$ is a polynomial of degree less than $m$; so, as $|\alpha|=m$,
\[
\partial_{\mu}^{\alpha}F(\mu)
= \partial_{\mu}^{\alpha}[R_{m}(|\mu|)]
= \sum_{k=1}^{m} Q_{\alpha,k}(\mu) R_{m}^{(k)}(|\mu|),
\]
where the $Q_{\alpha,k}$ are smooth and $(k-|\alpha|)$-homogeneous as in \eqref{eq:der_abs}; combining this identity with \eqref{eq:RNk} and the fact that $R_m^{(m)} = f^{(m)}$, we conclude that
\[\begin{split}
&|\partial_{\mu}^{\alpha}F(\mu)| 
\lesssim_{\dimWZX,m} \sum_{k=1}^{m} |\mu|^{k-m} |R_{m}^{(k)}(|\mu|)| \\
&\leq \sum_{k=1}^{m-1} |\mu|^{k-m} |\mu|^{m-k} \int_0^1 |f^{(m)}(s|\mu|)| \, \frac{(1-s)^{m-k-1}}{(m-k-1)!} \,\dd s+|f^{(m)}(|\mu|)|\\
&\lesssim_{m} \int_0^1 |f^{(m)}(s|\mu|)| \,\dd \mu_{m}(s),
\end{split}\]
as required, where $\mu_{m}$ is some probability measure on $[0,1]$.
\end{proof}

Let $\fS,\fT : \RR \to \RR$ be the even, real analytic functions defined by
\begin{equation}\label{eq:defST}
\fS(u) = \frac{u}{\sinh u}, \qquad \fT(u) = \frac{u}{\tanh u}.
\end{equation}
As these functions appear in the expression \eqref{eq:definitionHs} for $\Xi_s$, it is useful to collect a few estimates for their derivatives.

\begin{lemma}
The functions $\mu \mapsto \fS(|\mu|)$ and $\mu \mapsto \fT(|\mu|)$ are real-analytic on $\RR^{\dimWZX}$, and
\begin{align}
\label{eq:estimateS}
|\partial_\mu^\alpha \fS(|\mu|) | &\lesssim_{\alpha} (1+|\mu|) \, \ee^{-|\mu|},\\
\label{eq:estimateT}
|\partial_\mu^\alpha \fT(|\mu|) | &\lesssim_{\alpha} (1+|\mu|)^{1-|\alpha|}
\end{align}
for all $\alpha \in \NN^{\dimWZX}$ and $\mu \in \RR^{\dimWZX}$. In addition $\fT \geq 1$ pointwise.
\end{lemma}
\begin{proof}
As $\left|\tanh u\right| \leq |u|$ for all $u \in \RR$, we have $\fT(u) \geq 1$ for all $u \in \RR$. Moreover, as $\fS$ and $\fT$ are even and real-analytic on $\RR$, it is clear that $\fS(|\cdot|)$ and $\fT(|\cdot|)$ are real-analytic on $\RR^{\dimWZX}$, and the estimates \eqref{eq:estimateS}-\eqref{eq:estimateT} hold trivially for $|\mu| \leq 1$. It only remains to prove the estimates \eqref{eq:estimateS}-\eqref{eq:estimateT} for large $|\mu|$.

It is easily checked by induction that, for all $k \in \NN$,
\[
\partial_u^k \left( \frac{1}{\sinh u}\right) = \frac{P_k(\sinh u, \cosh u)}{\sinh^{k+1} u},
\]
where $P_{k}$ is a two-variate homogeneous polynomial of degree $k$;
consequently, as $\sinh u \simeq \cosh u \simeq \ee^u$ for $u \geq 1$, we get
\begin{equation}\label{eq:est_1sh}
\left|\partial_u^k \left(\frac{1}{\sinh u}\right)\right| \lesssim_k \ee^{-u} \qquad\text{for } u \geq 1.
\end{equation}
As
\[
\fS^{(k)}(u) = k \partial_u^{k-1} \left(\frac{1}{\sinh u}\right) + u \partial_u^k \left(\frac{1}{\sinh u}\right),
\]
from \eqref{eq:est_1sh} we deduce that
\begin{equation}\label{eq:derS}
|\fS^{(k)}(u)| \lesssim_k u \, \ee^{-u} \qquad\text{for } u \geq 1.
\end{equation}
Now, as
\[
\partial_u \coth u = -\frac{1}{\sinh^2 u},
\]
by the Leibniz rule and \eqref{eq:est_1sh} we deduce that, for $u \geq 1$ and $k \in \NN$, 
\begin{equation}\label{eq:est_cth}
\left|\coth u\right| = \left|\frac{\sinh u}{\cosh u}\right| \simeq 1, \qquad |\partial_u^{k+1} \coth u| = \left| \partial_u^k \frac{1}{\sinh^2 u} \right| \lesssim_k e^{-2u}.
\end{equation}
Thus, much as above, from
\[
\fT^{(k)}(u) = k \partial_u^{k-1} \coth u + u \partial_u^k \coth u
\]
and \eqref{eq:est_cth} we conclude that, for $u \geq 1$ and $k \in \NN$,
\[
|\fT^{(k)}(u)| \lesssim_k \begin{cases}
u & \text{if } k = 0,\\
1 &\text{if } k = 1,\\
u \,\ee^{-2u} & \text{if } k \geq 2,
\end{cases}
\]
and in particular
\begin{equation}\label{eq:derT}
|\fT^{(k)}(u)| \lesssim_k u^{1-k} \qquad\text{for } u \geq 1.
\end{equation}

Finally, arguing as in \eqref{eq:der_abs}, for all $\alpha \in \NN^{\dimWZX}$ and $|\mu| \geq 1$, from \eqref{eq:derS} and \eqref{eq:derT} we deduce that
\[
\begin{aligned}
|\partial_{\mu}^{\alpha} \fS(|\mu|) | &\lesssim_\alpha \sum_{k=1}^{|\alpha|} |\mu|^{k-|\alpha|} |\fS^{(k)}(\mu)| \lesssim_\alpha |\mu| \, \ee^{-|\mu|}, \\
|\partial_{\mu}^{\alpha} \fT(|\mu|) | &\lesssim_\alpha \sum_{k=1}^{|\alpha|} |\mu|^{k-|\alpha|} |\fT^{(k)}(\mu)| \lesssim_\alpha |\mu|^{1-|\alpha|}
\end{aligned}
\]
as desired.
\end{proof}

\begin{lemma}\label{lem:lemmaExpST}
For all $\gamma \in \CC$, $\mu \in \RR^{\dimWZX}$, $t > 0$, $\alpha \in \NN^{\dimWZX}$,
\[
|\partial_\mu^\alpha [\fS(t|\mu|)^\gamma] | \lesssim_{\alpha,\gamma}
 t^{|\alpha|} (1+t|\mu|)^{\Re\gamma} \, \ee^{- t |\mu| \Re \gamma} 
\]
while
\[
|\partial_\mu^\alpha \ee^{-\fT(t|\mu|)/t}| \lesssim_\alpha \ee^{-\fT(t|\mu|)/t} \left( 1 + \frac{t}{1+t|\mu|} \right)^{(|\alpha|-1)_+}.
\]
\end{lemma}
\begin{proof}
Both estimates are trivial for $\alpha =0$, so we may assume $|\alpha|>0$.

By homogeneity, the estimate for $\fS(t|\mu|)$ reduces to the case $t=1$; now, by the chain and Leibniz rules,
\[
|\partial_\mu^\alpha [\fS(|\mu|)^\gamma] | \lesssim_{\alpha,\gamma}
\sum_{P=1}^{|\alpha|} \fS(|\mu|)^{\Re\gamma-P} \sum_{\substack{\alpha_1+\dots+\alpha_P = \alpha \\ |\alpha_j| \geq 1 \ \forall j}} \prod_{j=1}^P |\partial_\mu^{\alpha_j} \fS(|\mu|)|,
\]
so the required estimate follows by \eqref{eq:estimateS} and the fact that $\fS(|\mu|) \simeq (1+|\mu|) \,\ee^{-|\mu|}$.

As for the second estimate, by the Leibniz and chain rules, together with \eqref{eq:estimateT},
\begin{equation*}
\begin{split}
\left| \partial_\mu^{\alpha} \ee^{-\fT(t|\mu|)/t} \right|
&\lesssim_{\alpha} \ee^{-\fT(t|\mu|)/t} \sum_{P=1}^{|\alpha|} \sum_{\substack{\alpha_1 +\dots +\alpha_{P}=\alpha\\|\alpha_{j}|\geq 1 \ \forall j}} \prod_{j=1}^{P} \left| \partial_\mu^{\alpha_{j}} \left[\fT(t|\mu|)/t \right] \right| \\
&\lesssim_\alpha \ee^{-\fT(t|\mu|)/t} \sum_{P=1}^{|\alpha|} \sum_{\substack{\alpha_1 +\dots +\alpha_{P}=\alpha\\|\alpha_{j}|\geq 1 \ \forall j}} \prod_{j=1}^{P} t^{|\alpha_j|-1} (1+t|\mu|)^{1-|\alpha_j|}\\
&\simeq_\alpha \ee^{-\fT(t|\mu|)/t} \sum_{P=1}^{|\alpha|} t^{|\alpha|-P} (1+t|\mu|)^{P-|\alpha|}\\
&\simeq_\alpha \ee^{-\fT(t|\mu|)/t} \left(1+ \frac{t}{1+t|\mu|}\right)^{|\alpha|-1} ,
\end{split}
\end{equation*}
as claimed.
\end{proof}

The previous lemmas allow us to obtain sufficiently precise estimates for the derivatives of the functions $\Xi_s$ from \eqref{eq:definitionHs}.

\begin{lemma}\label{lem:Propositionderivativeestimate}
Let $s \in \CC$, $\kappa  \in \Rpos$ and $c \in (0,1)$, and assume
that $\sigma \defeq \Re s > -1$. 
For all $\alpha \in \NN^{1+\dimZ}$ and $(\lambda,\mu) \in \RR^{1+\dimZ}$ such that $\lambda+(\sigma+1)|\mu| \geq \kappa |\mu|$,
\[
|\partial_{(\lambda,\mu)}^{\alpha} \Xi_{s}(\lambda,\mu)| \lesssim_{s,\alpha,c,\kappa}
\begin{cases}
(\lambda+(\sigma+1)|\mu|)^{-(|\alpha|-\sigma)_+} \,\ee^{-2c\sqrt{\lambda+(\sigma+1)|\mu|}} &\text{if } |\alpha| \neq \sigma,\\
\log\left(\ee+\frac{1}{\lambda+(\sigma+1)|\mu|}\right) \,\ee^{-2c\sqrt{\lambda+(\sigma+1)|\mu|}} &\text{if } |\alpha|=\sigma.
\end{cases}
\]
\end{lemma}
\begin{proof}
By the definition of $\Xi_s(\lambda,\mu)$ in \eqref{eq:definitionHs}, together with \eqref{eq:defST}, we can write
\[
\frac{\Gamma\left( \frac{Q+s}{2}\right) \Gamma\left( \frac{\dimV}{2}+s\right) }{2^{\dimV} \pi^{\frac{\dimN}{2}}\Gamma\left(\frac{\dimV}{4}+\frac{s}{2}\right) }
\Xi_s(\lambda,\mu)
=\int_0^{\infty} \fS(t|\mu|)^{1+s} \, \ee^{-\fT(t|\mu|)/t} \, \ee^{-t\lambda} \, t^{-s} \,\frac{\dd t}{t}.
\]
Write $\alpha = (r,\beta)$, where $r \in \NN$ and $\beta \in \NN^{\dimZ}$. Then, by Lemma \ref{lem:lemmaExpST},
\[\begin{split}
&|\partial_{(\lambda,\mu)}^\alpha \Xi_s(\lambda,\mu)| \\
&\lesssim_{s,\alpha} \sum_{\beta_1+\beta_2=\beta} \int_0^\infty | \partial_{\mu}^{\beta_1} [\fS(t|\mu|)^{1+s}] | \, | \partial_{\mu}^{\beta_2} [\ee^{-\fT(t|\mu|)/t}] | \, \ee^{-t\lambda} \, t^{r-\sigma} \,\frac{\dd t}{t} \\
&\lesssim_{s,\alpha} \int_0^\infty (1+t|\mu|)^{1+\sigma} \, \ee^{-t[\lambda + (1+\sigma)|\mu|]} \, \ee^{-\fT(t|\mu|)/t} \\
&\qquad\times \sum_{\beta_1+\beta_2=\beta} t^{|\beta_1|+r-\sigma}  \left(1+ \frac{t}{1+t|x|}\right)^{(|\beta_2|-1)_+}   \,\frac{\dd t}{t} \\
&\lesssim_{s,\alpha} \int_0^\infty (1+t|\mu|)^{1+\sigma} \, \ee^{-t[\lambda + (1+\sigma)|\mu|]} \, \ee^{-1/t} \, t^{r-\sigma} (1+t^{|\beta|}) \,\frac{\dd t}{t} \\
&\lesssim_{c,\kappa,s,\alpha} \int_0^\infty \ee^{-ct[\lambda + (1+\sigma)|\mu|]} \, \ee^{-c/t} \, t^{|\alpha|-\sigma} \,\frac{\dd t}{t} \\
&\simeq_{c,s,\alpha} (\lambda+(1+\sigma)|\mu|)^{(\sigma-|\alpha|)/2} \,\BesK_{|\alpha|-\sigma}(2c\sqrt{\lambda+(1+\sigma)|\mu|}),
\end{split}\]
where $\BesK_\nu$ is the modified Bessel function of the second kind \cite[10.32.10]{DLMF}, and we used the fact that $\fT \geq 1$, $c \in (0,1)$ and $\lambda + (1+\sigma)|\mu| \geq \kappa |\mu|$. The desired estimate thus follows from the known asymptotics for $\BesK_\nu$ with $\nu\in \RR$, namely
\[
\BesK_\nu(t) \simeq_\nu \begin{cases}
t^{-|\nu|} &\text{if } \nu\neq 0, \ t \in (0,1],\\
\log(\ee+1/t) &\text{if } \nu= 0, \ t \in (0,1],\\
t^{-1/2} \,\ee^{-t} &\text{if } t \in [1,\infty),
\end{cases}
\]
see \cite[10.25.3, 10.27.3, 10.30.2, 10.30.3]{DLMF}, by taking into consideration that, by slightly decreasing the value of $c \in (0,1)$, any power of $\lambda+(1+\sigma)|\mu|$ can be absorbed by the exponentially decaying term $\ee^{-2c\sqrt{\lambda+(\sigma+1)|\mu|}}$ when $\lambda+(1+\sigma)|\mu| \geq 1$.
\end{proof}

Now we are in the position to present the proof of Proposition \ref{prp:theoremtildeHtildeestimate}.

\begin{proof}[Proof of Proposition \ref{prp:theoremtildeHtildeestimate}]
We first prove \eqref{eq:l12H0est}. Let $\delta \in \Rpos$. For any $\alpha \in\NN\times\NN^{\dimZ}$, when $\lambda\geq \kappa|\mu|$, by Lemma \ref{lem:Propositionderivativeestimate},
\[\begin{split}
\left| \lambda^{|\alpha|} \partial_{(\lambda,\mu)}^{\alpha} [\lambda^\delta \Xi_0 (\lambda,\mu)] \right| 
&\lesssim_{\delta,\alpha} \sum_{\alpha'\leq \alpha} \lambda^{\delta+|\alpha'|} \left| \partial_{(\lambda,\mu)}^{\alpha'} \Xi_0(\lambda,\mu)\right| \\
&\lesssim_{\alpha,c,\kappa} \lambda^{\delta} \log(\ee+1/\lambda) \,\ee^{-2c\sqrt{\lambda+|\mu|}},
\end{split}\]
which, in light of \eqref{eq:F0Fbullet}, proves \eqref{eq:l12H0est}.

As for \eqref{eq:tH2est}, recalling the definitions in \eqref{eq:H123}, and setting 
\begin{equation*}
\begin{aligned}
\tilde \Xi_2^{(1)}(\lambda,s,\mu) &\defeq \int_{-1}^1 \partial_\lambda \Xi_2 (\lambda+2vs,\mu) \,\dd v,\\
\tilde \Xi_2^{(2)}(\lambda,s,\mu) &\defeq \int_{-1}^1 \partial_\lambda^2 \Xi_2 (\lambda+2vs,\mu) \, (1-|v|) \,\dd v,
\end{aligned}
\end{equation*}
we see that, for $b=1,2$, the function $s \mapsto \tilde \Xi_2^{(b)}(\lambda,s,\mu)$ is even and
\[
\tilde \Xi_2^{(b)}(\lambda,|\mu|,\mu) = \Xi_2^{(b)}(\lambda,\mu).
\]
Take $\alpha=(r,\beta) \in \NN \times \NN^{\dimZ}$; then, by Lemma \ref{lem:propositionderivative}, 
\begin{equation}\label{eq:H4}
\begin{split}
| \partial^{\alpha}_{(\lambda,\mu)} \Xi_2^{(b)} (\lambda,\mu)|
&\lesssim_\alpha \sum_{\beta_{1}+\beta_{2}=\beta} \sup_{u\in[0,1]} \left| \left.\partial^{r}_\lambda \partial_{s}^{|\beta_{1}|}\partial_{\mu}^{\beta_{2}} \tilde \Xi_2^{(b)}(\lambda,s,\mu)\right|_{s=u|\mu|}\right|\\
&\lesssim_\alpha \sup_{\substack{ v\in[-1,1] \\ \gamma \tc |\gamma| = b+|\alpha|}} \left| [\partial_{(\lambda,\mu)}^\gamma \Xi_2](\lambda+2v|\mu|,\mu) \right|.
\end{split}
\end{equation}
Notice now that, for $v \in [-1,1]$, we have $(\lambda+2v|\mu|)+3|\mu|\geq \lambda+|\mu|$. Hence from \eqref{eq:H4} and Lemma \ref{lem:Propositionderivativeestimate} we conclude that, for all $c \in (0,1)$,
\[
| \partial^{\alpha}_{(\lambda,\mu)} \Xi^{(b)}_2 (\lambda,\mu)| \lesssim_{\alpha,c} \begin{cases}
(\lambda+|\mu|)^{-(|\alpha|+b-2)_+} \,\ee^{-2c \sqrt{\lambda+|\mu|}} &\text{if } |\alpha| \neq 2-b,\\
\log(\ee+1/(\lambda+|\mu|)) \,\ee^{-2c \sqrt{\lambda+|\mu|}} &\text{if } |\alpha| = 2-b.
\end{cases}
\]
We thus see that, for $\lambda \geq \kappa|\mu|$,
\begin{multline*}
\left|\lambda^{|\alpha|} \partial_{(\lambda,\mu)}^{\alpha} [(\dimV/4) \Xi_2^{(1)}(\lambda,\mu) + \lambda \Xi_2^{(2)}(\lambda,\mu)] \right| \\
\lesssim_{\alpha,c,\kappa} \begin{cases}
\ee^{-2c\sqrt{\lambda+|\mu|}} &\text{if } |\alpha|=0,\\
\lambda \log(\ee+1/\lambda) \, \ee^{-2c\sqrt{\lambda+|\mu|}} &\text{if } |\alpha|=1,\\
\lambda \, \ee^{-2c\sqrt{\lambda+|\mu|}} &\text{if } |\alpha|\geq 2.
\end{cases}
\end{multline*}
Similarly, from Lemma \ref{lem:Propositionderivativeestimate} we also deduce that
\[
\left| \lambda^{|\alpha|} \partial_{(\lambda,\mu)}^{\alpha} \Xi_2(\lambda,\mu) \right| 
\lesssim_{\alpha,c,\kappa} \begin{cases}
\lambda^{|\alpha|} \ee^{-2c\sqrt{\lambda+|\mu|}} &\text{if } |\alpha|\leq 1,\\
\lambda^2 \log(\ee+1/\lambda) \, \ee^{-2c\sqrt{\lambda+|\mu|}} &\text{if } |\alpha|=2,\\
\lambda^2 \, \ee^{-2c\sqrt{\lambda+|\mu|}} &\text{if } |\alpha|\geq 3.
\end{cases}
\]
In light of \eqref{eq:definitiontildeHs} and \eqref{eq:F0Fbullet}, the estimate \eqref{eq:tH2est} follows.

From \eqref{eq:tH2est} we also deduce that
\[
\left|\nabla_{(\lambda,\mu)}[\lambda^{|\alpha|} \partial_{(\lambda,\mu)}^\alpha F_0(\lambda,\mu)] \right|
\lesssim_{c,\kappa,\alpha} \begin{cases}
\log(\ee+1/\lambda) \, \ee^{-2c\sqrt{\lambda+|\mu|}} &\text{if } |\alpha|\leq 1,\\
\ee^{-2c\sqrt{\lambda+|\mu|}} &\text{if } |\alpha|\geq 2,
\end{cases}
\]
and the claimed H\"older continuity thus follows from Lemma \ref{lem:Lemmanabla} below.
\end{proof}

\begin{lemma}\label{lem:Lemmanabla}
Let $\kappa,C \in \Rpos$, $\varepsilon \in (0,1)$, and $\dimWZX \in \Npos$. Let $\Omega \subseteq \{(\lambda,\mu)\in \Rpos \times \RR^{\dimWZX} \tc \lambda\geq \kappa |\mu|\}$ be convex. Assume that $F\in C^{1}(\Omega)$ satisfies 
\begin{equation}\label{eq:ass_grad_hoelder}
\left|\nabla_{(\lambda,\mu)}F(\lambda,\mu) \right|\leq C\lambda^{\varepsilon-1}  \qquad\forall(\lambda,\mu)\in\Omega.
\end{equation}
Then, $F$ is $\varepsilon$-H\"older continuous, i.e.,
\[
\left|F(\lambda,\mu)-F(\lambda',\mu') \right|\lesssim_{\varepsilon,\kappa}	C \left|(\lambda-\lambda',\mu-\mu') \right|^{\varepsilon} \qquad\forall(\lambda,\mu),(\lambda',\mu')\in\Omega.
\]
\end{lemma}
\begin{proof}
Without loss of generality, we may assume $C=1$. Via the fundamental theorem of calculus, and as $\Omega$ is convex, from \eqref{eq:ass_grad_hoelder} we deduce that 
\begin{equation}\label{eq:ftc_hoelder}
\begin{split}
\left|F(\lambda,\mu)-F(\lambda',\mu') \right|
&=\left| \int_0^1 \nabla_{(\lambda,\mu)} F((1-t)(\lambda',\mu')+t(\lambda,\mu))\cdot(\lambda-\lambda',\mu-\mu') \,\dd t\right| \\
&\leq \left|(\lambda-\lambda',\mu-\mu') \right| \int_0^1 \left[ (1-t)\lambda'+t\lambda\right] ^{\varepsilon-1} \,\dd t.
\end{split}
\end{equation}

Notice now that
\begin{equation}\label{eq:claim_hoelder}
\int_0^1 \left[ (1-t)\lambda'+t\lambda\right]^{\varepsilon-1} \,\dd t
\simeq_{\varepsilon} \max\{\lambda,\lambda'\}^{\varepsilon-1}.
\end{equation}
Indeed, if $\lambda\simeq\lambda'$, then $(1-t)\lambda'+t\lambda\simeq\lambda\simeq\lambda'\simeq\max\{\lambda,\lambda'\}$, and \eqref{eq:claim_hoelder} trivially follows. If instead $\lambda\gg\lambda'$ (the case $\lambda \ll \lambda'$ is treated analogously), the change of variables $u = (1-t)\lambda'+t\lambda$ gives
\[\begin{split}
\int_{0}^{1}\left[ (1-t)\lambda'+t\lambda\right] ^{\varepsilon-1} \,\dd t
&=\frac{1}{\lambda-\lambda'}\int_{\lambda'}^{\lambda}u^{\varepsilon-1} \,\dd u\\
&=\frac{\lambda^{\varepsilon}-(\lambda')^{\varepsilon}}{\varepsilon(\lambda-\lambda')}
\simeq_{\varepsilon}\frac{\lambda^{\varepsilon}}{\lambda}
=\lambda^{\varepsilon-1}
=\max\{\lambda,\lambda'\}^{\varepsilon-1},
\end{split}\]
and \eqref{eq:claim_hoelder} is proved.

From \eqref{eq:ftc_hoelder} and \eqref{eq:claim_hoelder} we see that
\[\begin{split}
\left|F(\lambda,\mu)-F(\lambda',\mu') \right|
&\lesssim_{\varepsilon} \left|(\lambda-\lambda',\mu-\mu') \right|\max\{\lambda,\lambda'\}^{\varepsilon-1} \\
&=\left|(\lambda-\lambda',\mu-\mu') \right|^{\varepsilon} \left( \frac{\left|(\lambda-\lambda',\mu-\mu') \right|}{\max\{\lambda,\lambda'\}}\right) ^{1-\varepsilon} \\
&\lesssim_{\varepsilon,\kappa} \left|(\lambda-\lambda',\mu-\mu') \right|^{\varepsilon},
\end{split}\]
as required; in the last step we used that
\[
\left|(\lambda-\lambda',\mu-\mu') \right|\lesssim |\lambda-\lambda'|+|\mu-\mu'|\leq\lambda+\lambda'+|\mu|+|\mu'|\lesssim_\kappa \max\{\lambda,\lambda'\}.
\]
for all $(\lambda,\mu),(\lambda',\mu') \in \Omega$.
\end{proof}

\subsection{\texorpdfstring{$L^{p}$}{Lp}-boundedness of the Riesz transforms for \texorpdfstring{$p\geq 2$}{p>=2}}

We now have all the ingredients to complete the proof of Theorem \ref{thm:maintheorem}.

\begin{proof}[Proof of Theorem \ref{thm:maintheorem}]
By Theorem \ref{thm:Theoremboundedleq2} we know that the Riesz transforms $\Rz_j$ are of weak type $(1,1)$ and $L^p$-bounded for $p \in (1,2]$. So, by duality, it only remains to prove that the adjoint transforms $\Rz_j^*$ are $L^p$-bounded for $p \in (1,2]$ and, by Propositions \ref{prp:reductioninfinity} and \ref{prp:RzNfactorisation}, we are reduced to showing that the convolution operators with kernels $K_0,K_{\lie{v}},K_{\lie{z}}$ are $L^p$-bounded. 

From the discussion in Section \ref{ss:convKopvalmult}, those convolution operators can be written as $M_0(\opLN,-i\nabla_z),M_{\lie{v}}(\opLN,-i\nabla_z),M_{\lie{z}}(\opLN,-i\nabla_z)$. By Proposition \ref{prp:propositionaboutkjk0}, the operator-valued symbols $M_\ell$ with $\ell \in \{0,\lie{v},\lie{z}\}$ satisfy the smoothness assumptions of the multiplier theorem stated in Corollary \ref{cor:applytheorem}, so by the latter result the corresponding operators $M_\ell(\opLN,-i\nabla_z)$ are indeed $L^p$-bounded, as required.
\end{proof}


\begin{thebibliography}{99}

\bibitem{AL94}
G. Alexopoulos and N. Lohou\'e, Riesz means on Lie groups and Riemannian manifolds of nonnegative curvature,
\textit{Bull. Soc. Math. France} \textbf{122} (1994), 209--223.

\bibitem{ADY96}
J.P. Anker, E. Damek, and C. Yacoub, Spherical analysis on harmonic $AN$ groups,
\textit{Ann. Scuola Norm. Sup. Pisa Cl. Sci. (4)} \textbf{23} (1996), 643--679.

\bibitem{ACDH04}
P. Auscher, T. Coulhon, X.T. Duong, and S. Hoffman, Riesz transform on manifolds and heat kernel regularity,
\textit{Ann. Sci. \'Ecole Norm. Sup. (4)} \textbf{37} (2004), 911--957.

\bibitem{B89}
D. Bakry, The Riesz transforms associated with second order differential operators.
In: \textit{Seminar on Stochastic Processes, 1988 (Gainesville, FL, 1988)}, 
Progr. Probab., 17, Birkh\"auser, Boston, MA, 1989, pp.\ 1--43.

\bibitem{CG84}
M. Christ and D. Geller, Singular integral characterizations of Hardy spaces on homogeneous groups,
\textit{Duke Math. J.} \textbf{51} (1984), 547--598.

\bibitem{CD03}
T. Coulhon and X.T. Duong, Riesz transform and related inequalities on noncompact Riemannian manifolds,
\textit{Comm. Pure Appl. Math.} \textbf{56} (2003), 1728--1751.

\bibitem{CDKR91}
M.G. Cowling, A.H. Dooley, A. Koranyi, and F. Ricci, H-type groups and Iwasawa decompositions,
\textit{Adv. Math.} \textbf{87} (1991), 1--41.

\bibitem{CDKR98}
M.G. Cowling, A.H. Dooley, A. Koranyi, and F. Ricci, An approach to symmetric spaces of rank one via groups of Heisenberg type,
\textit{J. Geom. Anal.} \textbf{8} (1998), 199--237.

\bibitem{CGHM94}
M.G. Cowling, S. Giulini, A. Hulanicki, and G. Mauceri, Spectral multipliers for a distinguished Laplacian on certain groups of exponential growth,
\textit{Studia Math.} \textbf{111} (1994), 103--121.

\bibitem{D87}
E. Damek, Curvature of a semidirect extension of a Heisenberg type nilpotent group,
\textit{Colloq. Math.} \textbf{53} (1987), 249--253.

\bibitem{DR92}
E. Damek and F. Ricci, A class of nonsymmetric harmonic Riemannian spaces,
\textit{Bull. Amer. Math. Soc. (N.S.)} \textbf{27} (1992), 139--142.

\bibitem{DR921}
E. Damek and F. Ricci, Harmonic analysis on solvable extensions of H-type groups,
\textit{J. Geom. Anal.} \textbf{2} (1992), 213--248.

\bibitem{DLMF}
Digital Library of Mathematical Functions, version 1.2.5 (2025-12-15), \url{https://dlmf.nist.gov}.
		
\bibitem{tERS97}
A.F.M. ter Elst, D.W. Robinson, and A. Sikora, Heat kernels and Riesz transforms on nilpotent Lie groups,
\textit{Colloq. Math.} \textbf{74} (1997), 191--218.

\bibitem{EMOT1}
A. Erdelyi, W. Magnus, F. Oberhettinger, and F.G. Tricomi, \textit{Higher transcendental functions. Vol. I}, Robert E. Krieger Publishing Co. Inc., Melbourne, Fla., 1981.

\bibitem{F75}
G.B. Folland, Subelliptic estimates and function spaces on nilpotent Lie groups,
\textit{Ark. Mat.} \textbf{13} (1975), 161--207.

\bibitem{GS96}
G.I. Gaudry and P. Sj\"ogren, Singular integrals on Iwasawa $NA$ groups of rank $1$,
\textit{J. Reine Angew. Math.} \textbf{479} (1996), 39--66.

\bibitem{GS99}
G.I. Gaudry and P. Sj\"ogren, Haar-like expansions and boundedness of a Riesz operator on a solvable Lie group,
\textit{Math. Z.} \textbf{232} (1999), 241--256.

\bibitem{HS03}
W. Hebisch and T. Steger, Multipliers and singular integrals on exponential growth groups,
\textit{Math. Z.} \textbf{245} (2003), 37--61.

\bibitem{H74}
A. Hulanicki, Subalgebra of $L^1(G)$ associated with Laplacian on a Lie group,
\textit{Coll. Math.} \textbf{31} (1974), 259--287.

\bibitem{HvNVW16}
T. Hyt\"onen, J. van Neerven, M. Veraar, and L. Weis, \textit{Analysis in Banach Spaces. Vol. I}, Ergebnisse der Mathematik und ihrer Grenzgebiete, vol. 63, Springer, Cham, 2016.

\bibitem{HvNVW17}
T. Hyt\"onen, J. van Neerven, M. Veraar, and L. Weis, \textit{Analysis in Banach Spaces. Vol. II}, Ergebnisse der Mathematik und ihrer Grenzgebiete, vol. 67, Springer, Cham, 2017.

\bibitem{K80}
A. Kaplan, Fundamental solutions for a class of hypoelliptic PDE generated by composition of quadratic forms,
\textit{Trans. Amer. Math. Soc.} \textbf{258} (1980), 147--153.

\bibitem{KW16}
C. Kriegler and L. Weis, Paley--Littlewood decomposition for sectorial operators and interpolation spaces,
\textit{Math. Nachr.} \textbf{289} (2016), 1488--1525.

\bibitem{KW18}
C. Kriegler and L. Weis, Spectral multiplier theorems via $H^{\infty}$ calculus and R-bounds,
\textit{Math. Z.} \textbf{289} (2018), 405--444.

\bibitem{LMSTV23}
M. Levi, A. Martini, F. Santagati, A. Tabacco, and M. Vallarino, Riesz transform for a flow Laplacian on homogeneous trees,
\textit{J. Fourier Anal. Appl.} \textbf{29} (2023), art.no. 15.

\bibitem{L06}
X.-D. Li, Riesz transforms for symmetric diffusion operators on complete Riemannian manifolds,
\textit{Rev. Mat. Iberoam.} \textbf{22} (2006), 591--648.

\bibitem{LM98}
N. Lohou\'e and S. Mustapha, Sur les transform\'ees de Riesz sur les groupes de Lie moyennables et sur certains espaces homog\`enes,
\textit{Canad. J. Math.} \textbf{50} (1998), 1090--1104.

\bibitem{LV85}
N. Lohou\'e and N.T. Varopoulos, Remarques sur les transform\'ees de Riesz sur les groupes de Lie nilpotents,
\textit{C. R. Acad. Sci. Paris S\'er. I Math.} \textbf{301} (1985), 559--560.

\bibitem{M10}
A. Martini, \textit{Algebras of differential operators on Lie groups and spectral multipliers}, PhD Thesis, Scuola Normale Superiore, 2010, \arXivlink{1007.1119}.

\bibitem{M11}
A. Martini, Spectral theory for commutative algebras of differential operators on Lie groups,
\textit{J. Funct. Anal.} \textbf{260} (2011), 2767--2814.

\bibitem{M12}
A. Martini, Analysis of joint spectral multipliers on Lie groups of polynomial growth,
\textit{Ann. Inst. Fourier (Grenoble)} \textbf{62} (2012), 1215--1263.

\bibitem{M17}
A. Martini, Joint functional calculi and a sharp multiplier theorem for the Kohn Laplacian on spheres,
\textit{Math. Z.} \textbf{286} (2017), 1539--1574.

\bibitem{M23}
A. Martini, Riesz transforms on $ax+b$ groups, \textit{J. Geom. Anal.} \textbf{33} (2023), art.no. 222.

\bibitem{MM16}
A. Martini and D. M\"uller, Spectral multipliers on 2-step groups: topological versus homogeneous dimension,
\textit{Geom. Funct. Anal.} \textbf{26} (2016), 680--702.	

\bibitem{MP24}
A. Martini and P. Plewa, Singular integrals on $ax+ b$ hypergroups and an operator-valued spectral multiplier theorem, 
\textit{Ann. Sc. Norm. Super. Pisa Cl. Sci. (5)} (to appear), \DOIlink{10.2422/2036-2145.202409\_036}.

\bibitem{MP24bis}
A. Martini and P. Plewa, $L^p$-boundedness of Riesz transforms on solvable extensions of Carnot groups, preprint (2024), \arXivlink{2409.13233}.

\bibitem{MSTV25}
A. Martini, F. Santagati, A. Tabacco, and M. Vallarino, Riesz transform and spectral multipliers for the flow Laplacian on nonhomogeneous trees,
\textit{Rev. Mat. Iberoam.} \textbf{41} (2025), 2221--2282.

\bibitem{MV21}
A. Martini and M. Vallarino, Riesz transforms on solvable extensions of stratified groups,
\textit{Studia. Math.} \textbf{259} (2021), 175--200.

\bibitem{MRS96}
D. M\"uller, F. Ricci, and E.M. Stein, Marcinkiewicz multipliers and multi-parameter structure on Heisenberg (-type) groups. II,
\textit{Math. Z.} \textbf{221} (1996), 267--291.

\bibitem{MV10}
D. M\"uller and M. Vallarino, Wave equation and multiplier estimates on Damek-Ricci spaces,
\textit{J. Four. Anal. Appl.} \textbf{16} (2010), 204--232.

\bibitem{RT20}
L. Roncal and S. Thangavelu, An extension problem and trace Hardy inequality for the sub-Laplacian on H-type groups,
\textit{Int. Math. Res. Not. IMRN} \textbf{2020} (2020), 4238--4294.

\bibitem{Sj99}
P. Sj\"ogren, An estimate for a first-order Riesz operator on the affine group,
\textit{Trans. Amer. Math. Soc.} \textbf{351} (1999), 3301--3314.

\bibitem{SV08}
P. Sj\"ogren and M. Vallarino, Boundedness from $H^1$ to $L^1$ of Riesz transforms on a Lie group of exponential growth,
\textit{Ann. Inst. Fourier (Grenoble)} \textbf{58} (2008), 1117--1151.

\bibitem{S95}
E.M. Stein, \textit{Harmonic Analysis: Real-Variable Methods, Orthogonality, and Oscillatory Integrals},
Princeton Mathematical Series, 43, Princeton University Press, Princeton, NJ, 1993.

\bibitem{SW07}
\v{Z}. \v{S}trkalj and L. Weis, On operator-valued Fourier multiplier theorems,
\textit{Trans. Amer. Math. Soc.} \textbf{359} (2007), 3529--3547.

\bibitem{V07}
M. Vallarino, Spectral multipliers on Damek--Ricci spaces,
\textit{J. Lie Theory} \textbf{17} (2007), 163--189.

\bibitem{VSC92}
N.T. Varopoulos, L. Saloff-Coste, and T. Coulhon, \textit{Analysis and Geometry on Groups},
Cambridge University Press, Cambridge, 1992.

\bibitem{W01}
L. Weis, Operator-valued Fourier multiplier theorems and maximal $L_p$-regularity,
\textit{Math. Ann.} \textbf{319} (2001), 735--758.

\end{thebibliography}
\end{document}